\documentclass[a4paper,pdflatex]{amsart}

\usepackage{etex}
\usepackage[utf8]{inputenc}
\usepackage[T1]{fontenc}
\usepackage{graphicx}
\usepackage{tikz}
\usepackage{pgfplots}
\usepackage{mathtools, eqparbox}
\usepackage{etoolbox}
\usepackage{adjustbox}
\usepackage{multirow}
\usepackage{import}
\usepackage{mathabx}
\usepackage{nextpage}

\makeatletter
\providerobustcmd*{\bigcupdot}{%
  \mathop{%
    \mathpalette\bigop@dot\bigcup
  }%
}
\newrobustcmd*{\bigop@dot}[2]{%
  \setbox0=\hbox{$\m@th#1#2$}%
  \vbox{%
    \lineskiplimit=\maxdimen
    \lineskip=-0.7\dimexpr\ht0+\dp0\relax
    \ialign{%
      \hfil##\hfil\cr
      $\m@th\cdot$\cr
      \box0\cr
    }%
  }%
}
\makeatother

\newcommand{\mystackrel}[3][T]{\stackrel{\eqmakebox[#1]{\scriptsize#2}}{#3}}
\usetikzlibrary{patterns}
\usetikzlibrary{arrows,decorations.markings}
\tikzset{
  schraffiert/.style={pattern=horizontal lines,pattern color=#1},
  schraffiert/.default=black
}
\usepackage[many]{tcolorbox}

\usepackage{float}
\usepackage{amsmath}
\makeatletter
\newcommand\numbereq{%
  \ifmeasuring@\else
    \refstepcounter{equation}%
  \fi
  \tag{\theequation}%
}
\makeatother
\allowdisplaybreaks
\usepackage{ltablex}
\usepackage{booktabs}
\usepackage{amsfonts}
\usepackage{amsthm}
\usepackage{amssymb}
\usepackage{setspace}	
\usepackage[format=hang]{caption}
\usepackage{thmtools}
\usepackage{thm-restate}
\usepackage{mathtools}
\usepackage{verbatim}
\usepackage[colorlinks,citecolor=blue]{hyperref}
\usepackage[capitalize]{cleveref}
\usepackage[all, cmtip]{xy}
\usepackage{chngcntr}
\usetikzlibrary{cd}
\usepackage{titlesec}
\titleformat{\section}
  {\normalfont\normalsize\bfseries}{\thesection.}{1em}{}
\titleformat{\subsection}
  {\normalfont\normalsize\bfseries}{\thesubsection.}{1em}{}
\titleformat{\subsubsection}
  {\normalfont\normalsize\bfseries}{\thesubsubsection.}{1em}{}
\titleformat{\listfigurename}
  {\normalfont\normalsize\bfseries}{Figures}{1em}{}
\counterwithin{figure}{section}

\usepackage[shortlabels]{enumitem}

\DeclareMathOperator{\Aut}{Aut}
\DeclareMathOperator{\Conf}{Conf}

\DeclareMathOperator{\MapOp}{Map}
\DeclareMathOperator{\PMapOp}{PMap}

\DeclareMathOperator{\HomeoOp}{Homeo}
\DeclareMathOperator{\PHomeoOp}{PHomeo}
\DeclareMathOperator{\Z}{Z}

\DeclareMathOperator{\Sym}{S}
\DeclareMathOperator{\PZ}{PZ}

\DeclareMathOperator{\id}{\mathrm{id}}

\title{Mapping class groups for 2-orbifolds} 
\author{Jonas Flechsig}
\date{May 7, 2023}
\makeatletter    
\def\l@section{\@tocline{1}{0,2pt}{2pc}{8mm}{\ \ }} 
\makeatletter    
\def\l@subsection{\@tocline{1}{0,2pt}{2pc}{8mm}{\ \ }} 

\renewcommand{\maketitle}{
    \begin{center}

        \phantom{.}  

        {\LARGE \bf \@title\par}
        \vspace{0.5cm}

    \end{center}
}\makeatother
\begin{document}
\newtheorem*{theorem*}{Theorem}
\newtheorem{theorem}{Theorem}[section]
\newtheorem{corollary}[theorem]{Corollary}
\newtheorem{lemma}[theorem]{Lemma}
\newtheorem{fact}[theorem]{Fact}
\newtheorem*{fact*}{Fact}
\newtheorem{proposition}[theorem]{Proposition}
\newtheorem{thmletter}{Theorem}
\newtheorem{observation}[theorem]{Observation}
\newtheorem{notation}[theorem]{Notation}
\renewcommand*{\thethmletter}{\Alph{thmletter}}
\theoremstyle{definition}
\newtheorem{example}[theorem]{Example}
\newtheorem{question}[theorem]{Question}
\newtheorem{definition}[theorem]{Definition}
\newtheorem{construction}[theorem]{Construction}
\theoremstyle{remark}
\newtheorem{remark}[theorem]{Remark}
\newtheorem{conjecture}[theorem]{Conjecture}
\newtheorem{case}{Case}
\newtheorem{claim}{Claim}
\newtheorem*{claim*}{Claim}
\newtheorem{step}{Step}
\counterwithin{case}{theorem}
\renewcommand{\thecase}{\arabic{case}}
\counterwithin{claim}{theorem}
\renewcommand{\theclaim}{\arabic{claim}}
\counterwithin{step}{theorem}
\renewcommand{\thestep}{\arabic{step}}

\newenvironment{intermediate}[1][\unskip]{%
\vspace*{5pt}
\par
\noindent
\textit{#1.}}
{}
\vspace*{5pt}

\newcommand{\doubletable}[1]{\begin{tabular}[l]{@{}l@{}}#1\end{tabular}}
\newcommand{\set}[1][ ]{\ensuremath{ \lbrace #1 \rbrace}}
\newcommand{\bsl}{\ensuremath{\setminus}}
\newcommand{\grep}[2]{\ensuremath{\left\langle #1 | #2\right\rangle}}
\renewcommand{\ll}{\left\langle}
\newcommand{\rr}{\right\rangle}
\newcommand{\cpxbrmn}{B(\TheOrder,\TheOrder,\TheStrand)}
\newcommand{\cpxbr}[2]{B(#1,#1,#2)}
\newcommand{\Map}[2]{\MapOp_{#1}\left({#2}\right)}
\newcommand{\PMap}[2]{\PMapOp_{#1}\left({#2}\right)}
\newcommand{\MapOrb}[2]{\MapOp_{#1}^{orb}\left({#2}\right)}
\newcommand{\PMapOrb}[2]{\PMapOp_{#1}^{orb}\left({#2}\right)}
\newcommand{\MapId}[2]{\MapOp_{#1}^{\id}\left({#2}\right)}
\newcommand{\PMapId}[2]{\PMapOp_{#1}^{\id}\left({#2}\right)}
\newcommand{\MapIdOrb}[2]{\MapOp_{#1}^{\id,orb}\left({#2}\right)}
\newcommand{\PMapIdOrb}[2]{\PMapOp_{#1}^{\id,orb}\left({#2}\right)}
\newcommand{\HomeoId}[2]{\HomeoOp_{#1}^{\id}({#2})}
\newcommand{\PHomeoId}[2]{\PHomeoOp_{#1}^{\id}({#2})}
\newcommand{\Homeo}[2]{\HomeoOp_{#1}({#2})}
\newcommand{\PHomeo}[2]{\PHomeoOp_{#1}({#2})}
\newcommand{\HomeoIdOrb}[2]{\HomeoOp_{#1}^{\id,orb}({#2})}
\newcommand{\PHomeoIdOrb}[2]{\PHomeoOp_{#1}^{\id,orb}({#2})}
\newcommand{\HomeoOrb}[2]{\HomeoOp_{#1}^{orb}({#2})}
\newcommand{\PHomeoOrb}[2]{\PHomeoOp_{#1}^{orb}({#2})}
\newcommand{\PHomeoOrbt}[3]{\PHomeoOp_{#1}^{orb}({#2})}

\newcommand{\omicron}{o}
\newcommand{\cp}{c}
\newcommand{\pct}{p}
\newcommand{\TheCone}{N}
\newcommand{\ThePct}{L}
\newcommand{\Pc}{\theta}
\newcommand{\Pct}{\iota}
\newcommand{\NPct}{\lambda}
\newcommand{\TheOrder}{m}
\newcommand{\Ord}{t}
\newcommand{\TheStrand}{n}
\newcommand{\Order}{l}
\newcommand{\Strr}{h}
\newcommand{\Str}{i}
\newcommand{\Strand}{j}
\newcommand{\NStrand}{k}
\newcommand{\NNStrand}{l}
\newcommand{\pStrand}{p}
\newcommand{\qStrand}{q}
\newcommand{\sStrand}{s}
\newcommand{\tStrand}{t}
\newcommand{\Subdivision}{i}
\newcommand{\TheSubdivision}{p}
\newcommand{\Dim}{i}
\newcommand{\NDim}{j}
\newcommand{\TheDim}{k}
\newcommand{\Subdiv}{i}
\newcommand{\TheSubdiv}{p}
\newcommand{\NSubdiv}{j}
\newcommand{\TheNSubdiv}{q}
\newcommand{\NNSubdiv}{k}
\newcommand{\TheNNSubdiv}{r}
\newcommand{\NNNSubdiv}{l}
\newcommand{\TheFrac}{\frac{2\pi}{\TheOrder}}
\newcommand{\HalfFrac}{\frac{\pi}{\TheOrder}}
\newcommand{\neigh}{\lambda}
\newcommand{\htwC}{h_1^\twsC}
\newcommand{\htw}{h_{\TheStrand-1}^{\twsC'}}
\newcommand{\col}{blue}
\newcommand{\colo}{olive}
\newcommand{\short}{red}
\newcommand{\mult}{orange}
\newcommand{\red}{red}
\newcommand{\green}{olive}
\newcommand{\gre}{green}
\newcommand{\blue}{blue}
\newcommand{\SGpct}{\Sigma_\freeprod(\ThePct,\TheCone)}
\newcommand{\SG}{\Sigma_\freeprod(\ThePct)}
\newcommand{\Spct}{\Sigma(\ThePct,\TheCone)}
\renewcommand{\S}{\Sigma(\ThePct)}
\newcommand{\Sk}[1]{\Sigma(#1)}
\newcommand{\orbtwo}{\Sigma_{\freeprodtwo}}
\newcommand{\G}{G}
\newcommand{\g}{g}
\newcommand{\freeprod}{\Gamma}
\newcommand{\freeprodtwo}{\cycm\ast\cyc{\TheOrder'}}
\newcommand{\freeprodex}{\Gamma_{\TheOrder_\nu}^\TheCone}
\newcommand{\freeprodexk}[1]{\Gamma_{\TheOrder_\nu}^{#1}}
\newcommand{\freegrp}[1]{F_{#1}}
\newcommand{\free}[1]{F^{(#1)}}
\newcommand{\cycm}{\ZZ_\TheOrder}
\newcommand{\cyc}[1]{\ZZ_{#1}}
\newcommand{\kernel}{K}
\newcommand{\inter}[1]{{#1}^{\circ}}
\newcommand{\interext}[1]{{#1}^{\circ,\text{ext}}}

\newcommand{\Twist}{A}
\newcommand{\TwistP}{B}
\newcommand{\TwistC}{C}
\newcommand{\TwP}[1]{T_{#1}}
\newcommand{\TwC}[1]{U_{#1}}
\newcommand{\twist}{a}
\newcommand{\twistP}{b}
\newcommand{\twistC}{c}
\newcommand{\twP}[1]{t_{#1}}
\newcommand{\twC}[1]{u_{#1}}
\newcommand{\twsP}{t}
\newcommand{\twsC}{u}
\newcommand{\Twistn}[1]{X_{#1}}
\newcommand{\TwistnP}[1]{Y_{#1}}
\newcommand{\TwistnC}[1]{Z_{#1}}
\newcommand{\twistn}[1]{x_{#1}}
\newcommand{\twistnP}[1]{y_{#1}}
\newcommand{\twistnC}[1]{z_{#1}}
\newcommand{\twistnsC}{z}
\newcommand{\TwsP}{T}
\newcommand{\TwsC}{U}
\newcommand{\FD}{F}

\newcommand{\MA}{\ensuremath{\mathcal{MA}_\TheStrand}}
\newcommand{\pMA}{\ensuremath{\mathcal{MA}_\TheStrand(F_\freeprod)}}

\newcommand{\HA}{\ensuremath{\mathcal{HA}_\TheStrand}}

\newcommand{\MAo}{\ensuremath{\mathcal{MA}_{\TheStrand}(\Sigma_\freeprod)}}
\newcommand{\pMAo}{\ensuremath{\mathcal{MA}_{\TheStrand}(\Sigma_\freeprod^\ast)}}

\newcommand{\MAoZ}{\ensuremath{\mathcal{MA}_{\TheStrand}(D_{\ZZ_\TheOrder})}}
\newcommand{\MAoD}{\ensuremath{\mathcal{MA}_{\TheStrand}(\CC_{D_\TheOrder})}}
\newcommand{\tMAo}{\ensuremath{\tilde{\mathcal{MA}}_{\TheStrand}(\Sigma_\freeprod)}}

\newcommand{\HAo}{\ensuremath{\mathcal{HA}_{\TheStrand}(\Sigma_\freeprod)}}

\newcommand{\bpHAk}[1]{\ensuremath{\mathcal{HA}_{\TheStrand,\TheStrand',#1}}}
\newcommand{\bpHAo}{\ensuremath{\mathcal{HA}_{\TheStrand,\TheStrand'}(\Sigma_\freeprod)}}
\newcommand{\bpHAok}[1]{\ensuremath{\mathcal{HA}_{\TheStrand,\TheStrand',#1}(\Sigma_\freeprod)}}
\newcommand{\bpMAo}{\ensuremath{\mathcal{MA}_{\TheStrand,\TheStrand'}(\Sigma_\freeprod)}}
\newcommand{\bpMAok}[1]{\ensuremath{\mathcal{MA}_{\TheStrand,\TheStrand',#1}(\Sigma_\freeprod)}}
\newcommand{\bpMAoF}{\ensuremath{\mathcal{MA}_{\TheStrand,\TheStrand',\ThePct}^{F_\freeprod}(\Sigma_\freeprod)}}

\newcommand{\pbpMAo}{\ensuremath{\mathcal{MA}_{\TheStrand,\TheStrand'}(\Sigma_\freeprod^\ast)}}
\newcommand{\tbpMAo}{\ensuremath{\mathcal{MA}_{\TheStrand,\TheStrand'}(\Sigma_\freeprod)}}
\newcommand{\bpMAos}{\ensuremath{\mathcal{MA}_{\TheStrand,\TheStrand'}^{sim}(\Sigma_\freeprod)}}
\newcommand{\bpMA}{\ensuremath{\mathcal{MA}_{\TheStrand,\TheStrand'}}}
\newcommand{\bpMAk}[2]{\ensuremath{\mathcal{MA}_{\TheStrand,\TheStrand',{#1}}}(#2)}
\newcommand{\pbpMA}{\ensuremath{\mathcal{MA}_{\TheStrand,\TheStrand'}(F_\freeprod^\ast)}}
\newcommand{\tbpMA}{\ensuremath{\mathcal{MA}_{\TheStrand,\TheStrand'}(\Sigma_\freeprod)}}

\newcommand{\bpM}{\ensuremath{\mathcal{M}_{\TheStrand,\TheStrand',k}}}
\newcommand{\seg}{s}
\newcommand{\st}{\mathrm{st}}
\newcommand{\map}{\rho}
\newcommand{\Rep}{T}

\newcommand{\GL}[2][\TheRank]{\ensuremath{\operatorname{GL_{#1}}(#2)}}
\newcommand{\hmu}[2]{h_{#2}^{\tau_{#1}}}
\newcommand{\Stab}{\operatorname{Stab}}
\newcommand{\CC}{\mathbb{C}}
\newcommand{\RR}{\mathbb{R}}
\newcommand{\ZZ}{\mathbb{Z}}
\newcommand{\NN}{\mathbb{N}}
\newcommand{\PP}{\mathbb{P}}
\newcommand{\HH}{\mathbb{H}}
\newcommand{\SSS}{\mathbb{S}}
\newcommand{\iotaPMap}{\iota_{\PMap_\TheStrand}}
\newcommand{\piPMap}{\pi_{\PMap_\TheStrand}}
\newcommand{\iotaMapast}{\iota_{\PMap_\TheStrand^\ast}}
\newcommand{\piMapast}{\pi_{\PMap_\TheStrand^\ast}}
\newcommand{\iotaPZ}{\iota_{\PZ_\TheStrand}}
\newcommand{\piPZ}{\pi_{\PZ_\TheStrand}}
\newcommand{\sPZ}{\mathrm{s}_{\PZ_\TheStrand}}
\newcommand{\iotaPZast}{\iota_{\PZ_\TheStrand^\ast}}
\newcommand{\piPZast}{\pi_{\PZ_\TheStrand^\ast}}
\newcommand{\iotaPZpct}{\iota_{\PZ_\TheStrand^\ast}}
\newcommand{\piPZpct}{\pi_{\PZ_\TheStrand^\ast}}
\newcommand{\Push}{\textup{Push}}
\newcommand{\Forget}{\textup{Forget}}
\newcommand{\PushPMap}{\textup{Push}_{\textup{PMap}_\TheStrand}}
\newcommand{\ForgetPMap}{\textup{Forget}_{\textup{PMap}_\TheStrand}}
\newcommand{\piMap}{\pi_{\textup{Map}}}
\newcommand{\varphiMap}{\varphi_{\textup{Map}}}
\newcommand{\new}[1]{\color{black} #1 \color{black}}
\newcommand{\enew}[1]{\color{black} #1 \color{black}}

\author{J. Flechsig}
\address{Jonas Flechsig: Fakult\"at f\"ur Mathematik, Universit\"at Bielefeld, D-33501 Bielefeld, Germany}

\maketitle
\begin{center}
Jonas Flechsig
\\[5pt]
May 7, 2023
\\[10pt]
\textbf{Abstract} 
\end{center}

We define orbifold mapping class groups (with marked points) and study them using their action on certain orbifold analogs of arcs and simple closed curves. 
Moreover, we establish a Birman exact sequence for suitable subgroups of orbifold mapping class groups. \new{The short exact sequence} allows us to deduce finite presentations of these groups. This is the basis for a similar discussion of orbifold braid groups in \cite{Flechsig2023braid}.

\section{Introduction}

This article is motivated by the study of orbifold braid groups in \cite{Flechsig2023braid}. Orbifold braid groups are analogs of Artin braid groups or, more generally, surface braid groups. Instead of considering braids moving inside a disk or a surface, orbifold braids move inside a $2$-dimensional orbifold. 
Orbifold braid groups attracted interest since some of them contain spherical and affine Artin groups of type $D_\TheStrand,\tilde{B}_\TheStrand$ and $\tilde{D}_\TheStrand$ as finite index subgroups by work of Allcock \cite{Allcock2002}. For these Artin groups, the orbifold braid groups provide us with braid pictures. 
Roushon published several articles on the structure of orbifold braid groups \cite{Roushon2021,Roushon2022b,Roushon2023,Roushon2022a} and the contained Artin groups \cite{Roushon2021a}. Further, Crisp--Paris~\cite{CrispParis2005} studied the outer automorphism group of the orbifold braid group. 

\new{The studied orbifold braid groups are related to orbifold mapping class groups that are associated to the following orbifolds.} Let $\Sigma_\freeprod$ be the orbifold that is defined using the following data: The group~$\freeprod$ is a free product of finitely many finite cyclic groups. As such, $\freeprod$ acts on a planar, contractible surface $\Sigma$ with boundary, obtained by thickening the Bass--Serre tree (see Example \ref{ex:good_orb_free_prod} for details). If we add $\ThePct$ punctures, we obtain a similar orbifold as studied by Roushon in \cite{Roushon2023}, which we denote by $\Sigma_\freeprod(\ThePct)$. In contrast to his paper, we consider orbifolds with non-empty boundary (which does not affect the structure of the orbifold braid groups). The only singular points in the orbifold $\Sigma_\freeprod(\ThePct)$ are cone points that correspond to the finite cyclic factors of the free product $\freeprod$. 

The associated orbifold mapping class group with respect to $\TheStrand$ marked points is denoted by $\MapOrb{\TheStrand}{\Sigma_\freeprod(\ThePct)}$ for the punctured orbifold $\Sigma_\freeprod(\ThePct)$. A mapping class of $\Sigma_\freeprod(\ThePct)$ is represented by a $\freeprod$-equivariant homeomorphism of $\Sigma(\ThePct)$ that fixes cone points and the boundary~$\partial\Sigma(\ThePct)$. Such a homeomorphism respects the $\TheStrand$ marked points if it preserves the $\freeprod$-orbit of the $\TheStrand$ marked points as a set. The equivalence relation is induced by $\freeprod$-equivariant ambient isotopies fixing cone points, marked points and the boundary. 

\enew{
Mapping class groups of surfaces are studied by their action on arcs and simple closed curves. A basic tool for that is the bigon criterion, see \cite[Proposition 1.7]{FarbMargalit2011}. In particular, the bigon criterion implies that homotopic arcs and homotopic simple closed curves are ambient isotopic. For the study of $\MapOrb{}{\Sigma_\freeprod}$, we introduce orbifold analogs of arcs and simple closed curves, called \textit{$\freeprod$-arcs} and \textit{simple closed $\freeprod$-curves} (see Definitions \ref{def:freeprod-arcs} and \ref{def:freeprod-scc}). Moreover, we establish a bigon criterion for these analogs (see Propositions \ref{prop:bigon_crit_orb} and \ref{prop:bigon_crit_orb_scc}). As in the classical case, this allows us to deduce that homotopic $\freeprod$-arcs and homotopic simple closed $\freeprod$-curves are ambient isotopic (see Proposition \ref{prop:homo_ind_amb_iso}). 
}

Moreover, orbifold mapping class groups admit a homomorphism \[
\Forget_\TheStrand^{orb}:\MapOrb{\TheStrand}{\Sigma_\freeprod(\ThePct)}\rightarrow\MapOrb{}{\Sigma_\freeprod(\ThePct)} 
\]
by forgetting the marked points. Let $\MapIdOrb{\TheStrand}{\Sigma_\freeprod(\ThePct)}$ be the kernel of $\Forget_\TheStrand^{orb}$. 

Since the cone points are fixed, we may restrict mapping classes 
to the subsurface $\Sigma(\ThePct,\TheCone)$ that is also punctured at the cone points. Using that the quotient $\Sigma(\ThePct,\TheCone)/\freeprod$ is a disk $D(\ThePct,\TheCone)$ with $\ThePct+\TheCone$ punctures, this allows us to construct an isomorphism from $\MapIdOrb{\TheStrand}{\Sigma_\freeprod(\ThePct)}$ to a similar subgroup $\MapId{\TheStrand}{D(\ThePct,\TheCone)}$ of the mapping class group $\Map{\TheStrand}{D(\ThePct,\TheCone)}$ (see Proposition \ref{prop:iso_Map_orb_disk} for details). Moreover, recall that the Birman exact sequence yields the following for $\Map{\TheStrand}{D(\ThePct,\TheCone)}$: 
\\
\begin{adjustbox}{center}
\begin{tikzcd}[column sep=35pt]
1\rightarrow\pi_1\left(\Conf_\TheStrand\left(D(\ThePct,\TheCone)\right)\right) \arrow[r,"\Push_\TheStrand"] & \Map{\TheStrand}{D(\ThePct,\TheCone)} \arrow[r,"\Forget_\TheStrand"] & \Map{}{D(\ThePct,\TheCone)}\rightarrow1, 
\end{tikzcd}
\end{adjustbox}
see, for instance, \cite[Theorem 9.1]{FarbMargalit2011}. Based on Proposition \ref{prop:iso_Map_orb_disk}, this allows us to deduce a similar short exact sequence for the groups $\MapIdOrb{\TheStrand}{\Sigma_\freeprod(\ThePct)}$: 
\begin{thmletter}[Birman exact sequence for orbifold mapping class groups]
\label{thm-intro:Birman_exact_seq}
The following diagram is a short exact sequence: 
\\
\begin{adjustbox}{center}
$1\rightarrow\pi_1\left(\Conf_\TheStrand\left(D(\ThePct,\TheCone)\right)\right)\xrightarrow{\Push_\TheStrand^{orb}}\MapIdOrb{\TheStrand}{\Sigma_\freeprod(\ThePct)}\xrightarrow{\Forget_\TheStrand^{orb}}\underbrace{\MapIdOrb{}{\Sigma_\freeprod(\ThePct)}}_{=1}\rightarrow1, $
\end{adjustbox}
see Theorem \textup{\ref{thm:Birman_es_orb}} for details. 
\end{thmletter}

Let $\PMapIdOrb{\TheStrand}{\Sigma_\freeprod(\ThePct)}$ be the subgroup of mapping classes in $\MapIdOrb{\TheStrand}{\Sigma_\freeprod(\ThePct)}$ that induce the trivial permutation on the orbits of marked points. For these subgroups, we obtain: 
\begin{thmletter}
\label{thm-intro:Birman_exact_seq_pure}
The following diagram is a short exact sequence that splits: 
\[
1\rightarrow\freegrp{\TheStrand-1+\ThePct+\TheCone}\xrightarrow{\PushPMap^{orb}}\PMapIdOrb{\TheStrand}{\Sigma_\freeprod(\ThePct)}\xrightarrow{\ForgetPMap^{orb}}\PMapIdOrb{\TheStrand-1}{\Sigma_\freeprod(\ThePct)}\rightarrow1, 
\]
see Corollary \textup{\ref{cor:pure_orb_mcg_ses}} for details. 
\end{thmletter}
In particular, we can deduce finite presentations for 
$\PMapIdOrb{\TheStrand}{\Sigma_\freeprod(\ThePct)}$ and $\MapIdOrb{\TheStrand}{\Sigma_\freeprod(\ThePct)}$ from Theorem \ref{thm-intro:Birman_exact_seq_pure} (see Corollary \ref{cor:pres_PMap_free_prod} and Proposition \ref{prop:pres_map_kcp}). 

\enew{As proven in \cite[Theorem A]{Flechsig2023braid}, certain orbifold braid groups are quotients of the groups $\MapIdOrb{\TheStrand}{\Sigma_\freeprod(\ThePct)}$. Together with the presentations of $\MapIdOrb{\TheStrand}{\Sigma_\freeprod(\ThePct)}$ and its pure subgroup $\PMapIdOrb{\TheStrand}{\Sigma_\freeprod(\ThePct)}$, we obtain presentations of the related orbifold braid groups $\Z_\TheStrand(\Sigma_\freeprod(\ThePct))$ and $\PZ_\TheStrand(\Sigma_\freeprod(\ThePct))$, see Theorem B and Corollary~5.6 in \cite{Flechsig2023braid}. These are essential to deduce structural results about orbifold braid groups. In particular, we obtain a similar Birman exact sequence for pure orbifold braid groups, see \cite[Theorem C]{Flechsig2023braid}. In this case, we surprisingly observe that the analog of the point-pushing map $\PushPMap^{orb}$ is not injective.} 

\section*{Acknowledgments} 

I would like to thank my adviser Kai-Uwe Bux for his support and many helpful discussions. Many thanks are also due to José Pedro Quintanilha and Xiaolei Wu for their helpful advice at different points of this project. Moreover, I am grateful to Elisa Hartmann and José Pedro Quintanilha for their comments on an earlier version of this text. 

The author was funded by the German
Research Foundation (Deutsche For-schungsgemeinschaft, DFG) – 426561549. Further, the author was partially supported by Bielefelder Nachwuchsfonds. 

\newpage

\section{Orbifolds and paths}
\label{sec:intro_orb}

\setlist[enumerate,1]{label=\textup{(\arabic*)}, ref=\thetheorem(\arabic*)}
\setlist[enumerate,2]{label=\textup{\alph*)}, ref=\alph*)}

In this article we only consider orbifolds that are given as the quotient of a manifold (typically a surface) by a proper group action. Recall that an action 
\[
\phi:\G\rightarrow\Homeo{}{M},\g\mapsto\phi_\g 
\]
on a manifold $M$ is \textit{proper} if for each compact set $K\subseteq M$ the set 
\[
\{\g\in\G\mid\phi_\g(K)\cap K\neq\emptyset\} 
\]
is compact. Since we endow $\G$ with the discrete topology, i.e.\ the above set is finite. Orbifolds that appear as proper quotients of manifolds are called \textit{developable} in the terminology of Bridson--Haefliger \cite{BridsonHaefliger2011} and \textit{good} in the terminology of Thurston \cite{Thurston1979}. 

\new{Above and in the following,} all manifolds are orientable and all homeomorphisms are orientation preserving. 

\begin{definition}[{Orbifolds}, {\cite[Chapter III.G, 1.3]{BridsonHaefliger2011}}]
\label{def:good_orb}
Let $M$ be a 
manifold, possibly with boundary, 
and $\G$ a group with a monomorphism 
\[
\phi:\G\rightarrow\Homeo{}{M} 
\]
such that $\G$ acts properly on $M$. Under these conditions the $3$-tuple $(M,\G,\varphi)$ is called an \textit{orbifold}, which we denote by $M_\G$. If $\Stab_\G(\cp)\neq\{1\}$ for a point $\cp\in M$, the point $\cp$ is called a \textit{singular point} of $M_\G$. If $\Stab_\G(\cp)$ for a point $\cp\in M$ is cyclic of finite order $\TheOrder$, the point $\cp$ is called a \textit{cone point} of $M_\G$ of order~$\TheOrder$. 
\end{definition}
 
A first example of an orbifold is the following: 

\begin{example}
\label{ex:good_orb_D_cyc_m}
Let $\cycm$ be a cyclic group of order $\TheOrder$. The group $\cycm$ acts on a disk~$D$ by rotations around its center. The action is via isometries and the acting group is finite, i.e.\ the action is proper. Consequently, $D_{\cycm}$ is an orbifold with exactly one singular point in the center of $D$, which is a cone point. 
\end{example}

Example \ref{ex:good_orb_D_cyc_m} motivates a more general construction for a free product of finitely many finite cyclic groups which we describe briefly in the following. For further details, we refer to the author's PhD thesis \cite[Section 2.1]{Flechsig2023}. We will consider this generalization of the orbifold $D_{\cycm}$ throughout the article.  

\begin{example}
\label{ex:good_orb_free_prod}
Let $\freeprod$ be a free product of finite cyclic groups $\cyc{\TheOrder_1},...,\cyc{\TheOrder_\TheCone}$. The group $\freeprod$ is the fundamental group of the graph of groups with trivial edge groups 
\vspace*{3mm}
\begin{center}
\begingroup%
  \makeatletter%
  \providecommand\color[2][]{%
    \errmessage{(Inkscape) Color is used for the text in Inkscape, but the package 'color.sty' is not loaded}%
    \renewcommand\color[2][]{}%
  }%
  \providecommand\transparent[1]{%
    \errmessage{(Inkscape) Transparency is used (non-zero) for the text in Inkscape, but the package 'transparent.sty' is not loaded}%
    \renewcommand\transparent[1]{}%
  }%
  \providecommand\rotatebox[2]{#2}%
  \newcommand*\fsize{\dimexpr\f@size pt\relax}%
  \newcommand*\lineheight[1]{\fontsize{\fsize}{#1\fsize}\selectfont}%
  \ifx\svgwidth\undefined%
    \setlength{\unitlength}{127.35202975bp}%
    \ifx\svgscale\undefined%
      \relax%
    \else%
      \setlength{\unitlength}{\unitlength * \real{\svgscale}}%
    \fi%
  \else%
    \setlength{\unitlength}{\svgwidth}%
  \fi%
  \global\let\svgwidth\undefined%
  \global\let\svgscale\undefined%
  \makeatother%
  \begin{picture}(1,0.10366749)%
    \lineheight{1}%
    \setlength\tabcolsep{0pt}%
    \put(0,0){\includegraphics[width=\unitlength,page=1]{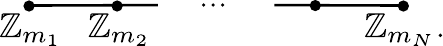}}%
  \end{picture}%
\endgroup%

\end{center}
As such, $\freeprod$ acts on its Bass--Serre tree $T$. The fundamental domain of this action is a path with $N-1$ edges. The action is free on edges and the vertex stabilizers are conjugates $\gamma\cyc{\TheOrder_\nu}\gamma^{-1}$ with $\gamma\in\freeprod$ and $1\leq\nu\leq\TheCone$. By the choice of a generator $\gamma_\nu$ for each $\cyc{\TheOrder_\nu}$ with $1\leq\nu\leq\TheCone$, the link of each vertex carries a cyclic ordering. 

Let us consider a proper embedding of the Bass--Serre tree~$T$ into $\CC$ that respects the local cyclic order on each link. If we choose a regular neighborhood of $T$ inside~$\CC$, we obtain a planar, contractible surface $\Sigma$ (with boundary), see Figure \textup{\ref{fig:constr_fund_domain}} for an example.

\begin{figure}[H]
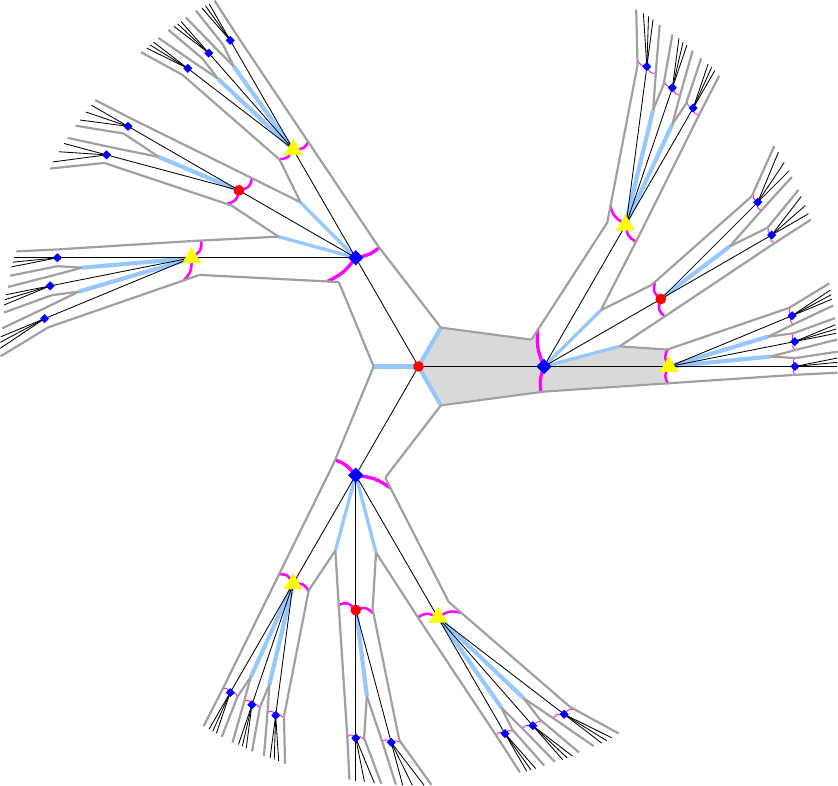
\caption{Thickened Bass--Serre tree for $\freeprod=\cyc{3}\ast\cyc{2}\ast\cyc{4}$ with fundamental domain shaded in gray. The red $\textcolor{red}{\bullet}$, blue~$\textcolor{blue}{\blacklozenge}$ and yellow $\textcolor{yellow}{\blacktriangle}$ are conjugates of the free factors $\cyc{3}$, $\cyc{2}$ and $\cyc{4}$, respectively.}
\label{fig:constr_fund_domain}
\end{figure}

This surface $\Sigma$ inherits a proper $\freeprod$-action from the Bass--Serre tree such that vertex stabilizers act with respect to the cyclic order on the link of the stabilized vertex. Moreover, the action admits a fundamental domain corresponding to the fundamental domain in $T$. 
In particular, we obtain an orbifold structure $\Sigma_\freeprod$. 

A point in $\Sigma_\freeprod$ is a singular point if and only if it corresponds to a vertex of $T$. Hence, the singular points in $\Sigma_\freeprod$ are all cone points and decompose into $\TheCone$ orbits. The quotient $\Sigma/\freeprod$ is a disk with $\TheCone$ distinguished points that correspond to the orbits of the cone points. 

In general, we may choose a fundamental domain $\FD$ that is a disk as pictured in Figure \textup{\ref{fig:fund_domain}} and contains exactly $\TheCone$ cone points $\cp_1,...,\cp_\TheCone$ that lie on the boundary \new{such that each has exactly two adjacent boundary arcs that lie in the same $\freeprod$-orbit.}
\begin{figure}[H]
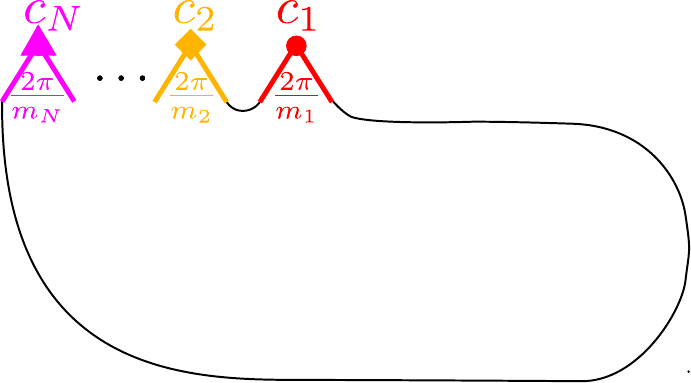
\caption{The fundamental domain $\FD$.}
\label{fig:fund_domain}
\end{figure}

In this particular case, we can also find a metric on $\Sigma$ such that the $\freeprod$-action on $\Sigma$ is isometric. For each $\gamma\in\freeprod$, let 
\[
\lambda_{\gamma^{-1}}:\gamma(\FD)\rightarrow\FD,\gamma(x)\mapsto x.
\] 
This allows us to define 
\begin{align*}
\label{eq:def_metric_iso_freeprod-action}
d:\Sigma\times\Sigma & \rightarrow\RR_{>0}, \numbereq
\\
(x,y) & \mapsto\inf_{(x_0,x_1,...,x_k)\in P(x,y)}\sum_{i=1}^kd_{\CC}(\lambda_{\gamma_i^{-1}}(x_{i-1}),(\lambda_{\gamma_i^{-1}}(x_i)). 
\end{align*}
Here above, $P(x,y)$ denotes the set of polygonal chains $(x_0,x_1,...,x_k)$ of arbitrary length $k\in\NN_0$ inside $\Sigma$ connecting $x$ and $y$ and such that consecutive points $x_{i-1}$ and $x_i$ are contained in $\gamma_i(\FD)$ with $\gamma_i\in\freeprod$ and $1\leq i\leq k$, see \cite[Section 2.1.7]{Flechsig2023} for details on the well-definedness of $d$. 

If we remove the boundary of $\Sigma$, the quotient $\Sigma^\circ/\freeprod$ is homeomorphic to the complex plane with $\TheCone$ distinguished points and associated cyclic groups $\cyc{\TheOrder_\nu}$ for $1\leq\nu\leq\TheCone$. Adding $\freeprod$-orbits of punctures $\freeprod(r_\NPct)$ for $1\leq\NPct\leq\ThePct$ to $\Sigma$ such that $\freeprod(r_\Pc)\neq\freeprod(r_\NPct)$ for $1\leq\Pc,\NPct\leq\ThePct, \Pc\neq\NPct$, we obtain the orbifold called 
\[
\CC(\ThePct,\TheCone,\textbf{\TheOrder}) \text{ with } \textbf{\TheOrder}=(\TheOrder_1,...,\TheOrder_\TheCone)
\]
in \cite{Roushon2021}. In \cite{Allcock2002}, Allcock studied braids on these orbifolds for 
\[
(\ThePct,\TheCone,\textbf{\TheOrder})=(0,2,(2,2)), (0,1,(2)) \text{ and } (1,1,(2)). 
\]
\end{example}

Since we want to study mapping class groups, which requires to fix the boundary, we will consider the orbifold $\Sigma_\freeprod$ with boundary. Moreover, we use the notation $\Sigma_\freeprod(\ThePct)$ for the orbifold with underlying surface (with boundary) 
\begin{equation}
\label{eq:Sigma_minus_pct}
\Sigma(\ThePct):=\Sigma\setminus\freeprod(\{r_1,...,r_\ThePct\}). 
\end{equation}


We will consider a concept of orbifold paths and their homotopy classes introduced in \cite[Chapter III.G, 3]{BridsonHaefliger2011}. There an orbifold is considered more generally as an \textit{\'etale groupoid} $(\mathcal{G},X)$, see \cite[Chapter III.G, 2]{BridsonHaefliger2011}. If $M_\G$ is an orbifold in the sense of Definition \ref{def:good_orb}, the associated \'etale groupoid is given by 
\[
(\mathcal{G},X)=(\G\times M,M). 
\]
In the following, we will simplify the notation using $\G$ instead of $\mathcal{G}=\G\times M$. In particular, we introduce $\G$-paths. These are the $\mathcal{G}$-paths in \cite{BridsonHaefliger2011}. 

\begin{definition}[{$\G$-path}, {\cite[Chapter III.G, 3.1]{BridsonHaefliger2011}}]
\label{def:G-path}
A \textit{$\G$-path} $\xi=(\g_0,c_1,\g_1,...,c_\TheSubdivision,\g_\TheSubdivision)$ in $M_\G$ with initial point $x\in M$ and terminal point $y\in M$ over a subdivision $a=t_0\leq...\leq t_\TheSubdivision=b$ of the interval $[a,b]$ consists of
\begin{enumerate}
\item continuous maps $c_\Subdivision:[t_{\Subdivision-1},t_\Subdivision]\rightarrow M$ for $1\leq\Subdivision\leq\TheSubdivision$ and
\item group elements $\g_\Subdivision\in\G$ such that $\g_0(c_1(t_0))=x$, $\g_\Subdivision(c_{\Subdivision+1}(t_\Subdivision))=c_\Subdivision(t_\Subdivision)$ for $1\leq\Subdivision<\TheSubdivision$ and $\g_\TheSubdivision(y)=c_\TheSubdivision(t_\TheSubdivision)$ (see Figure \ref{fig:G-path}). 
\end{enumerate}

For brevity, we write $(\g_0,c_1,\g_1,...,c_\TheSubdiv)$ for $(\g_0,c_1,\g_1,...,c_\TheSubdiv,\g_\TheSubdiv)$ if $\g_\TheSubdiv=1$. We say a $\G$-path is \textit{continuous} if it is of the form $(\g,c)$. 
\begin{figure}[H]
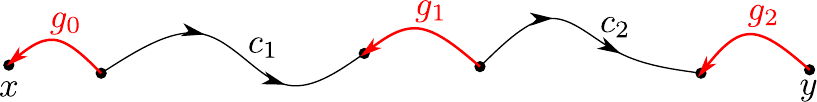
\caption{A $\G$-path.}
\label{fig:G-path}
\end{figure}
\end{definition} 

The following equivalence relation 
identifies certain $\G$-paths whose continuous pieces have the same $\G$-orbits. 

\begin{definition}[{Equivalence of $\G$-paths}, {\cite[Chapter III.G, 3.2]{BridsonHaefliger2011}}]
\label{def:equiv_paths}
Let 
\[
\xi=(\g_0,c_1,\g_1,...,c_\TheSubdivision,\g_\TheSubdivision) 
\]
be a $\G$-path over 
$a=t_0\leq...\leq t_\TheSubdivision=b$. 
\begin{enumerate}
\item
\label{def:equiv_paths_subdiv} 
A \textit{subdivision} of $\xi$ is a $\G$-path obtained from $\xi$ by choosing $t'\in[t_{\Subdivision-1},t_\Subdivision]$ for some $1\leq\Subdiv\leq\TheSubdiv$  and replacing the entry $c_\Subdiv$ with the sequence 
\[
(c_\Subdivision\vert_{[t_{\Subdivision-1},t']},1,c_\Subdivision\vert_{[t',t_\Subdivision]}). 
\] 
\item
\label{def:equiv_paths_shift} 
A \textit{shift} of $\xi$ is a $\G$-path obtained from $\xi$ by choosing $h\in\G$ and replacing a subsequence $(\g_{\Subdiv-1},c_\Subdiv,\g_\Subdiv)$ for some $1\leq\Subdiv\leq\TheSubdiv$ with 
\[
(\g_{\Subdiv-1}h^{-1},h\cdot c_\Subdiv,h\g_\Subdiv). 
\] 
\end{enumerate}
We say that two $\G$-paths are \textit{equivalent} if one can be obtained from the other by a sequence of subdivisions, inverses of subdivisions and shifts. 
\end{definition}
\begin{figure}[H]
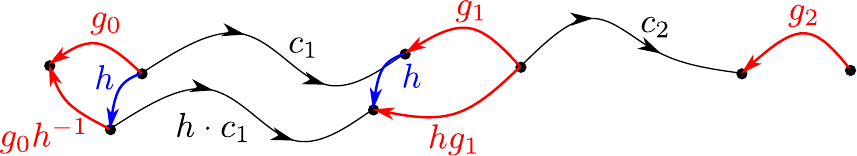
\caption{Two $\G$-paths equivalent by a shift.}
\label{fig:equivalent_G-paths}
\end{figure}

Using this equivalence relation, we mimic the homotopy relation for paths in topological spaces for $\G$-paths. 

\begin{definition}[{Homotopy of $\G$-paths}, {\cite[Chapter III.G, 3.5]{BridsonHaefliger2011}}]
\label{def:homo_G-paths}
An \textit{elementary homotopy} between two $\G$-paths $\xi$ and $\tilde{\xi}$ is a family of $\G$-paths $\xi_s=(\g_0,c_1^s,...,\g_\TheSubdivision)$ over the subdivisions $0=t_0\leq t_1\leq...\leq t_\TheSubdivision=1$. The family $\xi_s$ is parametrized by $s\in[s_0,s_1]$ such that $c_\Subdivision^s$ depends continuously on the parameter and $\xi^{s_0}=\xi$, $\xi^{s_1}=\tilde{\xi}$. 

\begin{figure}[H]
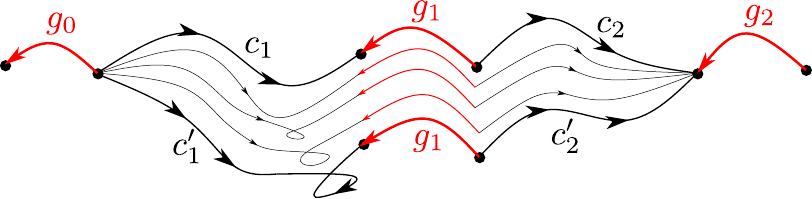
\caption{An elementary homotopy of $\G$-paths.}
\label{fig:el_homo_G-paths}
\end{figure}

Two $\G$-paths are \textit{homotopic (relative to their endpoints)} if one can pass from the first to the second by a sequence of the following operations: 
\begin{enumerate}
\item equivalence of $\G$-paths, 
\item elementary homotopies. 
\end{enumerate}
\end{definition}

Homotopy classes of \textit{$\G$-loops based at a fixed point $x\in M$} lead to the concept of \textit{orbifold fundamental groups}, see \cite[Chapter III.G, 3]{BridsonHaefliger2011} for details. 

We close this section with a lemma that allows us to restrict to continuous $\G$-paths in many situations. For a fixed index $1\leq\NSubdiv\leq\TheSubdiv$, the shift defined in Definition \ref{def:equiv_paths} allows for the following choice: \new{The element $h$ shifts $c_\NSubdiv$ to $c_\NSubdiv':=h\cdot c_\NSubdiv$, $\g_{\NSubdiv-1}'=\g_{\NSubdiv-1}h^{-1}$ and $\g_\NSubdiv'=h\g_\NSubdiv$. Thus, the path $\xi$ is equivalent to} 
\begin{equation}
\label{eq:restricted_shift}
\xi'=(\g_0,c_1,\g_1,...c_{\NSubdiv-1},\g_{\NSubdiv-1}',c_\NSubdiv',\g_\NSubdiv',c_{\NSubdiv+1},...,c_\TheSubdivision,\g_\TheSubdivision). 
\end{equation}
If we choose $h=\g_\NSubdiv^{-1}$, the element $\g_\NSubdiv'$ is trivial. Replacing $c_\NSubdiv',1,c_{\NSubdiv+1}$ by $c_\NSubdiv'\cup c_{\NSubdiv+1}$, we obtain the path 
\begin{equation}
\label{eq:inverted_subdiv}
\tilde{\xi}'=(\g_0,c_1,\g_1,...c_{\NSubdiv-1},\g_{\NSubdiv-1}',c_\NSubdiv'\cup c_{\NSubdiv+1},\g_{\NSubdiv+1},c_{\NSubdiv+1},\g_{\NSubdiv+2},...,c_\TheSubdivision,\g_\TheSubdivision)
\end{equation}
which is equivalent to $\xi$ and has shorter subdivision length. 

\begin{lemma}[{\cite[Chapter III.G, 3.9(1)]{BridsonHaefliger2011}}]
\label{lem:cont_path_homo}
Every $\G$-path connecting $x$ to $y$ in $M_\G$ is equivalent to a unique continuous $\G$-path $(\g,c)$ with $c:I\rightarrow M$ connecting $\g^{-1}(x)$ to $y$. 
\end{lemma}

\new{For a proof of Lemma \ref{lem:cont_path_homo}, we refer to \cite[Lemma 2.12]{Flechsig2023}.} 

\section{Orbifold mapping class groups}
\label{sec:basics_mcg_orb}

\setlist[enumerate,1]{label=\textup{(\arabic*)}, ref=\thetheorem(\arabic*)}
\setlist[enumerate,2]{label=\textup{\alph*)}, ref=\alph*)}

An important approach to the Artin braid groups is the identification with mapping class groups of punctured disks, see, for instance, \cite[Section 9.1.3]{FarbMargalit2011}. For disks and other surfaces, mapping class groups can be studied using their action on arcs and simple closed curves. This is based on the fact that homotopic arcs and homotopic simple closed curves are ambient isotopic. For instance, this follows from the bigon criterion which is stated in Proposition \ref{prop:bigon_crit_surf}. 

This section sets the base for a similar discussion in the orbifold case: 
It introduces a notion of \textit{orbifold mapping class groups} (see Definition \ref{def:mcg_orb}) and establishes orbifold analogs of arcs and simple closed curves, called \textit{$\freeprod$-arcs} and \textit{simple closed $\freeprod$-curves} (see Definitions \ref{def:freeprod-arcs} and \ref{def:freeprod-scc}). In analogy to the surface case, we prove a \textit{bigon criterion} for both $\freeprod$-arcs and simple closed $\freeprod$-curves (see Propositions~\ref{prop:bigon_crit_orb} and~\ref{prop:bigon_crit_orb_scc}). 

As in the surface case, the bigon criterion implies that homotopic $\freeprod$-arcs and homotopic simple closed $\freeprod$-curves are ambient isotopic (see Lemma \ref{prop:homo_ind_amb_iso}). This allows us to understand orbifold mapping class groups via their action on $\freeprod$-arcs and simple closed $\freeprod$-curves. In particular, we will prove that the orbifold mapping class group of $\Sigma_\freeprod$, in the absence of marked points and punctures, is trivial (see Lemma \ref{lem:orb_mcg_trivial}). This will serve as the base case for Section \ref{sec:orb_mcg_marked_pts}, where we study subgroups of orbifold mapping class groups with marked points. 

\subsection{Definition and first examples}
\label{subsec:def_mcg_and_ex}

\new{Given an orbifold $M_\G$,} we want to define its mapping class group as a group of certain homeomorphisms of $M$ modulo an equivalence relation. This generalizes the concept of mapping class groups of manifolds. If the acting group $\G$ is trivial, then the orbifold mapping class group of $M_{\{1\}}$ defined below coincides with the mapping class group of $M$. 

As it is usual in the case of manifolds, we consider homeomorphisms of $M$ 
that fix the boundary 
pointwise. Further, the \new{orbifold} mapping class group should reflect the structure of the $\G$-action on $M$. For this reason, we restrict to \textit{$\G$-equivariant} homeomorphisms. Subject to an isotopy relation on this group, we define the orbifold mapping class group. 

\begin{definition}[$\G$-equivariance]
\label{def:G-equivar}
Let $X,Y$ be topological $\G$-spaces. A map $\psi:X\rightarrow Y$ is $\G$-equivariant if 
\[
\psi(\g(x))=\g(\psi(x)) \text{ for all } \g\in\G, x\in X. 
\]
If $M_\G$ is an orbifold, the group of $\G$-equivariant self-homeomorphisms of $M$ that fix the boundary pointwise is denoted by $\HomeoOrb{}{M_\G,\partial M}$. \new{As a subgroup of the homeomorphism group of~$M$,} it carries the structure of a topological group if it is endowed with the compact-open topology. 
\end{definition}

\begin{definition}[Ambient isotopy]
\label{def:ambient_iso}
An \textit{ambient isotopy} is a continuous map 
\[
I\rightarrow\HomeoOrb{}{M_\G,\partial M}. 
\]
Two $\G$-equivariant homeomorphisms $H,H'$ are \textit{ambient isotopic}, denoted by $H\sim~H'$, if there exists an ambient isotopy $H_t$ with $H_0=H$ and $H_1=H'$. 
\end{definition}

\begin{definition}[Orbifold mapping class group]
\label{def:mcg_orb}
The group of $\G$-equivariant homeomorphisms that fix the boundary pointwise modulo ambient isotopy
\[
\MapOrb{}{M_\G}:=\HomeoOrb{}{M_\G,\partial M}/\sim
\]
is called the \textit{mapping class group} of $M_\G$. 
\end{definition}

Based on the fact that $\HomeoOrb{}{M_\G,\partial M}$ is a topological group, the mapping class group also carries the structure of a topological group. A particular basic example is the following: 

\begin{example}[$\MapOrb{}{D_{\cycm}}$ and Alexander trick]
\label{ex:Map(D_cycm)}
Let $H$ be a self-homeomorphism of $D$ that represents an element in $\MapOrb{}{D_{\cycm}}$ and 
\begin{align*}
H_t:D\rightarrow D, x\mapsto\begin{cases}
(1-t)H\left(\frac{x}{1-t}\right) & 0\leq\vert x\vert\leq 1-t, 
\\
x & 1-t\leq\vert x\vert\leq 1. 
\end{cases}
\end{align*}
For each $t\in I$, the map $H_t$ is $\cycm$-equivariant and fixes the boundary. Moreover, $H_0=H$ and $H_1=\id_D$. Since the map 
\[
I\rightarrow\HomeoOrb{}{D_{\cycm},\partial D},t\mapsto H_t
\]
is continuous, it yields an ambient isotopy that is known as the \textit{Alexander trick}, see \cite[Lemma 2.1]{FarbMargalit2011}. It shows that $\MapOrb{}{D_{\cycm}}$ is trivial. 
\end{example}

We can encode additional information in the orbifold mapping class group if we endow $M_\G$ with a finite set of marked points. 

\begin{definition}[Orbifold mapping class group with marked points]
\label{def:mcg_marked_pts}
Let $M_{\G}$ be an orbifold and let us fix a set of non-singular marked points $P=\{p_1,...,p_\TheStrand\}$ in~$M$ such that $\G(p_\Str)\neq\G(p_\Strand)$ for $1\leq\Str,\Strand\leq\TheStrand,\Str\neq\Strand$. \new{By $\HomeoOrb{\TheStrand}{M_\G,\partial M}$,} we denote the subgroup of homeomorphisms that preserve the orbit of the marked points $\G(P)$ as a set: 
\[
\{H\in\HomeoOrb{}{M_\G,\partial M}\mid H(\G(P))=\G(P)\}. 
\]
We consider these homeomorphisms up to ambient isotopies $I\rightarrow\HomeoOrb{\TheStrand}{M_\G,\partial M}$. The corresponding equivalence relation is denoted by $\sim_\TheStrand$. 
\new{By 
\[
\MapOrb{\TheStrand}{M_\G}:=\HomeoOrb{\TheStrand}{M_\G,\partial M}/\sim_\TheStrand,} 
\]
we denote the \textit{orbifold mapping class group of $M_\G$ with respect to the $\TheStrand$ marked~points}. 
\end{definition}

We stress that the orbit of marked points $\G(P)$ is a discrete set. Hence, a homotopy $H_t$ through $\HomeoOrb{\TheStrand}{M_\G,\partial M}$ is constant on marked points, i.e.\ $H_t(p_\Strand)=H_0(p_\Strand)=H_1(p_\Strand)$ for each $1\leq\Strand\leq\TheStrand$ and $t\in I$. 

While the group $\MapOrb{}{D_{\cycm}}$ is trivial, \new{we obtain for the group $\MapOrb{1}{D_{\cycm}}$ with one orbit of marked points:} 

\begin{example}
\label{ex:orb_mcg_one_pct_ZZ}
$\MapOrb{1}{D_{\cycm}}$ is infinite cyclic. The generator is represented by a $\frac{2\pi}{\TheOrder}$-twist $\TwsC$ around the center $\cp$ 
(see Figure \textup{\ref{fig:2_pi_m-twist}}). 
\end{example}

\begin{figure}[H]
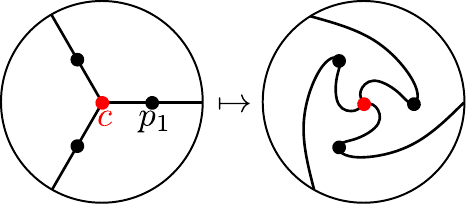
\caption{The $\frac{2\pi}{\TheOrder}$-twist $\TwsC$ for $\TheOrder=3$.}
\label{fig:2_pi_m-twist}
\end{figure}

We will give a proof of Example \ref{ex:orb_mcg_one_pct_ZZ} in the next section. It will follow as a corollary from Proposition \ref{prop:iso_Map_orb_disk} using the Alexander trick. 

Homeomorphisms that map each marked point inside its $\G$-orbit yield the so called \textit{pure orbifold mapping class group}: 

\begin{definition}[Pure orbifold mapping class group]
\label{def:pure_mcg} 
Let $\PHomeoOrb{\TheStrand}{M_\G,\partial M}$ be the group of \textit{pure homeomorphisms} 
\[
\{H\in\HomeoOrb{\TheStrand}{M_\G,\partial M}\mid H(p_\Strand)=\g_\Strand(p_{\Strand}) \text{ with } \g_\Strand\in\G \text{ for all } 1\leq\Strand\leq\TheStrand\}. 
\] 
The subgroup of $\MapOrb{\TheStrand}{M_\G}$ induced by pure homeomorphisms is called the \textit{pure orbifold mapping class group} 
\[
\PMapOrb{\TheStrand}{M_\G}:=\PHomeoOrb{\TheStrand}{M_\G,\partial M}/\sim_\TheStrand. 
\]
\end{definition}

\new{At this point,} we recall that a homeomorphism in the pure mapping class group of a manifold fixes each of the marked points. In contrast, we only require the homeomorphisms in $\PMapOrb{\TheStrand}{M_\G}$ to preserve the orbit of each marked point but not to fix the points themselves. Further, we emphasize that we allow different group actions on different orbits of marked points, i.e.\ $H(p_\Str)=\g_\Str(p_\Str)$ and $H(p_\Strand)=\g_\Strand(p_\Strand)$ with $\g_\Str\neq\g_\Strand$ for $\Str\neq\Strand$. 

\subsection{Arcs and simple closed curves}
\label{subsec:arcs_and_scc}

The study of arcs and simple closed curves is a useful tool to deduce information about mapping class groups. This is based on the fact that each homotopy of arcs or simple closed curves can be realized by an ambient isotopy. For instance, this follows from the so called \textit{bigon criterion} for arcs and simple closed curves which we recall in the following. For further information, we refer to \cite[Sections 1.2.2 and~1.2.7]{FarbMargalit2011}. 

Let $\bar{S}$ be a connected compact surface with non-empty boundary. Removing finitely many points from the interior, we obtain a connected surface $S$. Further, we endow $S$ with two finite sets of marked points $P$ and $Q$. The set $P$ is contained in the interior while the points from $Q$ lie on the boundary $\partial S$. 

A continuous map $b:I\rightarrow S$ is called a \textit{path}. We say that $b$ connects $b(0)$ to $b(1)$. In the following, we restrict to paths with $b(0)\in P$ and $b(1)\in Q$. A path $b:I\rightarrow S$ is called \textit{proper} if its image intersects $P,Q$ and $\partial S$ precisely in its endpoints. For every path, we identify $b$ with its image in $S$. If a path $b$ embeds $I$ into~$S$, we call it an \textit{arc}. 
All paths discussed in this article are proper. Mainly, we will deal with arcs. We consider proper paths up to homotopies relative to the endpoints. By $[b]_S$ we denote the homotopy class of a path $b$ in $S$.  


A continuous map $c:S^1\rightarrow S$ is called a \textit{closed curve}. We call a closed curve \textit{proper} if the image of $c$ does not intersect with any marked points or the boundary~$\partial S$. For every closed curve \new{$c$,} we identify $c$ with its image in $S$. A \textit{simple closed curve} is a closed curve $c:S^1\rightarrow S$ that embeds $S^1$ into $S$. 
As for paths, we only consider proper closed curves up to homotopies in the class of proper closed curves and denote the homotopy class of a closed curve $c$ by~$[c]_S$. 

\newcommand{\TheArc}{l}
\renewcommand{\TheDim}{k}

For a set $T$, let $\vert T\vert$ denote the cardinality. If $b,b'$ are arcs in $S$, let 
\[
i_S([b]_S,[b']_S)=\min\{\vert\tilde{b}\cap\tilde{b}'\vert\mid\tilde{b}\in[b]_S, \tilde{b}'\in[b']_S \text{ arcs}\} 
\] 
be the \textit{intersection number} of $[b]_S$ and $[b']_S$. Representing arcs $b$ and $b'$ are in \textit{minimal position} if $\vert b\cap b'\vert=i_S([b]_S,[b']_S)$. 

Likewise, the intersection number and minimal position are defined for homotopy classes of simple closed curves. 

\begin{definition}[{Bigon}, {\cite[Section 1.2.4]{FarbMargalit2011}}]
\label{def:bigon}
Two paths $b$ and~$b'$ form a \textit{bigon} if there exists a disk $D\subseteq S$ such that 
\begin{enumerate}[label={(\arabic*)}, ref={(\arabic*)}]
\item \label{def:bigon_it1}
no marked point is contained in the interior of $D$, 
\item \label{def:bigon_it2}
at most one marked point lies in $\partial D$, 
\item \label{def:bigon_it3}
$D\cap(b\cup b')=\partial D$, 
\item \label{def:bigon_it4}
$\partial D\cap b$ and $\partial D\cap b'$ are non-empty and 
\item \label{def:bigon_it5}
$\partial D\cap b$ and $\partial D\cap b'$ are connected. 
\end{enumerate}

Likewise, we define bigons for simple closed curves. Due to 
properness, the condition from \ref{def:bigon_it2} is automatically satisfied for 
closed curves. See Figure \ref{fig:bigons_and_non-bigons} for a list of examples and \new{counterexamples} of bigons. 

If two arcs or two simple closed curves form a bigon, they intersect exactly twice in $\partial D$. We call these intersection points the \textit{endpoints} of the bigon.  

\new{Clearly, for this definition,} it is not necessary that $S$ stems from a compact surface. In particular, we may also consider bigons in the tree-shaped surface $\S$. 
\end{definition}

\begin{figure}[H]
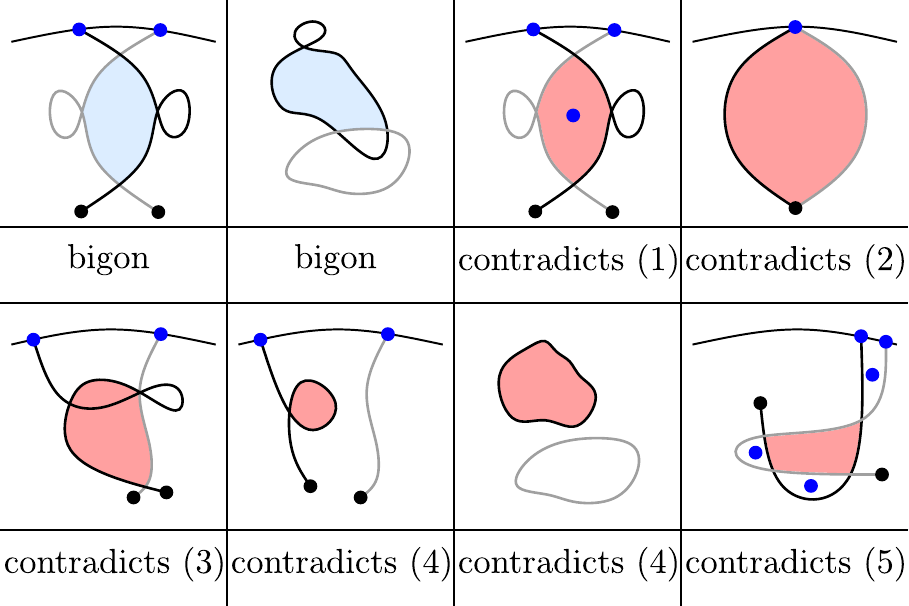
\caption{Examples and \new{counterexamples} of bigons.}
\label{fig:bigons_and_non-bigons}
\end{figure}

For arcs and simple closed curves, the following is well known: 

\begin{proposition}[{Bigon criterion}, {\cite[Proposition 1.7]{FarbMargalit2011}}]
\label{prop:bigon_crit_surf}
Two transverse arcs or two transverse simple closed curves in a surface $S$ are in minimal position if and only if they do not form a bigon. 
\end{proposition}

A similar criterion also holds for arcs and simple closed curves in the tree-shaped surface $\S$ from \eqref{eq:Sigma_minus_pct}. 

\begin{corollary}[Bigon criterion in $\S$]
\label{cor:bigon-crit_Sigma}
Two transverse arcs or two transverse simple closed curves in the tree-shaped surface $\S$ are in minimal position if and only if they do not form a bigon. 
\end{corollary}
\new{Since we want to consider compact subsurfaces of $\Sigma(\ThePct)$, for the proof and in the following,} it is often more convenient to consider the surface $\Sigma$ with marked points instead of punctures in $\freeprod(\{r_1,...,r_\ThePct\})$. We will also use the notation $\Sigma(\ThePct)$ for the surface $\Sigma$ with these marked points. 
\begin{proof}[Proof of  Corollary \textup{\ref{cor:bigon-crit_Sigma}}] 
On the one hand, it is obvious that two arcs in $\S$ that form a bigon are not in minimal position. 

On the other hand, if two arcs $b$ and $b'$ do not form a bigon, they are in minimal position in each compact subsurface that contains both arcs. If we assume that these arcs are not in minimal position in $\S$, there exist homotopic arcs $\tilde{b}$ and $\tilde{b}'$ that intersect less than $b$ and $b'$. Since the connecting homotopies are compactly supported, their images only intersect finitely many $\freeprod$-translates of $\FD(\ThePct)$. The union of these $\freeprod$-translates yields a compact subsurface $\Sigma'$ such that $\tilde{b}$ and $\tilde{b}'$ are contained in $\Sigma'$. Furthermore, in this surface $\Sigma'$ the arcs $\tilde{b}$ and $\tilde{b}'$ are homotopic to $b$ and $b'$, respectively. This contradicts the fact that $b$ and $b'$ are in minimal position in $\Sigma'$. Thus, $b$ and $b'$ are in minimal
position in~$\S$. 

The proof of the criterion for simple closed curves follows verbatim. 
\end{proof}

\subsection{$\freeprod$-arcs and simple closed $\freeprod$-curves}
\label{subsec:freeprod-arcs}

Similarly, we want to discuss the orbifold analogs of arcs and simple closed curves, called \textit{$\freeprod$-arcs} and \textit{simple closed $\freeprod$-curves}. \new{The main goal of this section} is to establish a bigon criterion for $\freeprod$-arcs and simple closed $\freeprod$-curves. 

For this purpose, 
we consider the tree shaped surface $\S$ with marked points in $\freeprod(\{r_1,...,r_\ThePct\})$ and the induced orbifold $\SG$. 

Now we endow the surface $\S$ with two sets of marked points: $P=\{p_1,...,p_\TheStrand\}$ in the interior of the fundamental domain $\FD(\ThePct)$ with $p_\Str\neq p_\Strand$ for $\Str\neq\Strand$ 
and $Q=\{q_1,...,q_{\TheStrand'}\}$ in $\partial\S\cap~\FD(\ThePct)$ with $q_\NStrand\neq q_\NNStrand$ for $\NStrand\neq\NNStrand$ such that all points $q_\Strand$ lie in the same connected component of $\partial\S\cap \FD(\ThePct)$. 

In the following, we consider $\freeprod$-paths that connect points in $P$ to points in $Q$ (see Definition \ref{def:G-path}). Recall that by Lemma \ref{lem:cont_path_homo} each $\freeprod$-path is equivalent to a unique continuous $\freeprod$-path. Considering $\freeprod$-paths \textit{without self-intersections} leads to an orbifold analog of arcs:   

\begin{definition}[$\freeprod$-arc]
\label{def:freeprod-arcs}
Let $(\gamma,b)$ be a continuous $\freeprod$-path. 
\new{By $\freeprod(b)$,} we denote the union of the $\freeprod$-translates of $b$. We call the $\freeprod$-path $(\gamma,b)$ \textit{proper} if $b$ intersects $\freeprod(P)$ and $\partial\S$ precisely in its endpoints. 
If $b$ embeds $I$ into the tree-shaped surface and the $\freeprod$-translates of $b$ are disjoint, we call $(\gamma,b)$ a \textit{$\freeprod$-arc}. Given two $\freeprod$-arcs $(\gamma,b)$ and $(\gamma',b')$, we call them \textit{disjoint} if the orbits $\freeprod(b)$ and $\freeprod(b')$ are disjoint. 

A $\freeprod$-path $\beta=(\gamma_0,b_1,\gamma_1,...,b_\TheSubdiv,\gamma_\TheSubdiv)$ is a $\freeprod$-arc if the unique equivalent continuous $\freeprod$-path $(\gamma,b)$ is a $\freeprod$-arc. We call two $\freeprod$-arcs $\beta$ and $\beta'$ \textit{disjoint} 
if the unique equivalent continuous $\freeprod$-paths are disjoint. 
\end{definition}

An example and a \new{counterexample} for a $\freeprod$-arc are pictured in Figure \ref{fig:freeprod-arcs}. \new{In the upper row from left to right, the figure depicts a $\cyc{3}$-arc, its unique equivalent continuous $\cyc{3}$-arc and its $\cyc{3}$-orbit. In the bottom row, the analogs are given for a $\cyc{3}$-path that is not a $\cyc{3}$-arc.} 
\begin{figure}[H]
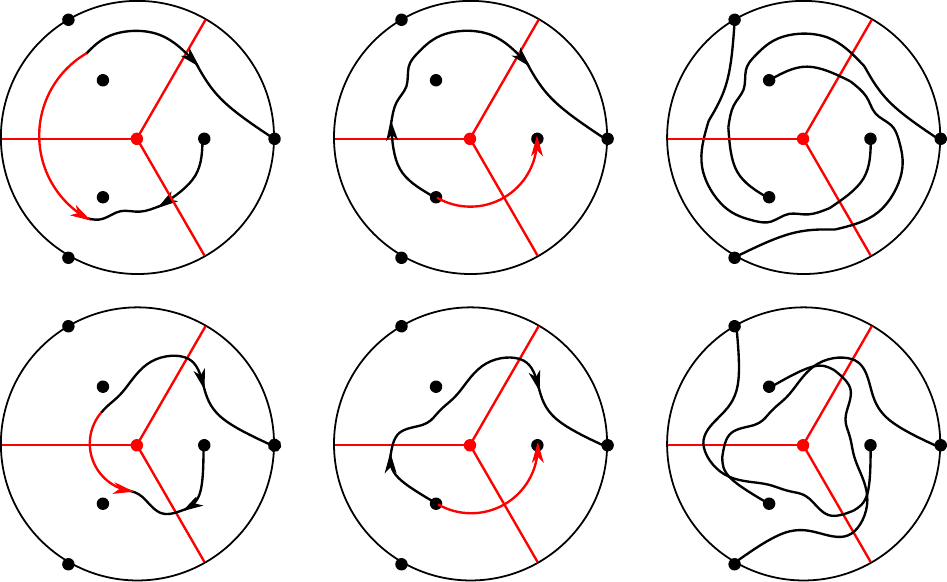
\caption{An example (upper row) and a \new{counterexample} (bottom row) of a $\freeprod$-arc.}
\label{fig:freeprod-arcs}
\end{figure}

As in the case of arcs, we homotope $\freeprod$-arcs through proper $\freeprod$-paths relative to the endpoints: 

\begin{definition}[{Homotopy and isotopy of $\freeprod$-arcs}, {\cite[Chapter III.G, 3.5]{BridsonHaefliger2011}}]
\label{def:homo_iso_orb}
Two $\freeprod$-arcs $\beta$ and $\beta'$ are \textit{homotopic} if one can be obtained from the other by a sequence of the following operations: 
\begin{enumerate}
\item equivalence of $\freeprod$-paths (see Definition \ref{def:equiv_paths}),
\item elementary homotopies $\beta_t$ \new{(see Definition~\ref{def:homo_G-paths})} such that $\beta_t$ is proper for all~$t$. 
\end{enumerate}
If in each elementary homotopy that appears in the finite sequence the $\freeprod$-path $\beta_t$ is a $\freeprod$-arc for every $t$, we call $\beta$ and $\beta'$ \textit{isotopic}. By $[\beta]$ we denote the homotopy class of $\beta$. For a continuous $\freeprod$-arc $(\gamma,b)$, we denote the homotopy class by $[\gamma,b]$. 
\end{definition}

Further, we introduce an orbifold analog of simple closed curves. 

\begin{definition}[Simple closed $\freeprod$-curve]
\label{def:freeprod-scc}
A simple closed curve $c:S^1\rightarrow\S$ is a \textit{simple closed $\freeprod$-curve} if the $\freeprod$-translates are either disjoint or coincide. 
A simple closed $\freeprod$-curve $c$ is \textit{proper} if it does not intersect $\freeprod(P)$ and $\partial\S$. Given two simple closed $\freeprod$-arcs $c$ and $c'$, we call them \textit{disjoint} if the orbits $\freeprod(c)$ and $\freeprod(c')$ are disjoint. 
\end{definition}

An example and a \new{counterexample} for a simple closed $\freeprod$-curve are pictured in Figure~\ref{fig:freeprod-scc}. \new{In the upper row from left to right, the figure depicts a simple closed $\cyc{3}$-curve and its $\cyc{3}$-orbit. In the bottom row, the analogs are given for a simple closed curve that is not a simple closed $\cyc{3}$-curve.} 

\begin{figure}[H]
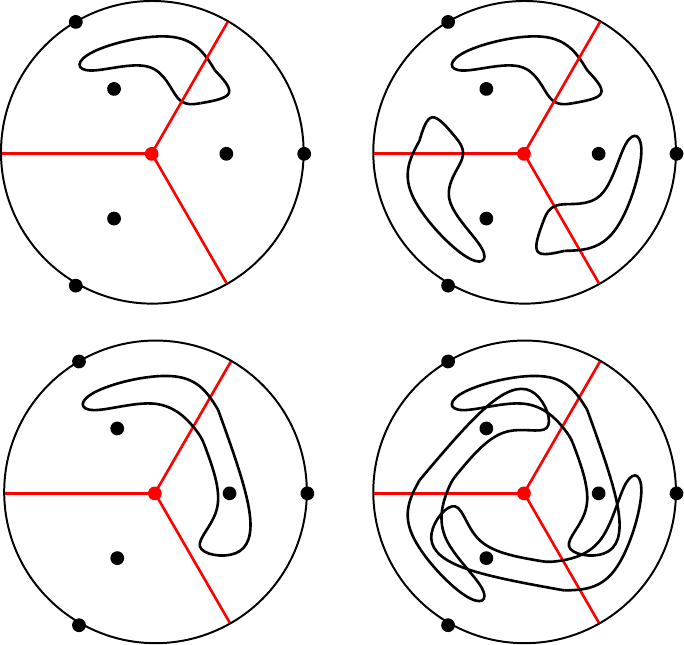
\caption{An example (upper row) and a \new{counterexample} (bottom row) of a simple closed $\freeprod$-curve.}
\label{fig:freeprod-scc}
\end{figure}

\begin{definition}[Homotopy and isotopy of simple closed $\freeprod$-curves]
\label{def:homo_iso_freeprod-scc}
Two simple closed $\freeprod$-curves are \textit{homotopic} if they are homotopic as simple closed curves in $\S$. Two simple closed curves are \textit{isotopic} if they are isotopic as simple closed curves via an isotopy $h_t$ such that $h_t$ is a simple closed $\freeprod$-curve at each time $t\in I$. 
\end{definition}

Motivated by the intersection number of homotopy classes of 
arcs and simple closed curves, we define \textit{intersection number} and \textit{minimal position} for homotopy classes of $\freeprod$-arcs and simple closed $\freeprod$-curves. 

\newpage

\begin{definition}[Minimal position for $\freeprod$-arcs and simple closed $\freeprod$-curves]
\label{def:intersec_num_disj_min_pos}
Let $\beta$ and $\beta'$ be $\freeprod$-arcs in $\SG$. The homotopy classes $[\beta]$ and $[\beta']$ have \textit{intersection number} 
\[
i([\beta],[\beta'])=\min\{\vert\tilde{\beta}\cap\freeprod(\tilde{\beta}')\vert\mid\tilde{\beta}\in[\beta] \text{ and } \tilde{\beta}'\in[\beta'] \;  \freeprod\text{-arcs}\}. 
\]
Here above, we identify a $\freeprod$-arc $(\gamma_0,b_1,\gamma_1,...,b_\TheSubdiv,\gamma_\TheSubdiv)$ with the image of the induced continuous function $\bigcupdot_{\Subdiv=1}^\TheSubdiv[t_{\Subdiv-1},t_\Subdiv]\rightarrow\S, s\mapsto b_\Subdivision(s)$ for $s\in[t_{\Subdiv-1},t_\Subdiv]$. 

If $i([\beta],[\beta'])=0$, for $[\beta]$ and $[\beta']$ there exist representatives $\tilde{\beta}$ and $\tilde{\beta}'$ that are $\freeprod$-arcs and $\tilde{\beta}$ is disjoint from $\tilde{\beta}'$. 
Representing $\freeprod$-arcs 
$\beta$ and $\beta'$ are in \textit{minimal position} if $\vert\beta\cap\freeprod(\beta')\vert=i([\beta],[\beta'])$. 

Likewise, we define the \textit{intersection number} for homotopy classes of simple closed $\freeprod$-curves and \textit{minimal position} for simple closed $\freeprod$-curves. 
\end{definition}

For $\freeprod$-arcs, we further observe: 

\begin{remark}
\label{rem:int_number}
By Lemma \ref{lem:cont_path_homo}, each $\freeprod$-arc is equivalent to a unique continuous 
$\freeprod$-arc. Since the equivalence relation preserves the $\freeprod$-orbit of a $\freeprod$-arc, the intersection number satisfies: 
\[
i([\beta],[\beta'])=\min\{\vert b\cap\freeprod(b')\vert\mid(\gamma,b)\in[\beta], (\gamma',b')\in[\beta'] \text{ continuous } \freeprod\text{-arcs}\}
\]
\end{remark}

The following yields a sufficient condition for $\freeprod$-arcs and simple closed $\freeprod$-curves in minimal position: 

\begin{lemma}
\label{lem:freeprod-arcs_curves_all_translates_in_min_pos}
Let $(\gamma,b)$ and $(\gamma',b')$ be continuous $\freeprod$-arcs in $\SG$. If $b$ 
is in minimal position with each $\freeprod$-translate of $b'$, 
the $\freeprod$-arcs $(\gamma,b)$ and $(\gamma',b')$ are in minimal position. 
\end{lemma}
\begin{proof}
If $b$ is in minimal position with each $\freeprod$-translate of $b'$, we have 
\begin{equation}
\label{eq:min_pos_each_translate}
i_{\S}([b]_{\S},[\tilde{\gamma}'(b')]_{\S})=\vert b\cap\tilde{\gamma}'(b')\vert
\end{equation} 
for each $\tilde{\gamma}'\in\freeprod$. Further, recall that the intersection number of $[\gamma,b]$ and $[\gamma',b']$ is given by 
\[
i([\gamma,b],[\gamma',b'])=\min\{\vert\tilde{b}\cap\freeprod(\tilde{b}')\vert\mid(\gamma,\tilde{b})\in[\gamma,b],(\gamma',\tilde{b}')\in[\gamma',b'] \; \freeprod\text{-arcs}\}. 
\]
For $(\gamma,\tilde{b})\in[\gamma,b]$ and $(\gamma',\tilde{b}')\in[\gamma',b']$, we observe: 
\begin{align*}
\label{eq:suff_cond_min_pos}
\vert\tilde{b}\cap\freeprod(\tilde{b}')\vert\mystackrel{}= & \sum_{\tilde{\gamma}'\in\freeprod}\vert\tilde{b}\cap\tilde{\gamma}'(\tilde{b}')\vert & \mystackrel{}\geq & \sum_{\tilde{\gamma}'\in\freeprod}i_{\S}([b]_{\S},[\tilde{\gamma}'(b')]_{\S}) \numbereq
\\
\mystackrel{(\ref{eq:min_pos_each_translate})}= & \sum_{\tilde{\gamma}'\in\freeprod}\vert b\cap\tilde{\gamma}'(b')\vert & \mystackrel{}= & \vert b\cap\freeprod(b')\vert. 
\end{align*}
Here above, the first equation follows from the fact that $(\gamma',\tilde{b}')$ is a $\freeprod$-arc. 
The definition of $i_ {\S}$ implies the following $\geq$-estimate, then we apply \eqref{eq:min_pos_each_translate} and the last equation follows as the first one. Finally, the above estimate implies that 
\[
i([\gamma,b],[\gamma',b'])\geq\vert b\cap\freeprod(b')\vert. 
\]
On the other hand, the definition of the intersection number yields 
\[
i([\gamma,b],[\gamma',b'])\leq\vert b\cap\freeprod(b')\vert. 
\]
Hence, $i([\gamma,b],[\gamma',b'])=\vert b\cap\freeprod(b')\vert$, i.e.\ $(\gamma,b)$ and $(\gamma',b')$ are in minimal position. 
\end{proof}

For the characterization of $\freeprod$-arcs and simple closed $\freeprod$-curves in minimal position, we introduce orbifold analogs of bigons: 

\begin{definition}[Pseudo-bigons and bigons for $\freeprod$-paths and simple closed $\freeprod$-curves]
\label{def:bigon_freeprod-arc}
Two continuous $\freeprod$-paths $(\gamma,b)$ and $(\gamma',b')$ form a \textit{pseudo-bigon} if there exist $\freeprod$-translates $\tilde{\gamma}(b)$ and $\tilde{\gamma}'(b')$ such that the paths $\tilde{\gamma}(b)$ and $\tilde{\gamma}'(b')$ form a bigon (in the sense of Definition~\ref{def:bigon}). 
If no other $\freeprod$-translate of $b$ or $b'$ intersects the bigon spanned by $\tilde{\gamma}(b)$ and $\tilde{\gamma}'(b')$, we say $(\gamma,b)$ and $(\gamma',b')$ form a \textit{bigon}. See Figure \ref{fig:bigon_pseudo_bigon} for examples of a bigon and a pseudo-bigon. 

Likewise, we define pseudo-bigons and bigons of simple closed $\freeprod$-curves. 

Two $\freeprod$-paths $\beta$ and $\beta'$ form a \textit{pseudo-bigon} or a \textit{bigon} if the unique equivalent continuous $\freeprod$-paths $(\gamma,b)$ and $(\gamma',b')$ form a pseudo-bigon or a bigon, respectively. 
\end{definition}

\begin{figure}[H]
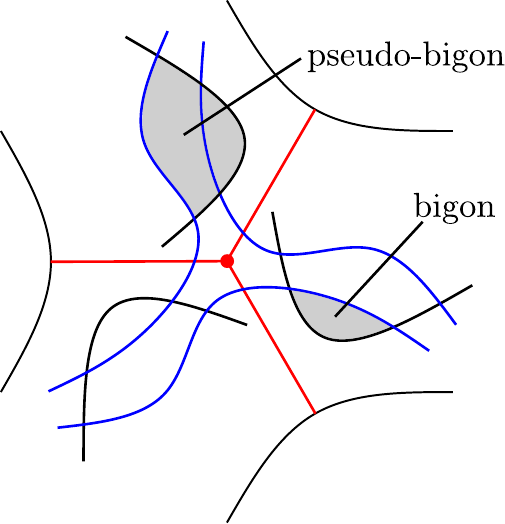
\caption{Examples of a bigon and a pseudo-bigon.}
\label{fig:bigon_pseudo_bigon}
\end{figure}

For $\freeprod$-arcs and simple closed $\freeprod$-curves, bigons and pseudo-bigons are related as follows: 

\begin{lemma}
\label{lem:pseudo-bigon_without_cp_bigon}
Let $(\gamma,b)$ and $(\gamma',b')$ be $\freeprod$-arcs that bound a pseudo-bigon such that the bigon bounded by the $\freeprod$-translates does not contain a cone point. Then $(\gamma,b)$ and $(\gamma',b')$ bound a bigon. 

The same holds for simple closed $\freeprod$-curves. 
\end{lemma}
\begin{proof}
Let us consider two $\freeprod$-translates, let us say $b$ and $b'$, that bound a bigon that does not contain a cone point. Then we may consider all the $\freeprod$-translates of $b$ and $b'$ that intersect the pseudo-bigon. 
Since 
$(\gamma,b)$ is a $\freeprod$-arc, its $\freeprod$-translates only intersect the boundary of the bigon in $b'$. This implies that each intersecting $\freeprod$-translate $\tilde{\gamma}(b)$ bounds a bigon with $b'$. Since no cone point lies inside the bigon, the bounded disk has the structure of a surface. This allows us, to pick a bigon that is not intersected by any other $\freeprod$-translate of $b$. Applying the same idea for $(\gamma',b')$, we obtain an innermost bigon that is not intersected by any other $\freeprod$-translates (see Figure \ref{fig:innermost_bigon}). This shows that $(\gamma,b)$ and $(\gamma',b')$ bound a bigon. 
\end{proof}

\begin{figure}[H]
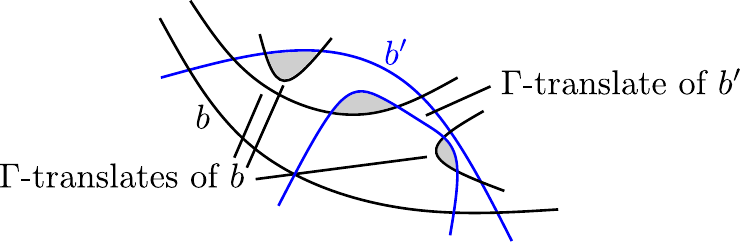
\caption{Pseudo-bigon without a cone point with innermost bigons shadowed in gray.}
\label{fig:innermost_bigon}
\end{figure}

\new{By Example \ref{ex:good_orb_free_prod},} the tree shaped surface $\S$ embeds into the complex plane. Formally, this map yields an atlas that endows $\S$ with a differentiable structure. In terms of this differentiable structure we define: 

\begin{definition}[Transversality of $\freeprod$-arcs and simple closed $\freeprod$-curves]
\label{def:trans_freeprod-arcs_curves}
Let $(\gamma,b)$ and $(\gamma',b')$ 
be two continuous $\freeprod$-arcs 
that intersect in a point $x$, i.e.\ $\tilde{\gamma}(b)$ and $\tilde{\gamma}'(b')$ 
intersect in $x$. The intersection is \textit{transverse} if the arcs $\tilde{\gamma}(b)$ and $\tilde{\gamma}'(b')$ 
intersect transversely as arcs 
in $\S$. If each intersection of $\freeprod$-translates $\tilde{\gamma}(b)$ and $\tilde{\gamma}'(b')$ is transverse, we call $(\gamma,b)$ and $(\gamma',b')$ \textit{transverse}. 

Likewise, we define transversality of simple closed $\freeprod$-curves. 

Two $\freeprod$-arcs $\beta$ and $\beta'$ are called \textit{transverse} if the unique equivalent continuous $\freeprod$-arcs are transverse. 
\end{definition}

Now we generalize the bigon criterion from Proposition~\ref{prop:bigon_crit_surf} to a criterion for $\freeprod$-arcs and simple closed $\freeprod$-curves. Since the proof for simple closed $\freeprod$-curves is the easier one, we begin with this. We want to prove: 

\begin{proposition}[Bigon criterion for simple closed $\freeprod$-curves]
\label{prop:bigon_crit_orb_scc}
Let $c,c'$ be two transverse simple closed $\freeprod$-curves in $\SG$. The simple closed $\freeprod$-curves are in minimal position if and only if they form no bigons. 
\end{proposition}

We start with an observation that has strong implications on pseudo-bigons. 

\begin{lemma}
\label{lem:freeprod-scc_intersecting_disks}
If two $\freeprod$-translates of a simple closed $\freeprod$-curve in $\SG$ 
bound intersecting disks, the $\freeprod$-translates coincide. 
\end{lemma}
\begin{proof}
Let $c$ be a simple closed $\freeprod$-curve and $\gamma(c)$ with $\gamma\neq1$ a $\freeprod$-translate such that $c$ and $\gamma(c)$ bound intersecting disks. \new{Due to the Jordan curve theorem \cite[Chapter~XVII, Theorem 5.4]{Dugundji1966},} each of the curves $c$ and $\gamma(c)$ divides the surface $\Sigma$ into two components, the inner and outer region. While the inner region is a bounded disk, the outer region is not bounded. In particular, $\gamma$ maps the disk $D_c$, bounded by $c$, to the disk $D_{\gamma(c)}$, bounded by $\gamma(c)$. 

\new{By Definition \ref{def:freeprod-scc},} $c$ and $\gamma(c)$ either coincide or their intersection is empty. In the first case we are done. The second case implies that $\gamma(c)$ lies inside the disk bounded by $c$ or likewise $c$ lies inside the disk bounded by $\gamma(c)$. We show \new{that} this cannot happen since $\freeprod$ acts isometrically \new{with respect to the metric defined in \eqref{eq:def_metric_iso_freeprod-action}.} 

Let us assume that $\gamma(c)$ lies in the disk bounded by $c$. This implies that the diameter of the bounded disks $D_c$ and $D_{\gamma(c)}$ satisfy 
\[
\sup_{x,y\in D_c}d(x,y)>\sup_{x,y\in D_{\gamma(c)}}d(x,y). 
\]
Since $D_c$ and $D_{\gamma(c)}$ are compact and the metric is continuous, both suprema are maxima. \new{Hence,} there exist $x_0,y_0\in D_c$ with $d(x_0,y_0)=\sup_{x,y\in D_c}d(x,y)$. Using the relation of the maxima and $\gamma(D_c)=D_{\gamma(c)}$, this in particular implies $d(x_0,y_0)\neq d(\gamma(x_0),\gamma(y_0))$. This contradicts the fact that $\freeprod$ acts isometrically on~$\S$. Consequently, $\gamma(c)$ and $c$ coincide and the Lemma holds. 
\end{proof}

\begin{corollary}
\label{cor:freeprod-scc_no_pseudo-bigon}
If two simple closed $\freeprod$-curves $c$ and $c'$ in $\SG$ 
bound a pseudo-bigon, it is a bigon. 
\end{corollary}
\begin{proof}
Let us assume that $c$ and $c'$ bound a pseudo-bigon such that there exists a $\freeprod$-translate, let us say $\gamma(c)$, that intersects the pseudo-bigon. Then $c$ and $\gamma(c)$ bound disks that intersect. By Lemma \ref{lem:freeprod-scc_intersecting_disks}, this implies that $c=\gamma(c)$. Hence, no $\freeprod$-translate intersects the pseudo-bigon, i.e.\ it is a bigon. 
\end{proof}

Further, Lemma \ref{lem:freeprod-scc_intersecting_disks} allows us to deduce the analog of Lemma \ref{lem:freeprod-arcs_curves_all_translates_in_min_pos} for simple closed $\freeprod$-curves. 

\begin{lemma}
\label{lem:freeprod-curves_all_translates_in_min_pos}
Let $c$ and $c'$ be simple closed $\freeprod$-curves in $\SG$. If $c$ is in minimal position with each $\freeprod$-translate of $c'$, the simple closed $\freeprod$-curves $c$ and $c'$ are in minimal position. 
\end{lemma}
\begin{proof}
\new{The proof of Lemma \ref{lem:freeprod-arcs_curves_all_translates_in_min_pos} was based on a counting argument. To apply the same argument for simple closed $\freeprod$-curves, we have to take into account that simple closed curves may have a non-trivial stabilizer. We address that distinguishing two cases.} 

Firstly, let $c$ bound no marked points or punctures. Since each $\freeprod$-translate $\gamma'(c')$ is in minimal position with $c$, Corollary \ref{cor:bigon-crit_Sigma} implies that $\gamma'(c')$ and $c$ form no bigon. Thus, no $\freeprod$-translate $\gamma'(c')$ intersects $c$. Hence, 
\[
0\leq i([c],[c'])\leq \vert c\cap\freeprod(c')\vert=0 
\] 
which implies $i([c],[c'])=\vert c\cap\freeprod(c')\vert=0$. 

If $c$ bounds a set of marked points and punctures $S\subseteq\freeprod(\{p_1,...,p_\TheStrand,r_1,...,r_\ThePct\})$, we observe that each curve $\tilde{c}\in[c]$ bounds the same set of marked points $S$. We prove: 

\begin{claim*}
An element $\gamma\in\freeprod$ preserves a simple closed $\freeprod$-curve $\tilde{c}\in[c]$ if and only if $\gamma$ preserves the set $S$. 
\end{claim*}

If $\gamma(\tilde{c})=\tilde{c}$, the element $\gamma$ preserves the bounded disk $D_{\tilde{c}}$. Due to the $\freeprod$-invariance of marked points and punctures, this implies that $\gamma(S)=S$. 

For the opposite implication, let $\gamma\in\freeprod$ such that $\gamma(S)=S$. We observe: 
\[
S=\gamma(S)\subseteq\gamma(D_{\tilde{c}})=D_{\gamma(\tilde{c})}, 
\]
i.e.\ $\gamma(\tilde{c})$ bounds $S$. By Lemma \ref{lem:freeprod-scc_intersecting_disks}, this implies $\gamma(\tilde{c})=\tilde{c}$. This proves the claim. 

Now let $\freeprod_c\subseteq\freeprod$ be a minimal subset such that $\freeprod_c(c)=\freeprod(c)$. The above claim implies that $\freeprod_c\subseteq\freeprod$ is also a minimal set that satisfies $\freeprod_c(\tilde{c})=\freeprod(\tilde{c})$ for each simple closed $\freeprod$-curve $\tilde{c}\in[c]$. Now the proof of the lemma follows verbatim as in Lemma~\ref{lem:freeprod-arcs_curves_all_translates_in_min_pos} if we replace the index set $\freeprod$ by $\freeprod_c$ in equation (\ref{eq:suff_cond_min_pos}). 
\end{proof}

\begin{proof}[Proof of Proposition \textup{\ref{prop:bigon_crit_orb_scc}}]
We use contraposition, i.e.\ we prove: $c$ and $c'$ are not in minimal position if and only if they form a bigon.  

If $c,c'$ form a bigon (in the sense of Definition \ref{def:bigon_freeprod-arc}), we can homotope through the bigon to reduce the number of intersections. Consequently, $c$ and $c'$ are not in minimal position. 

On the other hand, let us assume that the simple closed $\freeprod$-curves $c$ and $c'$ are not in minimal position. This assumption by Lemma \ref{lem:freeprod-curves_all_translates_in_min_pos} implies that there exists a $\freeprod$-translate $\tilde{\gamma}'(c')$ such that $c$ and $\tilde{\gamma}'(c')$ are not in minimal position. By Corollary~\ref{cor:bigon-crit_Sigma}, we obtain that $c$ and $\tilde{\gamma}'(c')$ form a bigon, i.e.\ the simple closed $\freeprod$-curves $c$ and $c'$ bound a pseudo-bigon. Thus, by Corollary \ref{cor:freeprod-scc_no_pseudo-bigon}, $c$ and $c'$ form a bigon. This proves: If $c$ and $c'$ are not in minimal position, then they form a bigon. The bigon criterion for simple closed $\freeprod$-curves follows. 
\end{proof}

\renewcommand{\TheDim}{k}
For $\freeprod$-arcs, in analogy to Proposition \ref{prop:bigon_crit_orb_scc} we want to prove: 

\begin{proposition}[Bigon criterion for $\freeprod$-arcs]
\label{prop:bigon_crit_orb}
Let $\beta$ and $\beta'$ be two transverse $\freeprod$-arcs in $\SG$. The $\freeprod$-arcs are in minimal position if and only if they form no bigons. 
\end{proposition}

The proof of the bigon criterion for $\freeprod$-arcs in $\SG$ is also based on Corollary~\ref{cor:bigon-crit_Sigma}. However, the proof is more complicated than for simple closed $\freeprod$-curves since we do not have a direct analog of Corollary \ref{cor:freeprod-scc_no_pseudo-bigon}. Instead, we will prove a necessary condition for $\freeprod$-paths that form a pseudo-bigon which contains a cone point: 

\newcommand{\Nbound}{J}
\newcommand{\NNbound}{K}

\begin{lemma}
\label{lem:freeprod-arcs_no_pseudo-bigon}
Let $(\gamma,b)$ and $(\gamma',b')$ be continuous $\freeprod$-paths. If $(\gamma,b)$ and $(\gamma',b')$ bound no bigons and bound a pseudo-bigon such that the bounded disk $D\subseteq\S$ 
contains a cone point, then at least one of the $\freeprod$-paths is not a $\freeprod$-arc. 
\end{lemma}

For the proof, we consider the parts of the boundary $\partial \FD(\ThePct)$ that are not contained in $\partial\S$, examples for these parts are the red arcs in Figure \ref{fig:bigon_pseudo_bigon}. These pieces are $\freeprod$-translates of arcs which we denote by $s_\nu$ for $1\leq\nu\leq\TheCone$. The arc $s_\nu$ connects the cone point $\cp_\nu$ to the boundary $\partial\S$. Let $s_\nu$ be parametrized by $I$ starting at the cone point $\cp_\nu$ and ending in $\partial\S$. 

\begin{proof}[Proof of Lemma \textup{\ref{lem:freeprod-arcs_no_pseudo-bigon}}]
Let us assume that $(\gamma,b)$ and $(\gamma',b')$ are \textit{$\freeprod$-arcs} that bound no bigons but a pseudo-bigon that contains the cone point $\rho(\cp_\mu)$ for some $\rho\in\freeprod$. Moreover, for $1\leq\Subdiv\leq\TheOrder_{\mu}$, let $\gamma_\Subdiv(s_{\mu})$ be the $\freeprod$-translates that connect the cone point $\rho(\cp_\mu)$ to the boundary. 

For brevity of notation, let us assume that the $\freeprod$-translates $b$ and $b'$ bound the pseudo-bigon. Furthermore, we may assume, for instance by piecewise linear approximation, that $b$ and $b'$ intersect finitely many times with the $\freeprod$-translates of~$\seg_\nu$ for $1\leq\nu\leq\TheCone$. Using that $(\gamma,b)$ is a $\freeprod$-arc, we may push off the arcs $\freeprod(s_{\mu})$ from $b$ by a $\freeprod$-equivariant Hatcher flow. See \cite{Hatcher1991} for further details on the classical Hatcher flow and Figure \ref{fig:Hatcher-flow} for an example of its $\freeprod$-equivariant version. This allows us to assume that $b$ is disjoint from~$\freeprod(s_{\mu})$, i.e.\ $b$ is contained in $\FD(\ThePct)$. In particular, $(\gamma,b)$ is a $\freeprod$-arc. 
\begin{figure}[H]
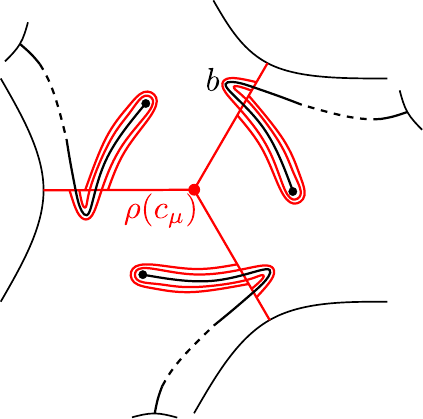
\caption{$\freeprod$-equivariant Hatcher flow.}
\label{fig:Hatcher-flow}
\end{figure}

Using that $b$ and $b'$ bound a disk $D$ that contains $\rho(\cp_{\mu})$, this implies that the arc~$b'$ intersects each of the $\freeprod$-translates $\gamma_\Subdiv(s_{\mu})$ such that the algebraic intersection number $i_{\pm}(b',\gamma_\Subdiv(s_{\mu}))$ has the same sign for each $1\leq\Subdiv\leq\TheOrder_{\mu}$ (see Figure \ref{fig:counterclockwise}). Without loss of generality, we assume that the algebraic intersection number is negative, i.e.\ $b'$ encircles $\gamma_\Subdiv(s_{\mu})$ counterclockwise. 

\begin{figure}[H]
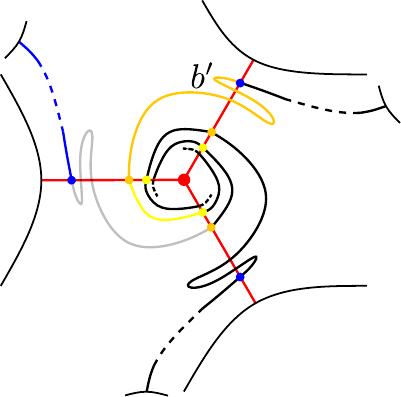
\caption{{\small A $\freeprod$-arc $(\gamma',b')$ encircling a cone point counterclockwise.}}
\label{fig:counterclockwise}
\end{figure}

Let $T_0$ be the maximal time such that $b'(T_0)\in b\cap\partial D$, i.e.\ $b'(T_0)$ is an endpoint of the pseudo-bigon, and consider the strictly decreasing sequence $(t_\NSubdiv)_{\NSubdiv\in N}$ of all times $t_\NSubdiv\in[0,T_0)$ such that $b'(t_\NSubdiv)\in\bigcup_{\Subdiv=1}^{\TheOrder_\mu}\gamma_\Subdiv(s_{\mu})$, where $N=\{0,...,\Nbound\}$ for some $\Nbound\in\NN$. The definition of $(t_\NSubdiv)_{\NSubdiv\in N}$ further induces a time $t_\NSubdiv^{(\mu)}$ and a $\freeprod$-translate $\gamma_{\Subdiv_\NSubdiv}(s_{\mu})$ such that $b'(t_\NSubdiv)=\gamma_{\Subdiv_\NSubdiv}(s_{\mu})(t_\NSubdiv^{(\mu)})$ for every $\NSubdiv\in N$. 

Now we may choose the following subsequence of elements $t_{\NSubdiv_\NNSubdiv}$: Let $t_{\NSubdiv_1}$ be the maximal time such that $b'$ intersects $\gamma_{\NSubdiv_\NNSubdiv}(s_{\mu})$ in counterclockwise direction. Moreover, let $t_{\NSubdiv_\NNSubdiv}$ be the maximal time $t_{\NSubdiv_\NNSubdiv}<t_{\NSubdiv_{\NNSubdiv-1}}$ such that 
\begin{itemize}
\item $b'$ intersects $\gamma_{\NSubdiv_\NNSubdiv}(s_{\mu})=\rho_\mu\gamma_{\NSubdiv_{\NNSubdiv-1}}(s_\mu)$ counterclockwise at time $t_{\NSubdiv_\NNSubdiv}$ with $\rho_\mu\in\freeprod$ inducing a counterclockwise rotation of angle $\frac{2\pi}{\TheOrder_\mu}$ around $\rho(\cp_\mu)$. 
\item there exists an $\varepsilon>0$ with $[t_{\NSubdiv_{\NNSubdiv-1}}+\varepsilon,t_{\NSubdiv_\NNSubdiv}-\varepsilon]$ containing all $t_\NSubdiv$ from $[t_{\NSubdiv_{\NNSubdiv-1}},t_{\NSubdiv_\NNSubdiv}]$ and $i_{\pm}(b'\vert_{[t_{\NSubdiv_{\NNSubdiv-1}}+\varepsilon,t_{\NSubdiv_\NNSubdiv}-\varepsilon]},\gamma_\Subdiv(s_{\mu}))=0$ for all $1\leq\Subdiv\leq\TheOrder_{\mu}$. 
\end{itemize}

This defines a subsequence $(t_{\NSubdiv_\NNSubdiv})_{\NNSubdiv\in N'}$ of $(t_\NSubdiv)_{\NSubdiv\in N}$. Since $b$ and $b'$ bound a bigon that contains $\rho(c_{\mu})$, this subsequence has at least $\TheOrder_{\mu}$ entries. 

\begin{claim*}
The sequence $(t_{\Strand_\NStrand}^{(\mu)})_{k\in N'}$ is strictly decreasing. 
\end{claim*}

We prove the claim by induction on $\NNSubdiv$. The idea is that the disjointness of the $\freeprod$-translates of $b'$ forces these $\freeprod$-translates to wind around $\rho(\cp_{\mu})$ as a snail shell (see Figure \ref{fig:counterclockwise}). 

If $t_{\NSubdiv_\NNSubdiv}^{(\mu)}=t_{\NSubdiv_{\NNSubdiv'}}^{(\mu)}$ for some $\NNSubdiv,\NNSubdiv'\in N',\NNSubdiv\neq\NNSubdiv'$, the $\freeprod$-translates of $b'$ intersect. This contradicts the assumption that $(\gamma',b')$ is a $\freeprod$-arc. 

\new{For $\NNSubdiv=1$,} the $\freeprod$-translate $\rho_{\mu}(b')\vert_{[t_{\NSubdiv_1},1]}$ (i.e.\ the blue arc in Figure \ref{fig:counterclockwise}) connects $\rho_{\mu}(b')(t_{\NSubdiv_1})=\gamma_{\Subdiv_{\NSubdiv_2}}(s_{\mu})(t_{\NSubdiv_1}^{(\mu)})$ to the boundary. Hence, $\rho_{\mu}(b')$ forces the arc $b'$ (the relevant part of $b'$ is drawn in orange in Figure \ref{fig:counterclockwise}) to intersect $\gamma_{\Subdiv_{\NSubdiv_2}}(s_{\mu})\vert_{[t_{\NSubdiv_1}^{(\mu)},1]}$ at time $t_{\NSubdiv_2}$ such that $t_{\NSubdiv_1}^{(\mu)}>t_{\NSubdiv_2}^{(\mu)}$. In Figure \ref{fig:counterclockwise} this is reflected by the fact that the orange intersection points lie closer to the cone point than the blue points. 

Now we may assume that $(t_{\NSubdiv_\NNSubdiv}^{(\mu)})_{\NNSubdiv\in N'}$ is strictly decreasing up to a fixed $\NNNSubdiv$. Recall that $\gamma_{\Subdiv_{\NSubdiv_{\NNNSubdiv+1}}}(s_{\mu})(t_{\NSubdiv_{\NNNSubdiv}})$ connects to the boundary via the $\freeprod$-translate $\rho_\mu(b')\vert_{\new{[t_{\NSubdiv_{\NNNSubdiv}},1]}}$. Thus, $\rho_\mu(b')\vert_{[t_{\NSubdiv_{\NNNSubdiv}},1]}$ enforces $b'\vert_{[0,t_{\NSubdiv_{\NNNSubdiv}}]}$ to intersect the next $\freeprod$-translate $\gamma_{\Subdiv_{\NSubdiv_{\NNNSubdiv+1}}}(s_\mu)$ at a point $\gamma_{\Subdiv_{\NSubdiv_{\NNNSubdiv+1}}}(s_\mu)(t_{\NSubdiv_{\NNSubdiv+1}}^{\new{(\mu)}})$ with $t_{\NSubdiv_\NNNSubdiv}^{(\mu)}>t_{\NSubdiv_{\NNNSubdiv+1}}^{(\mu)}$. For $\NNNSubdiv=2$, this is also visible in Figure \ref{fig:counterclockwise}: The gray piece forces the yellow arc to intersect the next $\freeprod$-translate of $s_\mu$ closer to the cone point than the gray piece. Thus, the yellow intersection points are closer to the cone point than the orange ones. By induction on~$\NNSubdiv$, this implies the claim. 
\\

Now let us consider the pseudo-bigon bounded by $b$ and $b'$ that contains the cone point $\rho(\cp_\mu)$. An example is shaded in gray in Figure \ref{fig:pseudo-bigon}. Moreover, let us consider the $\freeprod$-translate $\rho_\mu^{-1}(b')$. For this $\freeprod$-translate, there exists a time $t_{\NSubdiv_\NNSubdiv}$ such that $\rho_\mu^{-1}(b')(t_{\NSubdiv_\NNSubdiv})=\rho_\mu^{-1}\gamma_{\Subdiv_{\NSubdiv_\NNSubdiv}}(s_\mu)(t_{\NSubdiv_\NNSubdiv}^{(\mu)})$, i.e.\ the blue point in Figure \ref{fig:pseudo-bigon}. Since the sequence $(t_{\NSubdiv_\NNSubdiv}^{(\mu)})_{\NNSubdiv\in N'}$ is strictly decreasing, this intersection on $\rho_\mu^{-1}\gamma_{\Subdiv_{\NSubdiv_\NNSubdiv}}(s_\mu)=\gamma_{\Subdiv_{\NSubdiv_{\NNSubdiv-1}}}(s_\mu)$ lies closer to the cone point than the point $b'(t_{\NSubdiv_{\NNSubdiv-1}})=\gamma_{\Subdiv_{\NSubdiv_{\NNSubdiv-1}}}(s_\mu)(t_{\NSubdiv_{\NNSubdiv-1}}^{(\mu)})$, i.e.\ the yellow point in Figure~\ref{fig:pseudo-bigon}. This implies that the blue point $\rho_\mu^{-1}\gamma_{\Subdiv_{\NSubdiv_\NNSubdiv}}(s_\mu)(t_{\NSubdiv_\NNSubdiv}^{(\mu)})$ lies inside the pseudo-bigon. 

\begin{figure}[H]
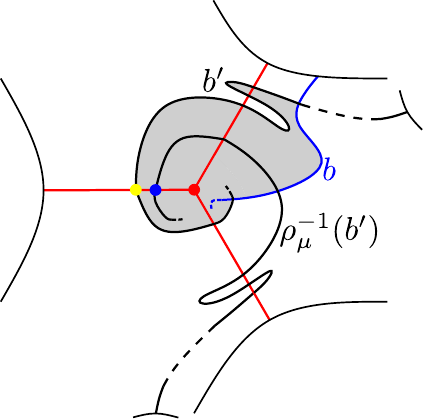
\caption{A pseudo-bigon that contains a cone-point.}
\label{fig:pseudo-bigon}
\end{figure}

Now we consider the piece $\rho_\mu^{-1}(b')([0,t_{\NSubdiv_\NNSubdiv}])$. 

If the piece $\rho_\mu^{-1}(b')([0,t_{\NSubdiv_\NNSubdiv}])$ is entirely contained in the disk bounded by the pseudo-bigon, this disk in particular contains the initial point of $\rho_\mu^{-1}(b')$. Since this point is a marked point, this contradicts the definition of a pseudo-bigon. 

If the piece $\rho_\mu^{-1}(b')([0,t_{\NSubdiv_\NNSubdiv}])$ intersects 
the arc $b'$, the $\freeprod$-translates of $b'$ are not disjoint. This contradicts the assumption that $(\gamma',b')$ is a $\freeprod$-arc. 

If the piece $\rho_\mu^{-1}(b')([0,t_{\NSubdiv_\NNSubdiv}])$ intersects 
the arc $b$, we obtain that the $\freeprod$-translate $\rho_\mu^{-1}(b')$ intersects the arc $b$ twice - when it enters and when it leaves the pseudo-bigon. Since the pseudo-bigon bounds a disk, this implies that $\rho_\mu^{-1}(b')$ and $b$ bound a bigon.

If this bigon does not contain a cone point, Lemma \ref{lem:pseudo-bigon_without_cp_bigon} implies that $(\gamma,b)$ and $(\gamma',b')$ form a bigon. This contradicts our assumption. 

If the bigon formed by $\rho_\mu^{-1}(b')$ and $b$ contains a cone point (see Figure \ref{fig:pseudo-bigon_add_crossing} shaded in dark gray), we obtain another pseudo-bigon bounded by $(\gamma,b)$ and $(\gamma',b')$ that contains a cone point. The existence of this pseudo-bigon that bounds a cone point requires an additional intersection of $\rho_\mu^{-1}(b')$ with $\langle\rho_\mu\rangle(s_\mu)$ that contributes to the sequence $(t_{\NSubdiv_\NNSubdiv})_{\NNSubdiv\in N'}$. In Figure \ref{fig:pseudo-bigon_add_crossing} the additional intersection is marked by the yellow point. 

Given the new pseudo-bigon, we may iteratively either apply one of the previous cases or construct new pseudo-bigons. As described above, each new pseudo-bigon requires an additional entry in the sequence $(t_{\NSubdiv_\NNSubdiv})_{\NNSubdiv\in N'}$. Since the index set $N'$ of the sequence $(t_{\NSubdiv_\NNSubdiv})_{\NNSubdiv\in N'}$ is finite, after finitely many steps no new pseudo-bigon can be obtained. Thus, we end in one of the previous cases and obtain a contradiction. Consequently, our assumption that both $(\gamma,b)$ and $(\gamma',b')$ are $\freeprod$-arcs was wrong and the lemma follows. 
\end{proof}

\begin{figure}[H]
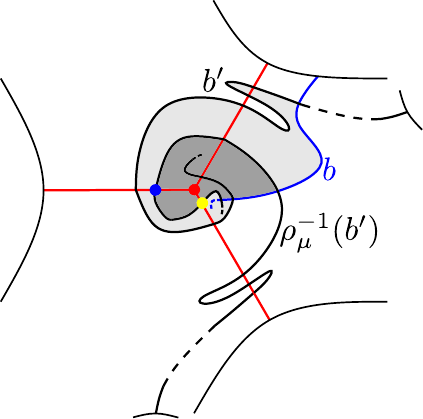
\caption{Another pseudo-bigon that contains a cone point.}
\label{fig:pseudo-bigon_add_crossing}
\end{figure}

\begin{proof}[Proof of Proposition \textup{\ref{prop:bigon_crit_orb}}]  
Using Remark \ref{rem:int_number}, the $\freeprod$-arcs $\beta$ and $\beta'$ are in minimal position if and only if the equivalent continuous arcs $(\gamma,b)$ and $(\gamma',b')$ are in minimal position. Therefore, it is enough to prove the claim for continuous $\freeprod$-arcs. We use contraposition, i.e.\ we prove: $(\gamma,b)$ and $(\gamma',b')$ are not in minimal position if and only if they form a bigon.  

If the $\freeprod$-arcs $(\gamma,b)$ and $(\gamma',b')$ form a bigon, homotoping through the bigon reduces the number of intersections. 

For the opposite implication, firstly, let $(\gamma,b)$ and $(\gamma',b')$ be $\freeprod$-arcs that are not in minimal position. Secondly, let us assume that $(\gamma,b)$ and $(\gamma',b')$ do not form any bigons. Using the first assumption, Lemma~\ref{lem:freeprod-arcs_curves_all_translates_in_min_pos} implies that there exists a $\freeprod$-translate $\tilde{\gamma}'(b')$ such that $b$ and $\tilde{\gamma}'(b')$ are not in minimal position. By Corollary~\ref{cor:bigon-crit_Sigma}, this implies that $b$ and $\tilde{\gamma}'(b')$ form a bigon as arcs in $\S$. On the basis of Definition~\ref{def:bigon_freeprod-arc}, this implies that $(\gamma,b)$ and $(\gamma',b')$ form a pseudo-bigon. Due to Lemma \ref{lem:freeprod-arcs_no_pseudo-bigon} and the second assumption, this pseudo-bigon does not contain a cone point. Hence, Lemma \ref{lem:pseudo-bigon_without_cp_bigon} implies that $(\gamma,b)$ and $(\gamma',b')$ also bound a bigon which contradicts our assumption. Consequently, each two $\freeprod$-arcs that are not in minimal position form a bigon. 
\end{proof}

As for arcs and simple closed curves, we want to use the bigon criterion 
to observe that each homotopy between two $\freeprod$-arcs or two simple closed $\freeprod$-curves can be realized by an ambient isotopy. Therefore, we prove:  

\begin{lemma}
\label{lem:compact_disj_translates_eps-nbhd}
Let $C$ be a compact set and 
let the group $\freeprod$ act on $\freeprod\times C$ by multiplication in the first factor. Moreover, let $\lambda:\freeprod\times C\rightarrow\S$ be a continuous $\freeprod$-equivariant map. If $\lambda(\{\gamma\}\times C)$ and $\lambda(\{\gamma'\}\times C)$ are disjoint or equal for all $\gamma,\gamma'\in\freeprod, \gamma\neq\gamma'$, then there exists $\varepsilon>0$ such that the $\varepsilon$-neighborhoods of $\lambda(\{\gamma\}\times C)$ for $\gamma\in\freeprod$ are pairwise either disjoint or equal. 
\end{lemma}
In particular, given a bigon bounded by the $\freeprod$-arcs $(\gamma,b)$ and $(\gamma',b')$, 
this lemma will allow us to construct disjoint $\varepsilon$-neighborhoods of the disks bounded by the $\freeprod$-translates of $b$ and $b'$. For this purpose, we choose the compact set $C$ from Lemma~\ref{lem:compact_disj_translates_eps-nbhd} to be a disk and endow $\freeprod$ with the discrete topology. Additionally, we consider a map $\lambda$ that for each $\gamma\in\freeprod$ embeds the disks $\{\gamma\}\times C$ into $\S$ such that each disk in the image of $\lambda$ is bounded by certain $\freeprod$-translates of $b$ and $b'$. 

To obtain disjoint $\varepsilon$-neighborhoods, we measure the distance of the $\freeprod$-translates of $\lambda(\{\gamma\}\times C)$. Using the tiling of $\Sigma$ into $\freeprod$-translates of $\FD$, we determine this distance considering finitely many $\freeprod$-translates and deduce that the distance is positive. 
\begin{proof}[Proof of Lemma \textup{\ref{lem:compact_disj_translates_eps-nbhd}}]
The compactness of $C$ implies that $\lambda(\{\gamma\}\times C)$ is compact for every $\gamma\in\freeprod$. Hence, each $\lambda(\{\gamma\}\times C)$ intersects only finitely many $\freeprod$-translates of the fundamental domain $\FD(\ThePct)$. Let $\{\gamma_1,...,\gamma_\TheSubdiv\}$ be the subset of $\freeprod$ such that $\lambda(\{1\}\times C)$ intersects $\gamma_i(\FD(\ThePct))$ or an adjacent $\freeprod$-translate. 
Further, the compactness of $\FD(\ThePct)$ implies that $\lambda(\{\gamma\}\times C)$ intersects $\gamma_\Subdiv(\FD(\ThePct))$ only for finitely many $\gamma\in\freeprod$. For each $1\leq\Subdiv\leq\TheSubdiv$, let $\gamma_1^{(\Subdiv)},...,\gamma_\TheNSubdiv^{(\Subdiv)}$ be the elements from $\freeprod$, such that $\lambda(\{\gamma_\NSubdiv^{(\Subdiv)}\}\times C)$ intersects $\gamma_\Subdiv(\FD(\ThePct))$. Thus, using the $\freeprod$-equivariance of $\lambda$, we obtain: 
\begin{align*}
& \inf\{d(\lambda(\{\gamma\}\times C),\lambda(\{\gamma'\}\times C))\mid\gamma,\gamma'\in\freeprod,\gamma\neq\gamma'\}
\\
= & \inf\{d(\lambda(\{\gamma\}\times C),\lambda(\{1\}\times C))\mid\gamma\in\freeprod,\gamma\neq1\}
\\
= & \min\{d(\lambda(\{\gamma_\NSubdiv^{(\Subdiv)}\}\times C),\lambda(\{1\}\times C))\mid1\leq\Subdiv\leq\TheSubdiv,1\leq\NSubdiv\leq\TheNSubdiv \text{ with } \gamma_\NSubdiv^{(\Subdiv)}\neq1\}. 
\end{align*}
The minimum is the distance of two disjoint compact sets. Thus, it has a positive value and we may define\[
\varepsilon:=\frac{1}{3}\min\{d(\lambda(\{\gamma_\NSubdiv^{(\Subdiv)}\}\times C),\lambda(\{1\}\times C))\mid1\leq\Subdiv\leq\TheSubdiv,1\leq\NSubdiv\leq\TheNSubdiv \text{ with } \gamma_\NSubdiv^{(\Subdiv)}\neq1\}. 
\]
\end{proof}

\begin{proposition}[Homotopic implies ambient isotopic]
\label{prop:homo_ind_amb_iso}
Let $\beta$ and $\beta'$ be homotopic $\freeprod$-arcs in $\SG$ 
with equivalent continuous 
$\freeprod$-arcs $(\gamma,b)$ and $(\gamma',b')$. There exists an ambient isotopy 
\[
I\rightarrow\HomeoOrb{\TheStrand}{\Sigma_\freeprod(\ThePct),\partial\S} 
\]
with $H_0=\id_{\S}$, $H_1(b(s))=b'(s)$ for all $s\in I$.  The same holds if $c$ and $c'$ are homotopic simple closed $\freeprod$-curves. 
\end{proposition}
\begin{proof}
Let $\beta$ and $\beta'$ be the $\freeprod$-arcs from above. By Lemma \ref{prop:bigon_crit_orb}, $\beta$ and $\beta'$ are either in minimal position or form a bigon. 

If $\beta$ and $\beta'$ form a bigon, 
we find a disk bounded by $\tilde{\gamma}(b)$ and $\tilde{\gamma}'(b')$ 
such that their $\freeprod$-translates do not intersect.  Applying Lemma \ref{lem:compact_disj_translates_eps-nbhd}, we obtain that the bounded disk has a compact $\varepsilon$-neighborhood such that its $\freeprod$-orbit is a disjoint union of its $\freeprod$-translates. Hence, the Alexander trick (see Example \ref{ex:Map(D_cycm)}) applies to each $\freeprod$-translate of the $\varepsilon$-neighborhood. We obtain an ambient isotopy connecting $\id_{\S}$ to a homeo-morphism that fixes every point outside the $\freeprod$-translates of the $\varepsilon$-neighborhood and removes the chosen bigon of $\tilde{\gamma}(b)$ and $\tilde{\gamma}'(b')$ and its $\freeprod$-translates. This reduces the number of intersections. Iterating this procedure, we may reduce the number of intersections of $\beta$ and $\beta'$ until it is minimal.

\begin{figure}[H]
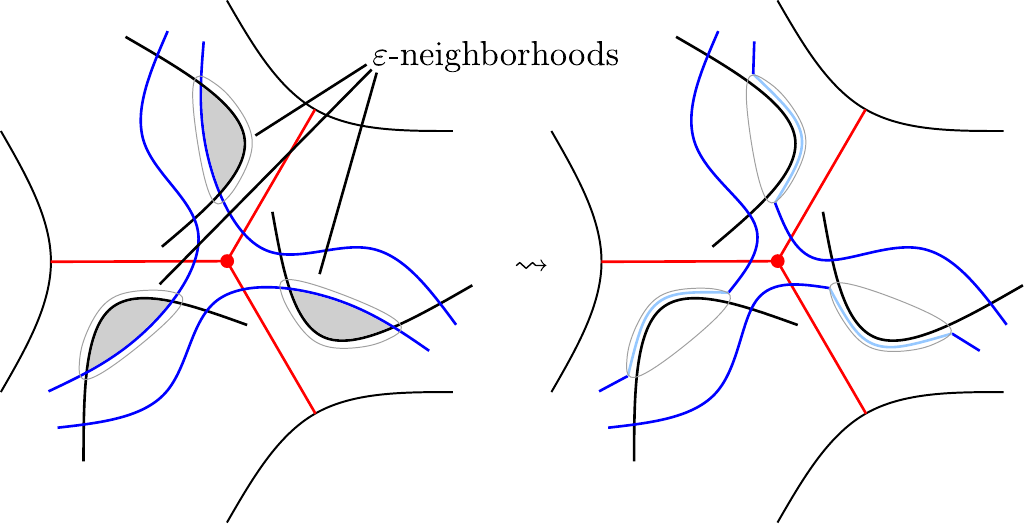
\caption{Removing the gray shaded bigon.}
\label{fig:remove_bigon}
\end{figure}

If $\beta$ and $\beta'$ are in minimal position, their continuous representatives $(\gamma,b)$ and $(\gamma',b')$ are also in minimal position. Since $\beta$ and $\beta'$ are homotopic, we have $\gamma=\gamma'$ and $b$ and $b'$ are homotopic. Consequently, $b$ and $b'$ share precisely their endpoints and bound a disk such that no $\freeprod$-translate of $b$ or $b'$ intersects. By Lemma \ref{lem:compact_disj_translates_eps-nbhd}, we once more find a compact $\varepsilon$-neighborhood with disjoint $\freeprod$-translates. Applying the Alexander trick to each $\freeprod$-translate, this allows us to construct an ambient isotopy which connects $\id_{\Sigma(\ThePct)}$ to a homeomorphism that maps $b$ to $b'$. 

Whenever a disk bounded by $\tilde{\gamma}(b)$ and $\tilde{\gamma}'(b')$ contains one or both of the endpoints, we may perturb the $\varepsilon$-neighborhood chosen above such that its boundary contains the endpoint. This allows us to assume that the ambient isotopy constructed above fixes the endpoints of $b$ and $b'$. 

Furthermore, the same proof applies for homotopic simple closed $\freeprod$-curves.  
\end{proof}

\begin{corollary}
If two $\freeprod$-arcs or two simple closed $\freeprod$-curves in $\SG$ are homotopic, they are isotopic. 
\end{corollary}

\subsection{The mapping class group $\MapOrb{}{\Sigma_\freeprod}$}
\label{subsec:Map_Sigma_freeprod}

So far we have determined the mapping class group of $D_{\cycm}$ in the case without marked points (see Example \ref{ex:Map(D_cycm)}) and with one orbit of marked points (see Example~\ref{ex:orb_mcg_one_pct_ZZ}). 
We want to close the section with a generalization to the orbifold $\Sigma_\freeprod$ for the case without marked points. This will serve as the base case for an inductive description of the subgroups $\PMapIdOrb{\TheStrand}{\Sigma_\freeprod}$ and $\MapIdOrb{\TheStrand}{\Sigma_\freeprod}$ for all $\TheStrand\in\NN$ in the next section. 

Therefore, we will determine how $\freeprod$-equivariant homeomorphisms manipulate the boundary of the fundamental domain $\FD$. A first step in this direction is to understand the action on cone points: 

\begin{lemma}
\label{lem:Homeo_fix_cp}
Every $\freeprod$-equivariant homeomorphism $H\in\HomeoOrb{}{\SG,\partial\S}$ fixes each cone point. 
\end{lemma}
\begin{proof}
\new{Let $H\in\HomeoOrb{}{\SG,\partial\S}$ be an arbitrary $\freeprod$-equivariant homeomorphism 
and let $1\neq\gamma\in\Stab_\freeprod(\cp_\nu)=\cyc{\TheOrder_\nu}$.} Then the $\freeprod$-equivariance implies 
\[
\gamma(H(\cp_\nu))=H(\gamma(\cp_\nu))=H(\cp_\nu). 
\]

Consequently, $\gamma\neq1$ is contained in both $\Stab_\freeprod(\cp_\nu)$ and $\Stab_\freeprod(H(\cp_\nu))$. Since the stabilizer of $H(\cp_\nu)$ is non-trivial, the point $H(\cp_\nu)$ is a cone point $\gamma'(\cp_{\mu})$. In this case 
\[
\Stab_\freeprod(H(\cp_\nu))=\Stab_\freeprod(\gamma'(\cp_{\mu}))=\gamma'\cyc{\TheOrder_{\mu}}\gamma'^{-1}. 
\]
This group intersects non-trivially with $\cyc{\TheOrder_\nu}$ if and only if $\nu=\mu$ and $\gamma'\in\cyc{\TheOrder_\nu}$, i.e.\ $H(\cp_\nu)=\cp_\nu$. Due to the $\freeprod$-equivariance of $H$, this also implies that $H$ fixes each $\freeprod$-translate of $\cp_\nu$. 
\end{proof}

\begin{corollary}
The mapping class groups $\MapOrb{\TheStrand}{\SG}$ and $\MapOrb{\TheStrand}{\SGpct}$ are isomorphic. 
\end{corollary}
\begin{proof}
Using Lemma \ref{lem:Homeo_fix_cp}, the map $[H]\mapsto[H\vert_{\Spct}]$ induces an isomorphism. 
\end{proof}

Further, Lemma \ref{prop:homo_ind_amb_iso} puts us in position to apply the Alexander trick to show: 

\begin{lemma}
\label{lem:orb_mcg_trivial}
The mapping class group $\MapOrb{}{\Sigma_\freeprod}=\MapOrb{\TheStrand}{\Sigma_\freeprod(\ThePct)}$ for $\ThePct=\TheStrand=0$ is trivial. 
\end{lemma}
\begin{proof}
Let $H$ be a homeomorphism that represents an element in $\MapOrb{}{\Sigma_\freeprod}$. Due to Lemma \ref{lem:Homeo_fix_cp}, $H$ fixes every cone point. For every $1\leq\nu\leq\TheCone$, let $S_\nu$ be a circle centered in $\cp_\nu$ of radius $\varepsilon_\nu>0$ such that the $\freeprod$-translates of these $\TheCone$ circles are either disjoint or coincide. Consequently, the same holds for the $\freeprod$-translates of $H(S_\nu)$ for $1\leq\nu\leq\TheCone$. For every $\nu$, let $D_\nu$ be the disk bounded by $S_\nu$. Using that $H$ fixes every cone point, this implies that the disk $H(D_\nu)$ bounded by $H(S_\nu)$ contains the cone point $\cp_\nu$. Further, the disjointness condition for $H(S_\nu)$ implies that $H(D_\nu)$ contains no further cone points. \new{Hence,} $S_\nu$ and $H(S_\nu)$ are simple closed $\freeprod$-curves that both bound exactly one cone point. \new{In both cases,} the bounded cone point is $\cp_\nu$. Since the tree-shaped surface $\Sigma$ contains no punctures or marked points, this implies that $S_\nu$ and $H(S_\nu)$ are homotopic. By Lemma \ref{prop:homo_ind_amb_iso}, this allows us to assume that $H$ fixes the simple closed $\freeprod$-curve $S_\nu$ and all its $\freeprod$-translates pointwise. 

Now we may consider the arcs $\seg_\nu$ for $1\leq\nu\leq\TheCone$ whose $\freeprod$-translates span the tessellation of $\Sigma$ into the $\freeprod$-translates of $\FD$. For every $\nu$, cutting along $S_\nu$ splits the arc $\seg_\nu$ into an inner part whose image is contained in $D_\nu$ and an outer part whose image is contained in $\Sigma\setminus\inter{D}_\nu$. Let $\seg_\nu'$ be the outer part. If we consider the endpoints of the arcs $\seg_\nu'$ for $1\leq\nu\leq\TheCone$ as marked points, we may consider these arcs as $\freeprod$-arcs. Since $H$ is $\freeprod$-equivariant, the arcs $H(\seg_\nu')$ can also be considered as $\freeprod$-arcs. Using that $H$ fixes the curves $S_\nu$, 
the arcs $\seg_\nu'$ and $H(\seg_\nu')$ share their endpoints. 

\new{By Proposition \ref{prop:bigon_crit_orb},} these $\freeprod$-arcs are either in minimal position or they form a bigon. If $\seg_\nu'$ and $H(\seg_\nu')$ form a bigon, Lemma \ref{prop:homo_ind_amb_iso} yields an ambient isotopy on $H$ that reduces the number of intersections of $H(\seg_\nu')$ and $\seg_\nu'$. Iterating this argument, we may adjust $H$ such that $H(\seg_\nu')$ and $\seg_\nu'$ are in minimal position, i.e.\ as $\freeprod$-arcs $H(\seg_\nu')$ and $\seg_\nu'$ do not bound any bigons. 

Lemma \ref{lem:freeprod-arcs_no_pseudo-bigon} implies that the $\freeprod$-arcs $H(\seg_\nu')$ and $\seg_\nu'$ also do not bound a pseudo-bigon that contains a cone point. This implies that the arcs $H(\seg_\nu')$ and $\seg_\nu'$ bound a disk without cone points that does not intersect with any of its $\freeprod$-translates. In particular, the $\freeprod$-arcs represented by $H(\seg_\nu')$ and $\seg_\nu'$ are homotopic. Thus, Lemma~\ref{prop:homo_ind_amb_iso} yields an ambient isotopy $H_t$ such that $H_0=H$ and $H_1=H'$ with $H'$ fixing $\seg_\nu'$ and $S_\nu$ for $1\leq\nu\leq\TheCone$ pointwise. 

The fact that $H'$ preserves the simple closed curves $S_\nu$  allows us to apply the Alexander trick on the disks $D_\nu$ inside $S_\nu$. Further, $H'$ preserves the disk $\FD\setminus\bigcup_{\nu=1}^\TheCone\inter{D}_\nu$, highlighted gray in Figure \ref{fig:Map_Sigma_freeprod_Alexander_trick}, fixing its boundary pointwise. This allows us another application of the Alexander trick on every $\freeprod$-translate of this disk. Hence, we obtain that $H'$ is homotopic to the identity. 
\end{proof}
\begin{figure}[H]
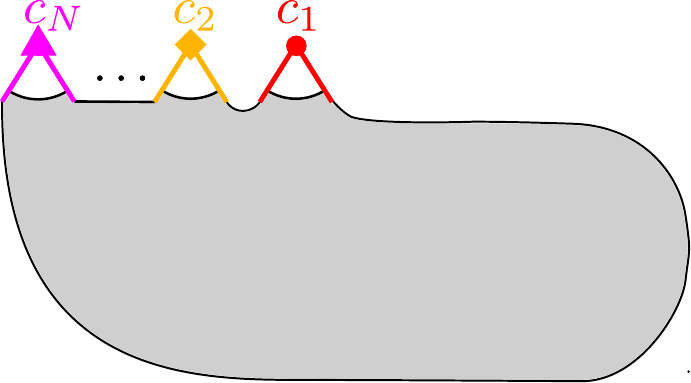
\caption{The disk $\FD\setminus\bigcup_{\nu=1}^\TheCone\inter{D}_\nu$.}
\label{fig:Map_Sigma_freeprod_Alexander_trick}
\end{figure}

\renewcommand{\Twist}{A}
\renewcommand{\TwistP}{B}
\renewcommand{\TwistC}{C}
\renewcommand{\twist}{a}
\renewcommand{\twistP}{b}
\renewcommand{\twistC}{c}

\section{Orbifold mapping class groups with marked points}
\label{sec:orb_mcg_marked_pts}

For surfaces with marked points, the mapping class group is determined by the Birman exact sequence. If we consider a disk $D$, 
it yields \cite[Theorem 9.1]{FarbMargalit2011}: 
\[
1\rightarrow\pi_1(\Conf_\TheStrand(D))\rightarrow\Map{\TheStrand}{D}\rightarrow\underbrace{\Map{}{D}}_{=1}\rightarrow1. 
\]
Based on the observation that $\Sk{\ThePct,\TheCone}/\freeprod$ is homeomorphic to $D(\ThePct,\TheCone)$, we identify a subgroup $\MapIdOrb{\TheStrand}{\Sigma_\freeprod(\ThePct)}$ of $\MapOrb{\TheStrand}{\Sigma_\freeprod(\ThePct)}$ that is isomorphic to a similar subgroup $\MapId{\TheStrand}{D(\ThePct,\TheCone)}$ of $\Map{\TheStrand}{D(\ThePct,\TheCone)}$ (see Proposition~\ref{prop:iso_Map_orb_disk}). This allows us to deduce a Birman exact sequence for $\MapIdOrb{\TheStrand}{\Sigma_\freeprod(\ThePct)}$ (see Theorem \ref{thm:Birman_es_orb}). Further, we deduce a short exact sequence of pure mapping class groups (see Corollary \ref{cor:pure_orb_mcg_ses}). In particular, we obtain presentations for $\PMapIdOrb{\TheStrand}{\Sigma_\freeprod^\TheStrand(\ThePct)}$ and $\MapIdOrb{\TheStrand}{\Sigma_\freeprod^\TheStrand(\ThePct)}$ (see Corollary \ref{cor:pres_PMap_free_prod} and Proposition \ref{prop:pres_map_kcp}). 

\subsection{Identification of subgroups of the orbifold mapping class group and the mapping class group of a disk}

Let $\SG$ be the orbifold with underlying surface $\Sigma$ punctured in $\freeprod(\{r_1,...,r_\ThePct\})$. Technically, as in Section \ref{sec:basics_mcg_orb}, it is often more convenient to consider the surface $\Sigma$ with marked points at $\freeprod(\{r_1,...,r_\ThePct\})$ instead of punctures. The orbifold mapping class group $\MapOrb{\TheStrand}{\Sigma_\freeprod(\ThePct)}$ has a homomorphism  
\[
\Forget_\TheStrand^{orb}:\MapOrb{\TheStrand}{\Sigma_\freeprod(\ThePct)}\rightarrow\MapOrb{}{\Sigma_\freeprod(\ThePct)} 
\]
by forgetting the marked points. In the following, we consider the kernel of $\Forget_\TheStrand^{orb}$. 
 
\begin{definition}
\label{def:Map_orb(ThePct)}
Let $\MapIdOrb{\TheStrand}{\Sigma_\freeprod(\ThePct)}$ denote the kernel of $\Forget_\TheStrand^{orb}$. 
This subgroup is induced by the subgroup
\[
\HomeoIdOrb{\TheStrand}{\Sigma_\freeprod(\ThePct),\partial\Sigma(\ThePct)}:=\{H\in\HomeoOrb{\TheStrand}{\Sigma_\freeprod(\ThePct),\partial\S}\mid H\sim\id_{\Sigma(\ThePct)}\} 
\]
of $\HomeoOrb{\TheStrand}{\Sigma_\freeprod(\ThePct)}$. Moreover, let $\PMapIdOrb{\TheStrand}{\Sigma_\freeprod(\ThePct)}:=\Forget_\TheStrand^{orb}\vert_{\PMapOrb{\TheStrand}{\Sigma_\freeprod(\ThePct)}}$. This subgroup is induced by the subgroup $\PHomeoIdOrb{\TheStrand}{\Sigma_\freeprod(\ThePct),\partial\Sigma(\ThePct)}$ that contains the pure homeomorphisms of $\HomeoIdOrb{\TheStrand}{\Sigma_\freeprod(\ThePct),\partial\Sigma(\ThePct)}$. 
\end{definition}

Additionally, we consider the disk $D(\ThePct,\TheCone)$ with $\ThePct+\TheCone$ distinct punctures at positions $\bar{r}_1,...,\bar{r}_\ThePct$ and $\bar{\cp}_1,...,\bar{\cp}_\TheCone$. If we endow $D(\ThePct,\TheCone)$ with $\TheStrand$ distinct marked points at $\bar{p}_1,...,\bar{p}_\TheStrand$, there exists an analogous forgetful map 
\[
\Forget_\TheStrand:\Map{\TheStrand}{D(\ThePct,\TheCone)}\rightarrow\Map{}{D(\ThePct,\TheCone)}. 
\] 

\begin{definition}
\label{def:Map_disk(ThePct)}
Let $\MapId{\TheStrand}{D(\ThePct,\TheCone)}$ denote the kernel of $\Forget_\TheStrand$. 
This subgroup is induced by the subgroup
\[
\HomeoId{\TheStrand}{D(\ThePct,\TheCone),\partial D(\ThePct,\TheCone)}:=\{H\in\Homeo{\TheStrand}{D(\ThePct,\TheCone),\partial D(\ThePct,\TheCone)}\mid H\sim\id_{D(\ThePct,\TheCone)}\} 
\]
of $\Homeo{\TheStrand}{D(\ThePct,\TheCone)}$. Moreover, let $\PMapId{\TheStrand}{D(\ThePct,\TheCone)}:=\Forget_\TheStrand\vert_{\PMap{\TheStrand}{D(\ThePct,\TheCone)}}$. This subgroup is induced by the subgroup $\PHomeoId{\TheStrand}{D(\ThePct,\TheCone),\partial D(\ThePct,\TheCone)}$ that contains the pure homeomorphisms of $\HomeoId{\TheStrand}{D(\ThePct,\TheCone),\partial D(\ThePct,\TheCone)}$. 
\end{definition}

In contrast to $\MapIdOrb{\TheStrand}{\Sigma_\freeprod(\ThePct)}$, the subgroup $\Map{\TheStrand}{D(\ThePct,\TheCone)}$ satisfies the additional condition that homeomorphisms and ambient isotopies 
fix points that correspond to cone points. Based on Lemma \ref{lem:Homeo_fix_cp} the cone points are automatically fixed by $\freeprod$-equivariant homeomorphisms of $\S$. This will allow us to prove: 

\begin{proposition}
\label{prop:iso_Map_orb_disk}
The group $\MapId{\TheStrand}{D(\ThePct,\TheCone)}$ is isomorphic to $\MapIdOrb{\TheStrand}{\Sigma_\freeprod(\ThePct)}$ and the isomorphism restricts to an isomorphism between 
$\PMapId{\TheStrand}{D(\ThePct,\TheCone)}$ and $\PMapIdOrb{\TheStrand}{\Sigma_\freeprod(\ThePct)}$. 
\end{proposition}

The proof is divided into the following four steps: 
\begin{itemize}
\item In Lemma \ref{lem:pi_Homeo_cont_homo}, we construct a continuous homomorphism 
\[
\pi:\HomeoOrb{\TheStrand}{\Sigma_\freeprod(\ThePct),\partial\Sigma(\ThePct)}
\rightarrow\Homeo{\TheStrand}{D(\ThePct,\TheCone),\partial D(\ThePct,\TheCone)}
\]
using the projection from $\Spct$ to $\Spct/\freeprod\cong D(\ThePct,\TheCone)$. 
\item Due to the continuity, the homomorphism $\pi$ induces a homomorphism 
\[
\piMap:\MapIdOrb{\TheStrand}{\Sigma_\freeprod(\ThePct)}\rightarrow\MapId{\TheStrand}{D(\ThePct,\TheCone)}, 
\]
see Lemma \ref{lem:pi_Map_homo}.  
\item In the opposite direction, we construct a homomorphism 
\[
\varphi:\HomeoId{\TheStrand}{D(\ThePct,\TheCone),\partial D(\ThePct,\TheCone)}\rightarrow\HomeoOrb{\TheStrand}{\Sigma_\freeprod(\ThePct),\partial\Sigma(\ThePct)} 
\]
in Lemma \ref{lem:varphi_Homeo_homo}. This requires to lift a self-homeomorphism of $D(\ThePct,\TheCone)\cong\Spct/\freeprod$ that satisfies the conditions from Definition \ref{def:Map_disk(ThePct)} to a self-homeo-morphism of $\Spct$. In contrast to Lemma \ref{lem:pi_Homeo_cont_homo}, we will not prove the continuity of $\varphi$ in this case. 
\item However, the homomorphism $\varphi$ induces a homomorphism 
\[
\varphiMap:\MapId{\TheStrand}{D(\ThePct,\TheCone)}\rightarrow\MapIdOrb{\TheStrand}{\Sigma_\freeprod(\ThePct)}, 
\]
see Lemma \ref{lem:varphi_Map_homo}. Since $\varphi$ is not necessarily continuous, in comparison to Lemma \ref{lem:pi_Map_homo} the well-definedness of $\varphiMap$ requires an additional argument. 
\end{itemize}
To deduce that $\MapId{\TheStrand}{D(\ThePct,\TheCone)}$ and $\MapIdOrb{\TheStrand}{\Sigma_\freeprod(\ThePct)}$ are isomorphic, we will finally check that the homomorphisms $\piMap$ and $\varphiMap$ are inverse to each other. 

Given $H\in\HomeoOrb{\TheStrand}{\Sigma_\freeprod(\ThePct)}$, we begin with the definition of an induced map 
\[
\bar{H}:D_\TheStrand(\ThePct,\TheCone)\rightarrow D_\TheStrand(\ThePct,\TheCone). 
\]
By Lemma \ref{lem:Homeo_fix_cp}, $H$ fixes all cone points. This allows us to consider $H$ as a homeomorphism of the surface $\Spct$. 
Using the $\freeprod$-equivariance of $H$, we obtain a well-defined map 
\[
\bar{H}:\Spct/\freeprod\rightarrow\Spct/\freeprod,\freeprod(x)\mapsto \freeprod(H(x)). 
\]
Since $\Spct/\freeprod$ is homeomorphic to $D(\ThePct,\TheCone)$, this defines a self-map of marked disks. 
Further, we observe that the disk $D(\ThePct,\TheCone)$ is homeomorphic to $\FD(\ThePct,\TheCone)/\sim_\freeprod$ with $\sim_\freeprod$ identifying boundary points from the same $\freeprod$-orbit. Hence, we may consider $\bar{H}$ as a map 
\[
\FD(\ThePct,\TheCone)\rightarrow \FD(\ThePct,\TheCone),\bar{x}\mapsto\bar{H}(\bar{x})
\]
that coincides on the identified boundary points, fixes the set of marked points $\{p_1,...,p_\TheStrand\}$ and restricts to the identity on $\partial\Sigma(\ThePct,\TheCone)\cap \FD(\ThePct,\TheCone)$. 

For the proof of Lemma \ref{lem:pi_Homeo_cont_homo}, we start with an observation about the fundamental domain $\FD$. 

\begin{lemma}
\label{lem:H(F)_fund_domain}
If $H\in\HomeoOrb{}{\Sigma_\freeprod,\partial\Sigma}$, then $H(\FD)$ is also a fundamental domain of the $\freeprod$-action on $\Sigma$. 
\end{lemma}
\begin{proof}
Let $y$ be an arbitrary point in $\Sigma$. The surjectivity of $H$ implies that there exists a $y'\in\Sigma$ such that $y=H(y')$. Since $\FD$ is a fundamental domain, there exists a point $x'\in \FD$ and $\gamma\in\freeprod$ such that $y'=\gamma(x')$. Using that $H$ is $\freeprod$-equivariant, we obtain $y=\gamma(H(x'))$, i.e.\ $\Sigma=\bigcup_{\gamma\in\freeprod}\gamma(H(\FD))$. 

Moreover, if $y$ is contained in $\gamma(H(\FD))\cap\gamma'(H(\FD))$ for $\gamma\neq\gamma'$, this implies $y=\gamma(H(x'))=\gamma'(H(x''))$ for $x',x''\in \FD$. This is equivalent to $H(\gamma(x'))=H(\gamma'(x''))$. Since $H$ is injective, that is $\gamma(x')=\gamma'(x'')$. Using that $\FD$ is a fundamental domain, this implies that $x'$ and $x''$ are contained in $\partial \FD$. The homeomorphism $H$ maps $\partial\FD$ to $\partial H(\FD)$, i.e.\ $H(x'),H(x'')$ are contained in $\partial H(\FD)$. This proves that $H(\FD)$ is a fundamental domain. 
\end{proof}

\begin{lemma}
\label{lem:pi_Homeo_cont_homo}
The map 
\[
\pi:\HomeoOrb{\TheStrand}{\Sigma_\freeprod(\ThePct,\TheCone),\partial\Sigma(\ThePct,\TheCone)}
\rightarrow\Homeo{\TheStrand}{D(\ThePct,\TheCone),\partial D(\ThePct,\TheCone)}, \quad H\mapsto\bar{H} 
\]
is a continuous homomorphism. 
\end{lemma}
\begin{proof}
We divide the proof into three steps. 

\begin{intermediate}[Step $1$. The map $\pi$ is a homomorphism]
Let $H,H'\in\HomeoOrb{\TheStrand}{\Sigma_\freeprod(\ThePct,\TheCone),\partial\Sigma(\ThePct,\TheCone)}$. Then the definition of $\bar{H}$ and $\bar{H}'$ implies
\[
\bar{H}'\circ\bar{H}(\freeprod(x))=\bar{H}'(\freeprod(H(x)))=\freeprod(H'\circ H(x))=\overline{H'\circ H}(\freeprod(x)). 
\]
Hence, $\pi$ is a homomorphism. 
\end{intermediate}

\begin{intermediate}[\hypertarget{lem:pi_Homeo_cont_homo_step2}{Step $2$.} The map $\bar{H}$ is a homeomorphism]
Since the map $\varphi$ is a homomorphism, the inverse of $\bar{H}$ is given by the image of $H^{-1}$ under $\pi$. As a consequence, it suffices to show that $\bar{H}$ is continuous for every $H$ in $\HomeoOrb{\TheStrand}{\Sigma_\freeprod(\ThePct),\partial\Sigma(\ThePct)}$ to obtain that $\bar{H}$ is a homeomorphism. 

Since $\FD$ is a fundamental domain, $H\vert_{\FD}$ determines the induced map $\bar{H}$. By Lemma \ref{lem:H(F)_fund_domain}, $H^{-1}(\FD)$ is also a fundamental domain. This allows us to cover $\FD$ with the $\freeprod$-translates of $H^{-1}(\FD)$. In this way, we obtain pieces 
\[
H:\FD\cap\gamma(H^{-1}(\FD))\rightarrow\gamma(\FD). 
\]
Since both, domain and codomain, are contained in a single $\freeprod$-translate of $\FD$, we may identify the above map with 
\[
\FD\cap\gamma(H^{-1}(\FD))\rightarrow \FD, x\mapsto \gamma^{-1}(H(x)). 
\]
This map coincides with $\bar{H}\vert_{\FD\cap\gamma(H^{-1}(\FD))}$. Using that $H$ is continuous, this implies that $\bar{H}$ is continuous on every piece $\FD\cap\gamma(H^{-1}(\FD))$. 

Next, we check that the continuous pieces match together to a continuous function $\bar{H}$. Therefore, we recall that $\FD$ and $H^{-1}(\FD)$ are compact disks. This implies that $\FD$ intersects only finitely many $\freeprod$-translates of $H^{-1}(\FD)$. Hence, we have decomposed $\bar{H}$ into finitely many continuous pieces. Let $\FD\cap\gamma(H^{-1}(\FD))$ and $\FD\cap\gamma'(H^{-1}(\FD))$ be pieces such that they or suitable $\freeprod$-translates intersect. If we consider intersecting translates, they share a set contained in the boundary of both (see Figure \ref{fig:continuity_ind_map_H} for an example). 

On the left of Figure \ref{fig:continuity_ind_map_H}, the fundamental domain $\FD$ is shaded in different colors which describe the above decomposition into pieces. An exemplary intersection of two pieces is colored in yellow. \new{For the blue segments,} there exist $\freeprod$-translates of the adjacent pieces that intersect. On the right of Figure \ref{fig:continuity_ind_map_H}, it is shown how the corresponding pieces cover $D_\TheStrand(\ThePct,\TheCone)$. 
\begin{figure}[H]
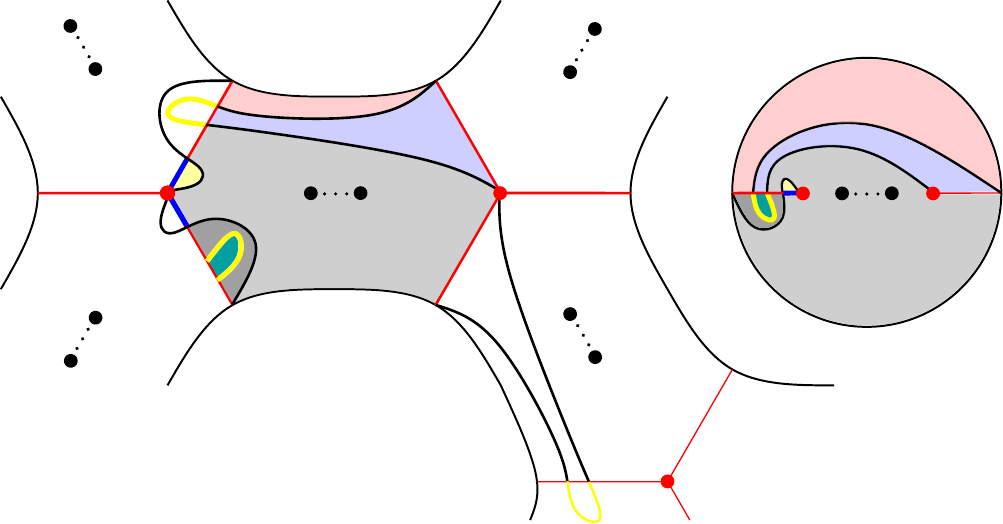
\caption{Decompositions of $H$ and $\bar{H}$.}
\label{fig:continuity_ind_map_H}
\end{figure} 

Let us consider an open neighborhood in $\Spct$ of the intersecting set of the two pieces. Using that $H$ is continuous on this neighborhood, the $\freeprod$-equivariance of $H$ implies that the maps $H\vert_{\FD\cap\gamma(H^{-1}(\FD))}$ and $H\vert_{\FD\cap\gamma'(H^{-1}(\FD))}$ coincide on the set that corresponds to the intersecting set. Since $H$ identifies with $\bar{H}$ on these pieces, $\bar{H}\vert_{\FD\cap\gamma(H^{-1}(\FD))}$ and $\bar{H}\vert_{\FD\cap\gamma'(H^{-1}(\FD))}$ coincide on their intersecting set. Thus, by \cite[Chapter III, Theorem 9.4]{Dugundji1966}, the finitely many continuous pieces of $\bar{H}$ induce a continuous map on $D_\TheStrand(\ThePct,\TheCone)$. 
\end{intermediate}

\begin{intermediate}[Step $3$. Continuity of $\pi$]
Finally, we need to check that $\pi$ is continuous, i.e.\ for every homeomorphism $H\in\HomeoOrb{\TheStrand}{\SG,\partial\S}$ and every neighborhood $V$ of $\pi(H)=\bar{H}$, there exists a neighborhood $V'$ of $H$ such that $\pi(V')\subseteq V$. Given a compact set $K$ and an open set $U$ in $D(\ThePct,\TheCone)$, let $V(K,U)$ denote the set 
\[
\{\phi\in\Homeo{\TheStrand}{D(\ThePct,\TheCone),\partial D(\ThePct,\TheCone)}\mid\phi(K)\subseteq U\}. 
\]
By definition of the compact-open topology on $\Homeo{\TheStrand}{D(\ThePct,\TheCone),\partial D(\ThePct,\TheCone)}$, the open set $V$ contains a subset $V(K,U)\ni\pi(H)=\bar{H}$. 
Now we have to find an appropriate open subset $V'\subseteq\HomeoIdOrb{\TheStrand}{\Sigma_\freeprod(\ThePct),\partial\Sigma(\ThePct)}$ such that $V'$ maps to $V(K,U)$. We will choose $V'=V'(K_\Sigma,U_\Sigma)$ for some $K_\Sigma,U_\Sigma\subseteq\Sigma(\ThePct)$ such that $K_\Sigma$ is compact and $U_\Sigma$ is open. 

Since $D(\ThePct,\TheCone)$ is homeomorphic to $\FD(\ThePct,\TheCone)/\sim_\freeprod$ and $\FD(\ThePct,\TheCone)\subseteq\Spct$, we may consider $K$ and $U$ as subsets in $\Sigma(\ThePct)$. We choose $K_\Sigma=K$ and $U_\Sigma=\freeprod(U)$. 

$K_\Sigma$ is compact, since it is a compact subset of the compact subset $\FD(\ThePct,\TheCone)$. 

\new{For the openness of $U_\Sigma$,} we observe: If $x\in U_\Sigma$ is contained in the interior of a $\freeprod$-translate of $\FD(\ThePct,\TheCone)$, we may use that $U$ is open in $\FD(\ThePct,\TheCone)$. This yields an $\varepsilon$-ball around $x$ that is contained in the interior of the $\freeprod$-translate of $\FD(\ThePct,\TheCone)$ and $U$. 
In particular, this $\varepsilon$-ball is contained in $U_\Sigma$. 
If $x\in U_\Sigma$ is contained in a $\freeprod$-translate of $\partial \FD(\ThePct,\TheCone)$, we may use the openness of $U$ in $D(\ThePct,\TheCone)$ to find a surrounding $\varepsilon$-ball contained in $U\subseteq D(\ThePct,\TheCone)$ for the point $\bar{x}$ which corresponds to $x$. If we identify this $\varepsilon$-ball with a subset in $\FD(\ThePct,\TheCone)$, it divides into two components. Shifting both of these halves by suitable group elements, yields the $\varepsilon$-ball centered at $x$. Since $U_\Sigma$ covers the whole $\freeprod$-orbit of $U\subseteq \FD(\ThePct,\TheCone)$, this $\varepsilon$-ball is contained in $U_\Sigma$. Thus, $U_\Sigma$ is open. 

For these sets $K_\Sigma$ and $U_\Sigma$, we observe
\[
H(K_\Sigma)=H(K)=H\circ\bar{H}^{-1}\circ\bar{H}(K)\subseteq H\circ\bar{H}^{-1}(U)\subseteq U_\Sigma. 
\]
In the last inclusion, we used that $H$ and $\bar{H}$ differ at most by a $\freeprod$-translation. Now $H(K_\Sigma)\subseteq U_\Sigma$ implies that $H\in V'(K_\Sigma,U_\Sigma)$ and further, for every $H'$ in $\HomeoOrb{\TheStrand}{\Sigma_\freeprod(\ThePct),\partial\Sigma(\ThePct)}$, the condition $H'(K_\Sigma)\subseteq U_\Sigma$ implies $\bar{H}'(K)\subseteq U$, i.e.\ $\pi(V'(K_\Sigma,U_\Sigma))\subseteq V(K,U)$. 
\end{intermediate}
\end{proof}

If we restrict $\pi$ to the subgroup $\HomeoIdOrb{\TheStrand}{\Sigma_\freeprod(\ThePct),\partial\Sigma(\ThePct)}$, we obtain a continuous map $\HomeoIdOrb{\TheStrand}{\Sigma_\freeprod(\ThePct),\partial\Sigma(\ThePct)}\rightarrow\Homeo{\TheStrand}{D(\ThePct,\TheCone),\partial D(\ThePct,\TheCone)}$. 

\begin{lemma}
\label{lem:pi_Map_homo}
The restricted map induces a homomorphism 
\[
\piMap:\MapIdOrb{\TheStrand}{\Sigma_\freeprod(\ThePct)}\rightarrow\MapId{\TheStrand}{D(\ThePct,\TheCone)}. 
\]
\end{lemma}
\begin{proof}
It remains to check that the induced map 
\[
\piMap:\MapIdOrb{\TheStrand}{\Sigma_\freeprod(\ThePct)}\rightarrow\Map{\TheStrand}{D(\ThePct,\TheCone)},[H]\mapsto[\bar{H}] 
\]
is well-defined and that its image is contained in $\MapId{\TheStrand}{D(\ThePct,\TheCone)}$. 

For $H,H'\in\HomeoIdOrb{\TheStrand}{\Sigma_\freeprod(\ThePct),\partial\Sigma(\ThePct)}$ with $H\sim_\TheStrand H'$, there exists an ambient isotopy $\phi:I\rightarrow\HomeoOrb{\TheStrand}{\Sigma_\freeprod(\ThePct),\partial\Sigma(\ThePct)}$ connecting $H$ and $H'$. Since $\pi$ is continuous, by Lemma \ref{lem:pi_Homeo_cont_homo}, we obtain an ambient isotopy $\bar{\phi}:=\pi\circ\phi$ connecting $\bar{H}$ and~$\bar{H}'$. Hence, the induced map $\piMap$ is well-defined. 

Moreover, let $H\in\HomeoIdOrb{\TheStrand}{\Sigma_\freeprod(\ThePct),\partial\Sigma(\ThePct)}$. If we forget the marked points $\freeprod(\{p_1,...,p_\TheStrand\})$, 
by Definition \ref{def:Map_orb(ThePct)}, we obtain an ambient isotopy 
\[
\psi:I\rightarrow\HomeoOrb{}{\Sigma_\freeprod(\ThePct),\partial\Sigma(\ThePct)} 
\]
from $H$ to $\id_{\Sigma(\ThePct)}$. By Lemma \ref{lem:Homeo_fix_cp}, $\psi_t$ fixes the cone points for every $t\in I$. Consequently, the induced ambient isotopy $\bar{\psi}$ connects $\bar{H}$ and $\id_{D(\ThePct,\TheCone)}$ relative 
$\bar{c}_1,...,\bar{c}_\TheCone$ and $\bar{r}_1,...,\bar{r}_\ThePct$. Hence, $\bar{H}$ is contained in the subgroup $\MapId{\TheStrand}{D(\ThePct,\TheCone)}$. 
\end{proof}

Next, we describe an inverse homomorphism for $\piMap$. The first step is to construct a self-homeomorphism of $\SG$ for each $\bar{H}\in\Homeo{\TheStrand}{D(\ThePct,\TheCone),\partial D(\ThePct,\TheCone)}$: 

\begin{construction}[A self-homeomorphism of $\Sigma(\ThePct)$] 
By Definition \ref{def:Map_disk(ThePct)}, there exists an ambient isotopy 
\[
\bar{\phi}:I\rightarrow\Homeo{}{D(\ThePct,\TheCone),\partial D(\ThePct,\TheCone)} 
\]
that connects $\id_{D(\ThePct,\TheCone)}$ to $\bar{H}$. For each point $\bar{x}\in D$ that is not a marked point or cone point (these points are fixed anyway), we may follow its trace $\bar{\phi}_t(\bar{x})$ and detect its algebraic intersections with the segments $\bar{s}_\nu$ for $1\leq\nu\leq\TheCone$. This yields a finite sequence $((\nu_1,\varepsilon_1),...,(\nu_\TheNSubdiv,\varepsilon_\TheNSubdiv))$ with $1\leq\nu_\NSubdiv\leq\TheCone$ and $\varepsilon_\NSubdiv=\pm1$ which induces a group element $\gamma_{\bar{H},\bar{x}}:=\gamma_{\nu_1}^{\varepsilon_1}...\gamma_{\nu_\TheNSubdiv}^{\varepsilon_\TheNSubdiv}$. Now we identify $\bar{x}$ with a point in $\FD(\ThePct,\TheCone)$ and define 
\[
\bar{H}_\Sigma:\gamma(\bar{x})\mapsto\tilde{\gamma}\gamma_{\bar{H},\bar{x}}(\bar{H}(\bar{x})). 
\]
The element $\tilde{\gamma}$ equals $\gamma$ if $\gamma(\bar{x})$ is an interior element in $\gamma(\FD(\ThePct,\TheCone))$ or $\gamma(\bar{x})$ is contained in $\gamma(s_\nu)$ for some $1\leq\nu\leq\TheCone$. If $\gamma(\bar{x})$ is contained in $\gamma'(s_\nu)$, we set $\tilde{\gamma}=\gamma'$.  

To verify that $\bar{H}_\Sigma$ is well-defined, we have to check that $\gamma_{\bar{H},\bar{x}}$ does not depend on the choice of the ambient isotopy $\bar{\phi}$. 

Let $\bar{\psi}:I\rightarrow\Homeo{}{D(\ThePct,\TheCone),\partial D(\ThePct,\TheCone)}$ be another ambient isotopy that connects $\id_{D(\ThePct,\TheCone)}$ and $\bar{H}$. Then $\bar{\psi}^{-1}\circ\bar{\phi}$ yields an ambient isotopy connecting $\id_{D(\ThePct,\TheCone)}$ through $\bar{H}$ to $\id_{D(\ThePct,\TheCone)}$. In particular, $(\bar{\psi}^{-1}\circ\bar{\phi})_t(\bar{x})$ represents an element in $\pi_1\left(D(\ThePct,\TheCone),\bar{x}\right)$. Via point-pushing, $\pi_1\left(D(\ThePct,\TheCone),\bar{x}\right)$ embeds into $\Map{1}{D(\ThePct,\TheCone)}$ and the element $[(\bar{\psi}^{-1}\circ\bar{\phi})_t(\bar{x})]$ maps onto $[\id_{D(\ThePct,\TheCone)}]$. Hence, $(\bar{\psi}^{-1}\circ\bar{\phi})_t(\bar{x})$ represents the trivial element in $\pi_1\left(D(\ThePct,\TheCone),\bar{x}\right)$. Since $\pi_1\left(D(\ThePct,\TheCone),\bar{x}\right)=\freegrp{\ThePct+\TheCone}$, this implies that the intersection patterns of $\bar{\phi}_t(\bar{x})$ and $\bar{\psi}_t(\bar{x})$ with the segments $\bar{s}_\nu$ coincide up to insertion of subsequences $(\mu,\varepsilon),(\mu,-\varepsilon)$ with $1\leq\mu\leq\TheCone$ and $\varepsilon=\pm1$. Hence, $\gamma_{\bar{H},\bar{x}}$ does not depend on the choice of the ambient isotopy. 

Moreover, we need to show: If $\freeprod$-translates of $\FD(\ThePct,\TheCone)$ intersect in their boundaries, the definition of $\bar{H}_\Sigma$ coincides on both $\freeprod$-translates. Let $\gamma(\bar{x})=\gamma'(\bar{y})$ for $\bar{x},\bar{y}\in \FD(\ThePct,\TheCone),\bar{x}\neq\bar{y}$, $\gamma,\gamma'\in\freeprod,\gamma\neq\gamma'$. This implies that $\bar{x}$ and $\bar{y}$ are contained in the same $\freeprod$-orbit. In particular, $\bar{x},\bar{y}\in\partial\FD(\ThePct,\TheCone)$ and $\bar{x},\bar{y}$ get identified via $\sim_\freeprod$. Thus, $\bar{H}(\bar{x})=\bar{H}(\bar{y})$. Without loss of generality, we may assume that $\gamma(\bar{x})=\gamma'(\bar{y})$ is contained in $\gamma(s_\nu)$, i.e.\  $\tilde{\gamma}=\gamma$ for $\gamma(\bar{x})$ and $\gamma'(\bar{y})$. 
Since the points $\bar{x}$ and $\bar{y}$ get identified by $\sim_\freeprod$, we further obtain $\gamma_{\bar{H},\bar{x}}=\gamma_{\bar{H},\bar{y}}$ and consequently 
\[
\bar{H}_\Sigma(\gamma(\bar{x}))=\gamma\gamma_{\bar{H},\bar{x}}(\bar{H}(\bar{x}))=\gamma\gamma_{\bar{H},\bar{y}}(\bar{H}(\bar{y}))=\bar{H}_\Sigma(\gamma'(\bar{y})). 
\]
\end{construction}

The following Lemma shows that $\bar{H}_\Sigma$ satisfies the expected properties. 

\begin{lemma}
\label{lem:varphi_Homeo_homo}
The map $\varphi:\HomeoId{\TheStrand}{D(\ThePct,\TheCone),\partial D(\ThePct,\TheCone)}\rightarrow\HomeoOrb{\TheStrand}{\Sigma_\freeprod(\ThePct),\partial\Sigma(\ThePct)}$ that maps $\bar{H}$ to $\bar{H}_\Sigma$ is a homomorphism. 
\end{lemma}
\begin{proof}
It remains to show 
that $\varphi$ satisfies the homomorphism property and that $\bar{H}_\Sigma$ is contained in $\HomeoIdOrb{\TheStrand}{\Sigma_\freeprod(\ThePct),\partial\Sigma(\ThePct)}$. 

\begin{intermediate}[Step $1$. The map $\varphi$ is a homomorphism]
For each $\bar{x}$ that is contained in the interior of $\FD(\ThePct,\TheCone)$ or in $\bigcup_{\nu=1}^\TheCone s_\nu$, we have 
\[
\bar{x}\stackrel{\varphi(\bar{H})}\mapsto\gamma_{\bar{H},\bar{x}}(\bar{H}(\bar{x}))\stackrel{\varphi(\bar{K})}\mapsto\gamma_{\bar{K},\bar{H}
(\bar{x})}\gamma_{\bar{H},\bar{x}}(\bar{K}\circ\bar{H})(\bar{x})=\gamma_{\bar{K},\bar{H}
(\bar{x})}\gamma_{\bar{H},\bar{x}}(\overline{K\circ H})(\bar{x}). 
\]
Let $\phi_{\bar{H}}$ (resp. $\phi_{\bar{K}}$) be a homotopy connecting $\bar{H}$ (resp. $\bar{K}$) to the identity. Then the concatenation $\phi_{\bar{K}}\circ\phi_{\bar{H}}$ connects $\overline{K\circ H}$ to $\id_{D(\ThePct,\TheCone)}$. This implies 
\[
\gamma_{\bar{K},\bar{H}
(\bar{x})}\gamma_{\bar{H},\bar{x}}=\gamma_{\overline{K\circ H},\bar{x}}. 
\]
Consequently, $\varphi(\bar{K})\circ\varphi(\bar{H})(\bar{x})=\varphi(\overline{K\circ H})(\bar{x})$, i.e.\ $\varphi$ is a homomorphism. 
\end{intermediate}

\begin{intermediate}[Step $2$. The map $\bar{H}_\Sigma$ is a homeomorphism]
Since $\varphi$ is a homomorphism, the inverse of $\bar{H}_\Sigma$ is given by $\bar{H}^{-1}_\Sigma$. Thus, it suffices to show that $\bar{H}_\Sigma$ is continuous for any $\bar{H}$ to obtain that $\bar{H}_\Sigma$ is a homeomorphism. 

As in Step \hyperlink{lem:pi_Homeo_cont_homo_step2}{2} in Lemma \ref{lem:pi_Homeo_cont_homo}, we consider pieces $\gamma(\FD(\ThePct,\TheCone))\cap\gamma'(\bar{H}_\Sigma^{-1}(\FD(\ThePct,\TheCone)))$ for $\gamma,\gamma'\in\freeprod$ to prove the continuity of $\bar{H}_\Sigma$. Using that $\FD(\ThePct,\TheCone)$ is a compact disk, each $\freeprod$-translate $\gamma(\FD(\ThePct,\TheCone))$ intersects with finitely many other $\freeprod$-translates $\gamma'(\bar{H}_\Sigma^{-1}(\FD(\ThePct,\TheCone)))$. 
Since $\FD(\ThePct,\TheCone)$ is a fundamental domain, $\freeprod$-translates only intersect in the boundary. Further, the domain and codomain of the restriction $\bar{H}_\Sigma\vert_{\gamma(\FD(\ThePct,\TheCone))\cap\gamma'(\bar{H}_\Sigma^{-1}(\FD(\ThePct,\TheCone)))}$ are both contained in a single $\freeprod$-translate of $\FD(\ThePct,\TheCone)$ and the restriction identifies with the corresponding piece 
\[
\bar{H}:\FD(\ThePct,\TheCone)\cap\gamma'(\bar{H}_\Sigma^{-1}(\FD(\ThePct,\TheCone)))\rightarrow\FD(\ThePct,\TheCone). 
\]
Since $\bar{H}$ is continuous, this implies that $\bar{H}_\Sigma$ is continuous on each piece. If two pieces intersect, the well-definedness of $\bar{H}_\Sigma$ implies that both definitions coincide on the intersection. Consequently, by \cite[Chapter III, Theorem 9.4]{Dugundji1966}, the continuous pieces match together to a continuous function~$\bar{H}_\Sigma$. 
\end{intermediate} 
\end{proof}

It remains to check that the above construction of $\bar{H}_\Sigma$ is appropriate in the sense that it induces an inverse homomorphism for $\piMap$. 

\begin{lemma}
\label{lem:varphi_Map_homo}
The map $\varphi:\HomeoId{\TheStrand}{D(\ThePct,\TheCone),\partial D(\ThePct,\TheCone)}\rightarrow\HomeoOrb{\TheStrand}{\Sigma(\ThePct),\partial\Sigma(\ThePct)}$ induces a homomorphism $\varphiMap:\MapId{\TheStrand}{D(\ThePct,\TheCone)}\rightarrow\MapIdOrb{\TheStrand}{\Sigma_\freeprod(\ThePct)}$. 
\end{lemma}
\begin{proof}
\new{We prove} that the induced map 
\[
\varphiMap:\MapId{\TheStrand}{D(\ThePct,\TheCone)}\rightarrow\MapIdOrb{\TheStrand}{\Sigma_\freeprod(\ThePct)},[\bar{H}]\mapsto[\bar{H}_\Sigma] 
\]
is well-defined and that its image is contained in $\MapIdOrb{\TheStrand}{\Sigma_\freeprod(\ThePct)}$. 

\begin{intermediate}[Step $1$. The map $\varphi_{\MapOp}$ is well-defined]
For $\bar{H},\bar{H}'\in\HomeoId{\TheStrand}{D(\ThePct,\TheCone),\partial D(\ThePct,\TheCone)}$ with $\bar{H}\new{\sim_\TheStrand}\bar{H}'$, we have an ambient isotopy $\bar{\phi}:I\rightarrow\Homeo{\TheStrand}{D(\ThePct,\TheCone),\partial D(\ThePct,\TheCone)}$ connecting $\bar{H}$ and $\bar{H}'$. We want to prove that $\bar{\phi}_\Sigma:=\varphi\circ\bar{\phi}$ is an ambient isotopy connecting $\bar{H}_\Sigma$ and $\bar{H}_\Sigma'$, i.e.\ we have to check that $\bar{\phi}_\Sigma$ is continuous. We begin with the following: 

\begin{claim*}
Let $t_0\in I$. For each $\varepsilon>0$, there exists a $\delta>0$ such that 
\[
\sup_{x\in D}\sup_{t\in(t_0-\delta,t_0+\delta)}d(\bar{\phi}_t(x),\bar{\phi}_{t_0}(x))<\varepsilon. 
\]
\end{claim*}

Let us consider the auxiliary function 
\[
\tilde{d}:D\times I\rightarrow\RR, \quad (x,t)\mapsto d(\bar{\phi}_t(x),\bar{\phi}_{t_0}(x)). 
\]
Firstly, we prove that $\tilde{d}$ is continuous. Observe that $\tilde{d}(x,t_0)=0$ and recall that $\bar{\phi}:I\rightarrow\Homeo{\TheStrand}{D(\ThePct,\TheCone),\partial D(\ThePct,\TheCone)},t\mapsto\bar{\phi}_t$ is continuous. By a classical result about function spaces, see, for instance, \cite[Chapter XII, Theorem 3.1]{Dugundji1966}, this implies that the map $I\times D(\ThePct,\TheCone)\rightarrow D(\ThePct,\TheCone),(t,x)\mapsto\bar{\phi}_t(x)$ is continuous.
Moreover, $d$ is continuous in both arguments. Hence, $\tilde{d}$ is continuous. 

To deduce the claim, let us fix an $\varepsilon>0$. We want to show that there exists a $\delta>0$ such that $\tilde{d}\vert_{D\times I_\delta}$ is bounded above by $\varepsilon$ with $I_\delta=(t_0-\delta,t_0+\delta)$. If no such $\delta$ exists, \new{for each $\delta>0$,} the intersection $D\times I_\delta\cap\tilde{d}^{-1}([\varepsilon,\infty))\neq\emptyset$. Choosing elements from these subsets, we find a sequence $(y_i)_{i\in\NN}$ with $y_i=(x_i,t_i)$ and $\vert t_i-t_0\vert<\frac{1}{i}$ for each $i\in\NN$. Since $D\times I$ is compact, there exists a convergent subsequence $(y_{i_j})_{j\in\NN}$. The condition on the $t_i$ implies that $y=\lim_{j\rightarrow\infty}y_{i_j}$ is of the form $y=(x,t_0)$. In particular, $\tilde{d}(y)=0$. Since $\tilde{d}(y_{i_j})\geq\varepsilon$ for all $j\in\NN$, this contradicts the continuity of $\tilde{d}$. Hence, there exists $\delta>0$ such that $\tilde{d}\vert_{D\times I_\delta}$ is bounded above by $\varepsilon$ and the claim follows. 

Now let $K\subseteq\Sigma$ be a compact set and $U_\varepsilon=(\bar{\phi}_{\Sigma,t_0}(K))_\varepsilon$ an open $\varepsilon$-neighborhood of $\bar{\phi}_{\Sigma,t_0}(K)$. Clearly, $\bar{\phi}_{\Sigma,t_0}\in V(K,U_\varepsilon)$. Further, the claim 
implies
\begin{equation}
\label{eq:phi_t_in_eps-ball}
\bar{\phi}_t(\bar{x})\in B_\varepsilon(\bar{\phi}_{t_0}(\bar{x})) \text{ for } t\in(t_0-\delta,t_0+\delta) \text{ and } \bar{x}\in D. 
\end{equation}

Recall that $\bar{\phi}_{\Sigma,t}$ maps $\gamma(\bar{x})$ to $\tilde{\gamma}\gamma_{\bar{\phi}_t,\bar{x}}(\bar{\phi}_t(\bar{x}))$. The group element $\gamma_{\bar{\phi}_t,\bar{x}}$ detects if $\bar{\phi}_{\Sigma,t}(\bar{x})$ is contained in the same $\freeprod$-translate as $\bar{\phi}_{\Sigma,t_0}(\bar{x})$: It coincides with $\gamma_{\bar{\phi}_{t_0},\bar{x}}$ if the trace $\bar{\phi}_{s}(\bar{x})$ of the ambient isotopy does not intersect any $\freeprod$-translates of $\partial\FD(\ThePct,\TheCone)$ for $s$ between $t$ and $t_0$. Otherwise $\gamma_{\bar{\phi}_t,\bar{x}}\neq\gamma_{\bar{\phi}_{t_0},\bar{x}}$ describes an adjacent translate of $\FD(\ThePct,\TheCone)$ that contains $\bar{\phi}_{\Sigma,t}(\bar{x})$. Using (\ref{eq:phi_t_in_eps-ball}), this implies that $\bar{\phi}_{\Sigma,t}(\gamma(\bar{x}))=\tilde{\gamma}\gamma_{\bar{\phi}_t,\bar{x}}(\bar{\phi}_t(\bar{x}))$ is contained in $B_\varepsilon(\bar{\phi}_{\Sigma,t_0}(\bar{x}))$ for $t\in(t_0-\delta,t_0+\delta)$. This is equivalent to $\bar{\phi}_{\Sigma,t}\in V(K,U_\varepsilon)$ for $t\in(t_0-\delta,t_0+\delta)$, i.e.\ $\bar{\phi}_\Sigma$ is continuous. Hence, the induced map $\varphiMap$ is well-defined. 
\end{intermediate}

\begin{intermediate}[Step $2$. $\new{[\bar{H}_\Sigma]}$ is contained in $\MapIdOrb{\TheStrand}{\Sigma_\freeprod(\ThePct)}$]
\new{This requires to prove that} $\bar{H}_\Sigma$ is $\freeprod$-equivariantly homotopic to $\id_{\Sigma(\ThePct)}$ relative $r_1,...,r_\ThePct$. By definition, $\bar{H}$ is homotopic to $\id_{D(\ThePct,\TheCone)}$ relative $\bar{r}_1,...,\bar{r}_\ThePct$ and $\bar{\cp}_1,...,\bar{\cp}_\TheCone$ via a homotopy~$\bar{\psi}$. Following the same argument as above, the induced map $\bar{\psi}_\Sigma:=\varphi\circ\bar{\psi}$ yields the required homotopy between $\bar{H}_\Sigma$ and $\id_{\Sigma(\ThePct)}$. Hence, $\bar{H}_\Sigma$ is contained in $\MapIdOrb{\TheStrand}{\Sigma_\freeprod(\ThePct)}$. 
\end{intermediate}
\end{proof}

\begin{proof}[Proof of Proposition \textup{\ref{prop:iso_Map_orb_disk}}]
On the one hand, let $\bar{H}\in\HomeoId{\TheStrand}{D(\ThePct,\TheCone)}$. Then, by definition, $\varphi(\bar{H})=\bar{H}_\Sigma$ maps $\gamma(x)$ onto $\gamma\gamma_{\bar{H},x}(\bar{H}(\bar{x}))$. This implies that $\pi\circ\varphi(\bar{H})=\bar{H}$ and consequently, $\piMap\circ\varphiMap=\id_{\MapId{\TheStrand}{D(\ThePct,\TheCone)}}$. 

On the other hand, for each homeomorphism $H$ that represents an element in $\MapIdOrb{\TheStrand}{\Sigma_\freeprod(\ThePct)}$, there exists a group element $\gamma_{H,x}$ with $H(x)=\gamma_{H,x}(\bar{H}(\bar{x}))$ for each $x\in\Sigma(\ThePct)$. Further, recall from Definition \ref{def:Map_orb(ThePct)} that if we forget the marked points $p_1,...,p_\TheStrand$, we obtain an ambient isotopy $\phi_t$ from $\id_{\Sigma(\ThePct)}$ to $H$. The induced ambient isotopy $\bar{\phi}_t$ yields $\gamma_{\bar{H},\bar{x}}$. More precisely, $\gamma_{\bar{H},\bar{x}}$ is determined by the sequence of algebraic intersections of $\bar{\phi}_t(\bar{x})$ with $\bar{s}_\nu$ for $1\leq\nu\leq\TheCone$. The induced homotopy $\bar{\phi}_t(\bar{x})$ intersects $\bar{s}_\nu$ if and only if $\phi_t(x)$ is going over to another $\freeprod$-translate of $\FD(\ThePct,\TheCone)$ \new{passing} a $\freeprod$-translate of $s_\nu$. Hence, $\gamma_{\bar{H},\bar{x}}$ coincides with $\gamma_{H,x}$, i.e.\ $\bar{H}_\Sigma=H$. This implies $\varphi\circ\pi(H)=H$ and consequently $\varphiMap\circ\piMap=\id_{\MapIdOrb{\TheStrand}{\Sigma_\freeprod(\ThePct)}}$. 

Hence, the groups $\MapId{\TheStrand}{D(\ThePct,\TheCone)}$ and $\MapIdOrb{\TheStrand}{\Sigma_\freeprod(\ThePct)}$ are isomorphic. Obviously, the constructed isomorphisms restrict to the pure subgroups. 
\end{proof}

In particular, Proposition \ref{prop:iso_Map_orb_disk} yields a proof that $\MapOrb{1}{D_{\cycm}}$ is infinite cyclic, \new{which was the object of Example \ref{ex:orb_mcg_one_pct_ZZ}.} 

\begin{proof}[Proof of Example \textup{\ref{ex:orb_mcg_one_pct_ZZ}}]
The Alexander trick shows that the groups $\MapOrb{}{D_{\cycm}}$ and $\Map{}{D(0,1)}$ are both trivial, see  Example \ref{ex:orb_mcg_one_pct_ZZ} and \cite[Lemma 2.1]{FarbMargalit2011}. This implies that 
\[
\MapOrb{1}{D_{\cycm}}\cong\MapIdOrb{1}{D_{\cycm}} \; \text{ and } \; \Map{1}{D(0,1)}\cong\MapId{1}{D(0,1)}. 
\]
By Proposition \ref{prop:iso_Map_orb_disk}, we further obtain that $\MapIdOrb{1}{D_{\cycm}}\cong\MapId{1}{D(0,1)}$. Consequently, the group $\MapOrb{1}{D_{\cycm}}$ is isomorphic to $\Map{1}{D(0,1)}$. Moreover, the point-pushing map 
\[
\Push_1:\pi_1(D(0,1),p_1)\rightarrow\Map{1}{D(0,1)}
\]
defined in \cite[Theorem 4.6]{FarbMargalit2011} yields an isomorphism. Since the fundamental group of the punctured disk $D(0,1)$ is isomorphic to $\ZZ$, this implies that $\MapOrb{1}{D_{\cycm}}$ is infinite cyclic. 

A generator of $\pi_1(D(0,1),p_1)$ is represented by a loop that encircles the puncture. The image of this loop under $\Push_1$ is a twist of $p_1$ around the puncture in $D(0,1)$. In $\MapOrb{1}{D_{\cycm}}$ this twist corresponds to the $\frac{2\pi}{\TheOrder}$-twist $\TwsC$ \new{from Figure \ref{fig:2_pi_m-twist}.} 
\end{proof}

\newpage

\subsection{The generalized Birman exact sequence for orbifold mapping class groups}
\label{subsec:gen_Birman_exact_seq_orb}

The identification of $\MapIdOrb{\TheStrand}{\Sigma_\freeprod(\ThePct)}$ and $\MapId{\TheStrand}{D(\ThePct,\TheCone)}$ leads to the Birman exact sequence for the orbifold mapping class groups from Theorem \ref{thm-intro:Birman_exact_seq}. 

\begin{theorem}[Birman exact sequence]
\label{thm:Birman_es_orb}
The following diagram commutes and each row is exact: 
\\
\begin{adjustbox}{center}
\begin{tikzcd}[column sep=40pt]
1\rightarrow\pi_1\left(\Conf_\TheStrand\left(D(\ThePct,\TheCone)\right)\right) \arrow[r,"\Push_\TheStrand"] \arrow[equal]{d} & \Map{\TheStrand}{D(\ThePct,\TheCone)}\arrow[r,"\Forget_\TheStrand"] & \Map{}{D(\ThePct,\TheCone)}\rightarrow1 
\\
1\rightarrow\pi_1\left(\Conf_\TheStrand\left(D(\ThePct,\TheCone)\right)\right) \arrow[r,"\Push_\TheStrand"] \arrow[equal]{d} & \MapId{\TheStrand}{D(\ThePct,\TheCone)} \arrow[r,"\Forget_\TheStrand"] \arrow[u,hook] \arrow[d,"\varphiMap"] & \MapId{}{D(\ThePct,\TheCone)}=1
\\
1\rightarrow\pi_1\left(\Conf_\TheStrand\left(D(\ThePct,\TheCone)\right)\right) \arrow[r,"\new{\Push_\TheStrand^{orb}}"] & \MapIdOrb{\TheStrand}{\Sigma_\freeprod(\ThePct)} \arrow[r,"\Forget_\TheStrand^{orb}"] & \MapOrb{}{\Sigma_\freeprod(\ThePct)}=1  
\end{tikzcd}
\end{adjustbox}
\end{theorem}
\begin{proof}
The upper row is the generalized Birman exact sequence for the disk with $\ThePct+\TheCone$ punctures. The map $\Forget_\TheStrand$ is defined by forgetting the marked points $\bar{p}_1,...,\bar{p}_\TheStrand$. The map $\Push_\TheStrand$ comes from the extension of a braid $b$ in $\Conf_\TheStrand\left(D(\ThePct,\TheCone)\right)$ to an ambient isotopy $I\rightarrow D(\ThePct,\TheCone)$ from $\id_{D(\ThePct,\TheCone)}$ to $H_b$. Then the image of $\Push_\TheStrand$ is defined by $\Push_\TheStrand([b])=[H_b]$. Intuitively, the ambient isotopy is obtained by placing our fingers at the initial points of the strand of $b$ and pushing along the braid. The disk $D(\ThePct,\TheCone)$ drags along as we push. Formally, the point-pushing map is induced by the fiber bundle 
\[
\Homeo{\TheStrand}{D(\ThePct,\TheCone),\partial D(\ThePct,\TheCone)}\rightarrow\Homeo{}{D(\ThePct,\TheCone),\partial D(\ThePct,\TheCone)}\rightarrow\Conf_\TheStrand\left(\inter{D}(\ThePct,\TheCone)\right). 
\]
See \cite[Sections 4.2, 9.1.4]{FarbMargalit2011} for details. 

The ambient isotopy from $\id_{D(\ThePct,\TheCone)}$ to $H_b$ that defines $\Push_\TheStrand$ is relative $\bar{\cp}_1,...,\bar{\cp}_\TheCone$ and $\bar{r}_1,...,\bar{r}_\ThePct$. Hence, the image of $\Push_\TheStrand$ is contained in $\MapId{\TheStrand}{D(\ThePct,\TheCone)}$ and by Definition \ref{def:Map_disk(ThePct)} the image of $\Forget\vert_{\MapId{\TheStrand}{D(\ThePct,\TheCone)}}$ is trivial. Hence, the first row restricts to the short exact sequence in the second row. 

The map \new{$\Push_\TheStrand^{orb}$} 
in the last row is given by the composition $\varphiMap\circ \Push_\TheStrand$. As a composition of isomorphisms, this map is an isomorphism. \new{Hence, the third row is a short exact sequence, i.e.\ the short exact sequence from Theorem~\ref{thm-intro:Birman_exact_seq}, and the diagram commutes.} 
\end{proof}

\new{As in the classical case,} the Birman exact sequence restricts to pure subgroups. The following two lemmas are required for the proof. We recall from \eqref{eq:def_metric_iso_freeprod-action}, that we endowed $\Sigma$ with a metric such that the $\freeprod$-action is isometric. The \textit{$\varepsilon$-collar} of $\partial\Sigma(\ThePct)$ is the subsurface $\{x\in\Sigma(\ThePct)\mid d(x,y)\leq\varepsilon \text{ for some } y\in\partial\Sigma(\ThePct)\}$. 

\begin{lemma}
\label{lem:mcg_id_eps-collar}
Let $\varepsilon>0$ such that each marked point in $\freeprod(\{p_1,...,p_{\TheStrand-1}\})$ and each puncture $\freeprod(\{r_1,...,r_\ThePct\})$ has distance greater than $\varepsilon$ from the boundary. Then 
each homeomorphism $H$ that represents a mapping class in $\MapIdOrb{\TheStrand-1}{\Sigma_\freeprod(\ThePct)}$ is ambient isotopic (relative $p_1,...,p_{\TheStrand-1},r_1,...,r_\ThePct$) to a homeomorphism $H_\varepsilon$ that coincides with the identity on an $\varepsilon$-collar of $\partial\Sigma(\ThePct)$. 
\end{lemma}
\begin{proof}
The claimed ambient isotopy is a variant of the Alexander trick: Let $\partial\Sigma(\varepsilon)^\varepsilon$ be a system of boundary parallel arcs with distance $\varepsilon$ to the boundary. Using that the $\varepsilon$-collar bounded by $\partial\Sigma(\ThePct)^\varepsilon$ and $\partial\Sigma$ neither contains marked points nor punctures, we can push the boundary along the collar to $\partial\Sigma(\ThePct)^\varepsilon$ and replace $H$ with the identity outside $\partial\Sigma(\ThePct)^\varepsilon$ (see Figure~\ref{fig:amb_iso_collar}). This ambient isotopy connects $H$ to a homeomorphism~$H_\varepsilon$ that coincides with the identity on the $\varepsilon$-collar of $\partial\Sigma(\ThePct)$. 
\end{proof}

\begin{figure}[H]
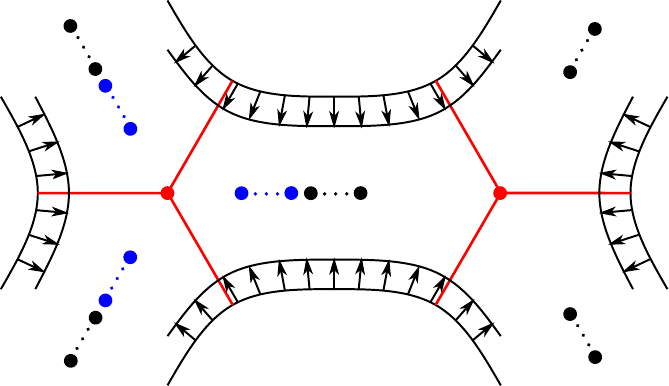
\caption{A variant of the Alexander trick.}
\label{fig:amb_iso_collar}
\end{figure}

Formally, the following lemma is independent from the previous one. However, given marked points $p_1,...,p_{\TheStrand-1}$, our goal is to apply it to add another marked point lying in a fixed $\varepsilon$-collar. 

\begin{lemma}
\label{lem:mcg_p_n_eps}
Let $\varepsilon>0$ and $p_{\TheStrand(\varepsilon)}$ be a point in $\inter{\FD}(\ThePct,\TheCone)$ with $p_{\TheStrand(\varepsilon)}\neq p_\Strand$ for $1\leq\Strand<\TheStrand$ and distance less than $\varepsilon$ to the boundary. 
The mapping class group $\MapIdOrb{\TheStrand}{\Sigma_\freeprod(\ThePct)}$ with respect to marked points $p_1,...,p_\TheStrand$ is isomorphic to $\MapIdOrb{\TheStrand(\varepsilon)}{\Sigma_\freeprod(\ThePct)}$ with respect to marked points $p_1,...,p_{\TheStrand-1},p_{\TheStrand(\varepsilon)}$. The isomorphism restricts to the pure subgroups. 
\end{lemma}
\begin{proof}
Since $p_\TheStrand$ and $p_{\TheStrand(\varepsilon)}$ are both contained in $\inter{\FD}(\ThePct,\TheCone)$, there exists a connecting arc contained in $\inter{\FD}(\ThePct,\TheCone)\setminus\{p_1,...,p_{\TheStrand-1}\}$. \new{Via point pushing along this arc,} we find a homeomorphism $P_\TheStrand^\varepsilon$ in $\HomeoOrb{\TheStrand-1}{\Sigma_\freeprod(\ThePct),\partial\Sigma(\ThePct)}$ which maps $p_\TheStrand$ to $p_{\TheStrand(\varepsilon)}$. 
The homeomorphism $P_\TheStrand^\varepsilon$ induces an isomorphism 
\[
\MapIdOrb{\TheStrand}{\Sigma_\freeprod(\ThePct)}\rightarrow\MapIdOrb{\TheStrand(\varepsilon)}{\Sigma_\freeprod(\ThePct)}, [H]\mapsto [P_\TheStrand^\varepsilon\circ H\circ(P_\TheStrand^\varepsilon)^{-1}]. 
\]
\end{proof}

\begin{corollary}
\label{cor:pure_orb_mcg_ses}
The following diagram is a short exact sequence that splits:
\\
\begin{adjustbox}{center}
\begin{tikzcd}[column sep=50pt]
1\rightarrow\freegrp{\TheStrand-1+\ThePct+\TheCone} \arrow[r,"\PushPMap^{orb}"] & \PMapIdOrb{\TheStrand}{\Sigma_\freeprod(\ThePct)}\arrow[r,"\ForgetPMap^{orb}"] & \PMapIdOrb{\TheStrand-1}{\Sigma_\freeprod(\ThePct)}\rightarrow1. 
\end{tikzcd}
\end{adjustbox}
\end{corollary}
\begin{proof}
For $\TheStrand=1$, the point-pushing along loops in $\pi_1\left(\Conf_1\left(D(\ThePct,\TheCone)\right)\right)$ from Theorem~\ref{thm:Birman_es_orb} yields pure homeomorphisms that represent elements in $\PMapIdOrb{1}{\Sigma_\freeprod(\ThePct)}$. Hence, the short exact sequence for orbifold mapping class groups restricts to a short exact sequence of pure subgroups:
\\
\begin{adjustbox}{center}
\begin{tikzcd}[column sep=30pt]
1\rightarrow\pi_1(\Conf_1(D(\ThePct,\TheCone))\arrow[r,"\Push_1^{orb}"] & \PMapIdOrb{1}{\Sigma_\freeprod(\ThePct)}\arrow[r,"\Forget_1^{orb}"] & \underbrace{\PMapIdOrb{}{\Sigma_\freeprod(\ThePct)}}_{=1}\rightarrow1. 
\end{tikzcd}
\end{adjustbox}
Since $\Conf_1\left(D(\ThePct,\TheCone)\right)$ is homeomorphic to a disk with $\ThePct+\TheCone$ punctures, its fundamental group is the free group $\freegrp{\ThePct+\TheCone}$. Replacing $\ThePct$ by $\TheStrand-1+\ThePct$, yields the first row of the following diagram: 
\\
\begin{adjustbox}{center}
\begin{tikzcd}[column sep=40pt]
1\rightarrow\freegrp{\TheStrand-1+\ThePct+\TheCone} \arrow[r,"\Push_1^{orb}"] \arrow[d,equal] & \PMapIdOrb{1}{\Sigma_\freeprod(\TheStrand-1+\ThePct)} \arrow[r,"\Forget_1^{orb}"] \arrow[d,hook] & \PMapIdOrb{}{\Sigma_\freeprod(\TheStrand-1+\ThePct)}=1 \arrow[d,hook] 
\\
1\rightarrow\freegrp{\TheStrand-1+\ThePct+\TheCone} \arrow[r,"\PushPMap^{orb}"] & \PMapIdOrb{\TheStrand}{\Sigma_\freeprod(\ThePct)} \arrow[r,"\ForgetPMap^{orb}"] & \PMapIdOrb{\TheStrand-1}{\Sigma_\freeprod(\ThePct)}\rightarrow1. 
\end{tikzcd}
\end{adjustbox}
Every homeomorphism that is $\freeprod$-equivariantly homotopic to the identity relative $r_1,...,r_\ThePct,p_1,...,p_{\TheStrand-1}$ in particular is $\freeprod$-equivariantly homotopic to $\id_{\Sigma(\ThePct)}$ relative $r_1,...,r_\ThePct$. This allows us to define the vertical maps in the above diagram as embeddings. If we compose the $\Push_1^{orb}$-map with the embedding 
\[
\PMapIdOrb{1}{\Sigma_\freeprod(\TheStrand-1+\ThePct)}\hookrightarrow\PMapIdOrb{\TheStrand}{\Sigma_\freeprod(\ThePct)}, 
\]
we obtain $\PushPMap^{orb}$. \new{As in the first row,} $\ForgetPMap^{orb}$ is given by forgetting \new{the} marked point $p_\TheStrand$. Using these maps, the above diagram commutes. 

It remains to show that the bottom row is a short exact sequence. \new{As a composition of embeddings,} $\PushPMap^{orb}$ is \new{injective}. Since the diagram commutes, its image 
is contained in the kernel of $\ForgetPMap^{orb}$. 
On the other hand, if $[H]$ is contained in the kernel of $\ForgetPMap^{orb}$, 
$H$ is homotopic to $\id_{\Sigma(\ThePct)}$ relative marked points $r_1,...,r_\ThePct,p_1,...,p_{\TheStrand-1}$. This implies that $[H]$ is contained in the pure subgroup $\PMapIdOrb{1}{\Sigma_\freeprod(\TheStrand-1+\ThePct)}$ and the exactness of the first row implies that an element of $\freegrp{\TheStrand-1+\ThePct+\TheCone}$ maps onto~$[H]$. 

Further, we need to show that $\ForgetPMap^{orb}$ is surjective. For this purpose, we construct a map $\mathrm{Res}_\varepsilon:\MapIdOrb{\TheStrand-1}{\Sigma_\freeprod(\ThePct)}\rightarrow\MapIdOrb{\TheStrand}{\Sigma_\freeprod(\ThePct)}$ that restricts to a section of $\ForgetPMap^{orb}$. Therefore, we use Lemma \ref{lem:mcg_p_n_eps} to assume that the marked points $\freeprod(p_\TheStrand)$ are contained in the $\varepsilon$-collar of $\partial\Sigma(\ThePct)$ but no other marked point or puncture is contained in this $\varepsilon$-collar. By Lemma \ref{lem:mcg_id_eps-collar}, each mapping class $[H]$ in $\MapIdOrb{\TheStrand-1}{\Sigma_\freeprod(\ThePct)}$ is represented by a homeomorphism $H_\varepsilon$ that is the identity on the $\varepsilon$-collar. Since all $\freeprod$-translates of $p_\TheStrand$ are contained in the $\varepsilon$-collar, the homeomorphism $H_\varepsilon$ in particular fixes the $\freeprod$-translates of the $\TheStrand$-th marked point, i.e.\ $H_\varepsilon$ represents an element in $\MapIdOrb{\TheStrand}{\Sigma_\freeprod(\ThePct)}$. We define 
\[
\textrm{Res}_\varepsilon:\MapIdOrb{\TheStrand-1}{\Sigma_\freeprod(\ThePct)}\rightarrow\MapIdOrb{\TheStrand}{\Sigma_\freeprod(\ThePct)}, \; [H]\mapsto[H_\varepsilon]. 
\]
To obtain that $\textrm{Res}_\varepsilon$ is well-defined let $H$ and $H'$ be two homeomorphisms that are ambient isotopic (relative marked points $r_1,...,r_\ThePct,p_1,...,p_{\TheStrand-1}$) via $H_t$.  We can apply the pushing from Lemma \ref{lem:mcg_id_eps-collar} to $H_t$. This yields an ambient isotopy $(H_t)_\varepsilon$ that fixes $r_1,...,r_\ThePct,p_1,...,p_{\TheStrand-1}$ and restricts to the identity on the $\varepsilon$-collar. This ambient isotopy connects $H_\varepsilon$ and $H'_\varepsilon$ which implies that the map $\textrm{Res}_\varepsilon$ is well-defined. It is immediate that the restricted map $\textrm{Res}_\varepsilon\vert_{\PMapIdOrb{\TheStrand-1}{\Sigma_\freeprod(\ThePct)}}$ is a section of $\ForgetPMap^{orb}$. Moreover, the pushing procedure from Lemma \ref{lem:mcg_id_eps-collar} is compatible with the group structure: If $H$, $H'$ represent elements in $\MapIdOrb{\TheStrand-1}{\Sigma_\freeprod(\ThePct)}$, we have $(H'\circ H)_\varepsilon=H'_\varepsilon\circ H_\varepsilon$. This implies that $\textrm{Res}_\varepsilon$ is a homomorphism. Thus, the second row of the above diagram is a short exact sequence that splits. 
\end{proof}

\begin{definition}
\label{def:semidir_prod}
A group $G$ is a \textit{semidirect product} with \textit{normal subgroup}~$N$ and \textit{quotient} $H$ if there exists a short exact sequence 
\[
1\rightarrow N\xrightarrow{\iota}G\xrightarrow{\pi}H\rightarrow1
\]
that has a section $s:H\rightarrow G$. In this case, we denote $G=N\rtimes H$. 
\end{definition}

In particular, Corollary \ref{cor:pure_orb_mcg_ses} shows: $\PMapIdOrb{\TheStrand}{\Sigma_\freeprod(\ThePct)}$ is a semidirect product 
\[
\freegrp{\TheStrand-1+\ThePct+\TheCone}\rtimes\PMapIdOrb{\TheStrand-1}{\Sigma_\freeprod(\ThePct)}. 
\]

In the following, presentations of groups will be an important tool for us. In particular, presentations allow us to define group homomorphisms by assignments defined on generating sets \new{by the Theorem of von Dyck, see \cite[p.\ 346]{Rotman2012}.} 

\begin{lemma}[{\cite[Lemma 5.17]{Flechsig2023}}]
\label{lem:semidir_prod_pres}
Let $N$ and $H$ be groups given by presentations $N=\langle X\mid R\rangle$ and $H=\langle Y\mid S\rangle$. Then the following are equivalent: 
\begin{enumerate}
\item \label{lem:semidir_prod_pres_it1}
$G$ is a semidirect product with normal subgroup $N$ and quotient $H$. 
\item \label{lem:semidir_prod_pres_it2}
$G$ has a presentation 
\[
G=\langle X,Y\mid R,S,y^{\pm1}xy^{\mp1}=\phi_{y^{\pm1}}(x) \text{ for all } x\in X, y\in Y\rangle
\]
such that $\phi_{y^{\pm1}}(x)$ is a word in the alphabet~$X$ for all $x\in X$ and $y\in Y$. Moreover, for each $y\in Y$, the assignments 
\begin{equation}
\label{lem:semidir_prod_it2_cond_auto}
x\mapsto\phi_y(x)
\end{equation}
induce an automorphism $\phi_y\in\Aut(N)$ and the assignments 
\begin{equation}
\label{lem:semidir_prod_it2_cond_homo}
y\mapsto\phi_y 
\end{equation}
induce a homomorphism $H\rightarrow\Aut(N)$. 
\end{enumerate}
\end{lemma}

\subsection{Generating set and presentation of $\PMapIdOrb{\TheStrand}{\Sigma_\freeprod(\ThePct)}$} 
\label{subsec:gen_sets_and_pres_PMap}

Endowed with the Birman exact sequence for pure subgroups, our next step is to deduce a generating set and a presentation of $\PMapIdOrb{\TheStrand}{\Sigma_\freeprod(\ThePct)}$. \enew{Therefore, we recall the shape of the fundamental domain $\FD$ from Figure \ref{fig:fund_domain}. This allows us to embed $\FD$ in $\CC$ as the disk of radius $\frac{\TheStrand+\ThePct+\TheCone+2}{2}$ centered at $\frac{\TheStrand-\ThePct-\TheCone}{2}$. For each $1\leq\nu\leq\TheCone$, let $\cp_\nu$ be the upper boundary point of $\partial \FD$ with $\text{Re}(\cp_\nu)=-\ThePct-\nu$, for each $1\leq\NPct\leq\ThePct$, let $r_\NPct$ be the point $-\NPct\in\RR$ and for each $1\leq\Strand\leq\TheStrand$, let $p_\Strand$ be the point $\Strand$ in $\RR$ (see Figure \ref{fig:disks_with_marked_pts}). Moreover, recall that each cone point in $\partial\FD$ has two adjacent arcs that lie in $\partial\FD$. For technical reasons, let us assume that the arcs adjacent to $\cp_\nu$ embed into $\partial\FD$ as the boundary arcs with positive imaginary part and real part between $-\ThePct-\nu-\frac{1}{2}$ and $-\ThePct-\nu$ or $-\ThePct-\nu$ and $-\ThePct-\nu+\frac{1}{2}$, respectively. 

For every $1\leq\Str<\Strand\leq\TheStrand$, let $D_{\Str,\Strand}\subseteq\FD(\ThePct)$ be the disk
\begin{align*}
& \left(B_{\frac{1}{4}}(p_\Str)\cup B_{\frac{1}{4}}(p_\Strand)\right)\cap\{x\in\CC\mid\textup{Im}(x)\geq0\}
\\
\cup & A_{\frac{\Strand-\Str}{2}-\frac{1}{4},\frac{\Strand-\Str}{2}+\frac{1}{4}}\left(\tfrac{p_\Str+p_\Strand}{2}\right)\cap\{x\in\CC\mid\textup{Im}(x)\leq0\}
\end{align*}
where $A_{r,R}(x)$ denotes the annulus with inner radius $r$ and outer radius $R$ centered around $x$. 
The disk $D_{\Str,\Strand}$ contains precisely the marked points $p_\Str$ and $p_\Strand$. See Figure~\ref{fig:disks_with_marked_pts} (left) for a picture of $D_{\Str,\Strand}$. 

Moreover, for every $1\leq\NStrand\leq\TheStrand$ and $1\leq\NPct\leq\ThePct$, let $D_{r_\NPct,\NStrand}\subseteq\FD(\ThePct)$ be the disk 
\begin{align*}
& \left(B_{\frac{1}{4}}(r_\NPct)\cup B_{\frac{1}{4}}(p_\NStrand)\right)\cap\{x\in\CC\mid\textup{Im}(x)\geq0\}
\\
\cup & A_{\frac{\NStrand+\NPct}{2}-\frac{1}{4},\frac{\NStrand+\NPct}{2}+\frac{1}{4}}\left(\tfrac{r_\NPct+p_\NStrand}{2}\right)\cap\{x\in\CC\mid\textup{Im}(x)\leq0\}. 
\end{align*}
The disk $D_{r_\NPct,\NStrand}$ contains precisely the marked point $r_\NPct$ and $p_\NStrand$. See Figure \ref{fig:disks_with_marked_pts} (right) for a picture of $D_{r_\NPct,\NStrand}$. 
\begin{figure}[H]
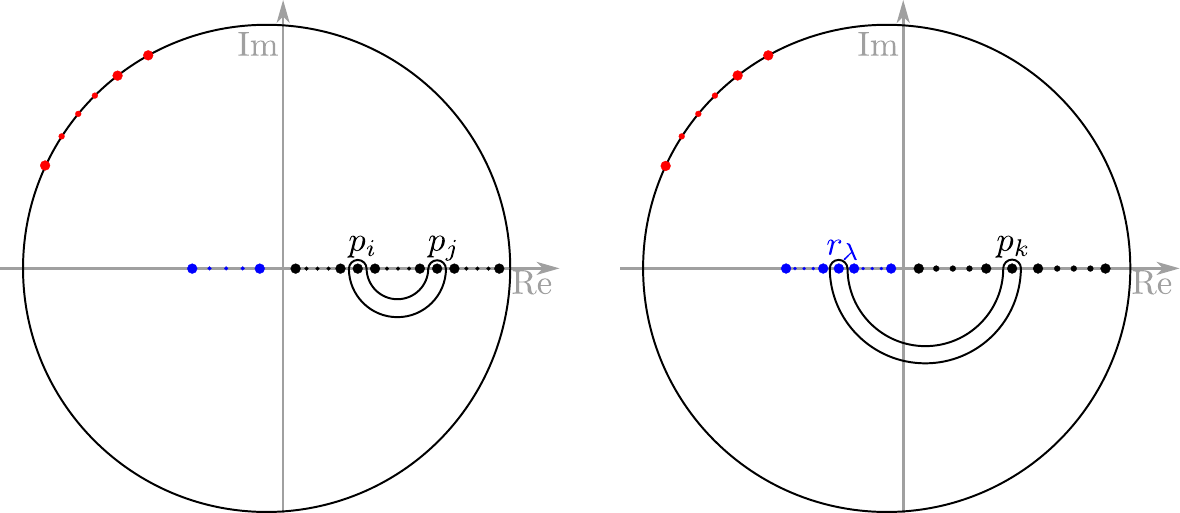
\caption{The disks $D_{\Str,\Strand}$ (left) and $D_{r_\NPct,\NStrand}$ (right).}
\label{fig:disks_with_marked_pts}
\end{figure}
}
The homeomorphisms $\Twist_{\Strand\Str}$ and $\TwistP_{\NStrand\NPct}$ perform the twists pictured in Figure \ref{fig:homeos_twist_marked_pts} on each $\freeprod$-translate of $D_{\Str,\Strand}$ and $D_{r_\NPct,\NStrand}$. 
\begin{figure}[H]
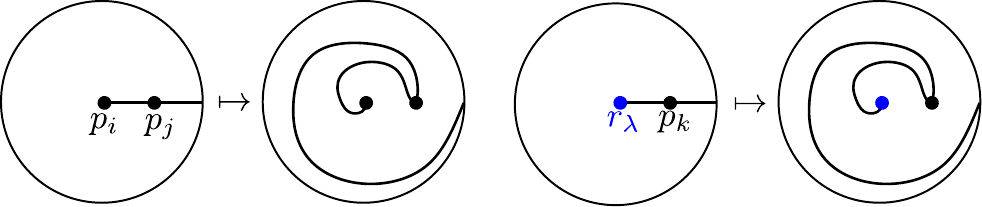
\caption{The twists induced by $\Twist_{\Strand\Str}$ (left) and $\TwistP_{\NStrand\NPct}$ (right).}
\label{fig:homeos_twist_marked_pts}
\end{figure}

\enew{Moreover, for every $1\leq\NStrand\leq\TheStrand$ and $1\leq\nu\leq\TheCone$, let $\tilde{D}_{\cp_\nu,\NStrand}$ be the disk 
\begin{align*}
& B_{\frac{1}{4}}(p_\NStrand)\cap\{x\in\CC\mid\textup{Im}(x)\geq0\}
\\
\cup & A_{\frac{\NStrand+\ThePct+\nu}{2}-\frac{1}{4},\frac{\NStrand+\ThePct+\nu}{2}+\frac{1}{4}}\left(\tfrac{-\ThePct-\nu+p_\NStrand}{2}\right)\cap\{x\in\CC\mid\textup{Im}(x)\leq0\}
\\
\cup & \left\lbrace x\in\CC\mid\textup{Im}(x)\geq0, \text{Re}(x)\in\left[-\ThePct-\nu-\tfrac{1}{4},-\ThePct-\nu+\tfrac{1}{4}\right]\right\rbrace\cap\FD. 
\end{align*}
Then $D_{\cp_\nu,\NStrand}:=\cyc{\TheOrder_\nu}.\tilde{D}_{\cp_\nu,\NStrand}$ is a $\cyc{\TheOrder_\nu}$-invariant disk that contains the cone point $\cp_\nu$ and the adjacent marked points $\ZZ_{\TheOrder_\nu}(p_\NStrand)$. See Figure \ref{fig:disk_with_marked_pts_cp} for a picture of $\tilde{D}_{\cp_\nu,\NStrand}$ (left) and an example of the disk $D_{\cp_\nu,\NStrand}\subseteq\Sigma(\ThePct)$ (right). 
\begin{figure}[H]
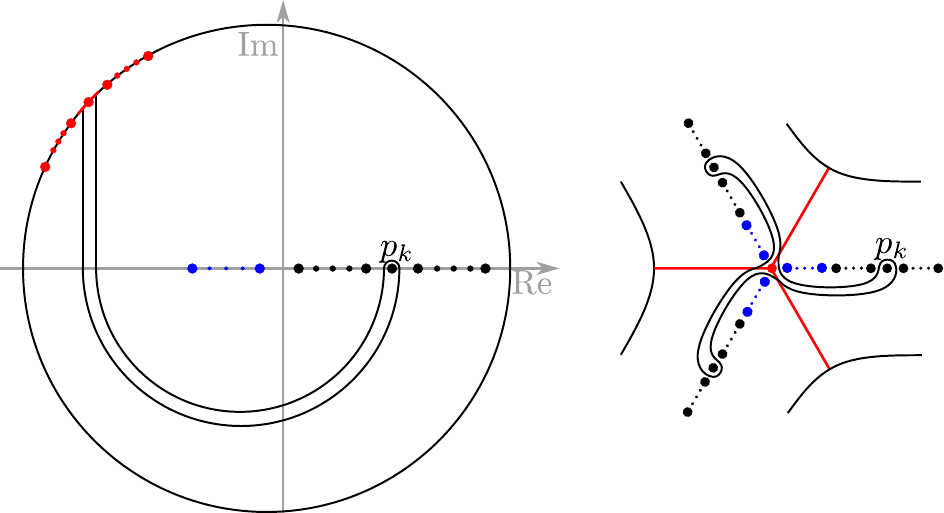
\caption{The disk $\tilde{D}_{\cp_\nu,\NStrand}\subseteq\FD$ (left) and $D_{\cp_\nu,\NStrand}\subseteq\Sigma(\ThePct)$ for a cone point of order three (right).}
\label{fig:disk_with_marked_pts_cp}
\end{figure}
}

Let $\TwistC_{\NStrand\nu}$ be the homeomorphism that performs a $\frac{2\pi}{\TheOrder_\nu}$-twist as in Figure \ref{fig:2_pi_m-twist} 
on each $\freeprod$-translate of $D_{\cp_\nu,\NStrand}$. For the homeomorphisms $\Twist_{\Strand\Str}, \TwistP_{\NStrand\NPct}$ and $\TwistC_{\NStrand\nu}$, we will use their names as acronyms of the corresponding mapping classes. 

\begin{corollary}
\label{cor:PMap_gen_set}
The pure mapping class group $\PMapIdOrb{\TheStrand}{\Sigma_\freeprod(\ThePct)}$ is generated by 
\[
\Twist_{\Strand\Str},\TwistP_{\NStrand\NPct} \; \text{ and } \; \TwistC_{\NStrand\nu}
\]
for $1\leq\Str,\Strand,\NStrand\leq\TheStrand$ with $\Str<\Strand$, \; $1\leq\NPct\leq\ThePct$ and $1\leq\nu\leq\TheCone$. 
\end{corollary}
\begin{proof}
The proof that the above elements generate $\PMapIdOrb{\TheStrand}{\Sigma_\freeprod(\ThePct)}$ proceeds by induction on $\TheStrand$. 

For $\TheStrand=0$, the group coincides with $\PMapIdOrb{}{\Sigma_\freeprod(\ThePct)}$. By Definition \ref{def:Map_orb(ThePct)}, this group is trivial. For $\TheStrand\geq1$, we recall the split short exact sequence from Corollary~\ref{cor:pure_orb_mcg_ses}: 
\[
1\rightarrow\freegrp{\TheStrand-1+\ThePct+\TheCone}\xrightarrow{\PushPMap^{orb}}\PMapIdOrb{\TheStrand}{\Sigma_\freeprod(\ThePct)}\xrightarrow{\ForgetPMap^{orb}}\PMapIdOrb{\TheStrand-1}{\Sigma_\freeprod(\ThePct)}\rightarrow1. 
\]
The included free group $\freegrp{\TheStrand-1+\ThePct+\TheCone}$ stems from $\pi_1(D(\TheStrand-1+\ThePct+\TheCone))$. This group embeds into $\PMapIdOrb{\TheStrand}{\Sigma_\freeprod(\ThePct)}$ via $\PushPMap^{orb}$. Recall that the natural generators of $\pi_1(D(\TheStrand-1+\ThePct+\TheCone))$ are represented by loops 
that encircle precisely one puncture of $D(\TheStrand-1+\ThePct+\TheCone)$. \new{These generators can be chosen such that $\PushPMap^{orb}$} maps them to the mapping classes represented by $\Twist_{\TheStrand\Strand},\TwistP_{\TheStrand\NPct}$ and $\TwistC_{\TheStrand\nu}$ for $1\leq\Strand<\TheStrand,1\leq\NPct\leq\ThePct$ and $1\leq\nu\leq\TheCone$. For $\TheStrand=1$, this implies that $\PMapIdOrb{1}{\Sigma_\freeprod(\ThePct)}$ is isomorphic to $\freegrp{\ThePct+\TheCone}$ and the elements $\TwistP_{1 \NPct},\TwistC_{1\nu}$ with $1\leq\NPct\leq\ThePct$ and $1\leq\nu\leq\TheCone$ form a basis of the free group. 

By induction, this allows us to assume that $\PMapIdOrb{\TheStrand-1}{\Sigma_\freeprod(\ThePct)}$ is generated by 
\[
\Twist_{\Strand\Str},\TwistP_{\NStrand\NPct} \; \text{ and } \; \TwistC_{\NStrand\nu}
\]
with $1\leq\Str,\Strand,\NStrand<\TheStrand,\Str<\Strand,1\leq\NPct\leq\ThePct$ and $1\leq\nu\leq\TheCone$. The section from Corollary~\ref{cor:pure_orb_mcg_ses} embeds $\PMapIdOrb{\TheStrand-1}{\Sigma_\freeprod(\ThePct)}$  into $\PMapIdOrb{\TheStrand}{\Sigma_\freeprod(\ThePct)}$ sending these generators 
of $\PMapIdOrb{\TheStrand-1}{\Sigma_\freeprod(\ThePct)}$ to their homonyms in $\PMapIdOrb{\TheStrand}{\Sigma_\freeprod(\ThePct)}$. Moreover, as described above, $\freegrp{\TheStrand-1+\ThePct+\TheCone}$ embeds into $\PMapIdOrb{\TheStrand}{\Sigma_\freeprod(\ThePct)}$ as the subgroup generated by 
\[
\Twist_{\TheStrand\Strand},\TwistP_{\TheStrand\NPct} \; \text{ and } \; \TwistC_{\TheStrand\nu} 
\]
with $1\leq\Strand<\TheStrand,1\leq\NPct\leq\ThePct$ and $1\leq\nu\leq\TheCone$. Hence, the short exact sequence from Corollary \ref{cor:pure_orb_mcg_ses} implies that the above set generates $\PMapIdOrb{\TheStrand}{\Sigma_\freeprod(\ThePct)}$. 
\end{proof}

The next step to deduce a presentation from Corollary \ref{cor:pure_orb_mcg_ses} is to observe relations for the above generators. 

\begin{lemma}
\label{lem:conj_manipulates_supp}
Let $A$ be a 
homeomorphism that represents an element in the group $\MapIdOrb{\TheStrand}{\Sigma_\freeprod(\ThePct)}$. Further, let $H$ be one of the homeomorphisms 
\[
\Twist_{\Strand\Str}, \TwistP_{\NStrand\NPct} \; \text{ and } \; \TwistC_{\NStrand\nu} 
\]
for some $1\leq\Str,\Strand,\NStrand\leq\TheStrand$ with $\Str<\Strand$, $1\leq\nu\leq\TheCone$ and $1\leq\NPct\leq\ThePct$. If $D_H\subseteq\FD(\ThePct)$ is the associated supporting disk of $H$ and $A(\freeprod(D_H))=\freeprod(D_H)$, then $[AHA^{-1}]=[H]$. 
\end{lemma}
\begin{proof}
In general, a conjugate $AHA^{-1}$ is related to $H$ by the following commuting diagram
\begin{center}
\begin{tikzcd}[column sep=100pt]
A(\freeprod(D_H)) \arrow[r,"AHA^{-1}\vert_{A(\freeprod(D_H))}"] & A(\freeprod(D_H))
\\
\freeprod(D_H) \arrow[u,"A\vert_{\freeprod(D_H)}"] \arrow[r,"H\vert_{\freeprod(D_H)}"] & \freeprod(D_H). \arrow[u,"A\vert_{\freeprod(D_H)}"]
\end{tikzcd}
\end{center}
Since we further assume $A(\freeprod(D_H))=\freeprod(D_H)$, both elements $H$ and $AHA^{-1}$ induce an element in $\MapOrb{\TheStrand(H)}{D_H}$, respectively. 

If the supporting disk contains no cone point, it either is a disk $D_{\Str,\Strand}$ with two marked points $p_\Str$ and $p_\Strand$ or a disk $D_{r_\NPct,\NStrand}$ with one puncture $r_\NPct$ and one marked point~$p_\NStrand$. In the first case, $\MapOrb{\TheStrand(H)}{D_H}$ is isomorphic to the braid group on two strands by an application of the Alexander trick, see \cite[Section 9.1.3]{FarbMargalit2011}. Since the Alexander trick also applies to a disk with one puncture, the group in the second case is isomorphic to the fundamental group of a punctured disk. In both cases this implies that $\MapOrb{\TheStrand(H)}{D_H}$ is infinite cyclic. 

Otherwise $D_H$ contains a cone point $\cp_\nu$ and $\TheOrder_\nu$ translates of marked points, i.e.\ $(D_H)_{\cyc{\TheOrder_\nu}}$ is the orbifold described in Example \ref{ex:orb_mcg_one_pct_ZZ}. \new{As determined in the example,} $\MapOrb{\TheStrand(H)}{D_H}$ is also infinite cyclic. In particular, the group $\MapOrb{\TheStrand(H)}{D_H}$ is abelian in each case. 

Since $A(\freeprod(D_H))=\freeprod(D_H)$, there exists an element $\gamma\in\freeprod$ such that 
\[
\gamma A:x\mapsto \gamma(A(x)) 
\]
defines a homeomorphism of $D_H$ that represents an element in $\MapOrb{\TheStrand(H)}{D_H}$. Using that these groups are abelian, this implies $[\gamma AH(\gamma A)^{-1}]=[H\vert_{D_H}]$. This is equivalent to $[AHA^{-1}\vert_{\gamma^{-1}(D_H)}]=[\gamma^{-1}H\gamma]$ with 
\[
\gamma^{-1}H\gamma:\gamma^{-1}(D_H)\rightarrow\gamma^{-1}(D_H), \quad x\mapsto\gamma^{-1}(H(\gamma(x))). 
\]
Since $H$ and $AHA^{-1}$ are $\freeprod$-equivariant, this implies that $[AHA^{-1}\vert_{D_H}]=[H\vert_{D_H}]$ and consequently $[AHA^{-1}]=[H]$. 
\end{proof}

Using Lemma \ref{lem:conj_manipulates_supp}, we observe the following relations in $\PMapIdOrb{\TheStrand}{\Sigma_\freeprod(\ThePct)}$: 

\begin{lemma}
\label{lem:PMap_gens_sat_rels}
Let $1\leq\Strr,\Str,\Strand,\NStrand,\NNStrand<\TheStrand$ with $\Strr<\Str<\Strand<\NStrand<\NNStrand$, $1\leq\Pc,\NPct\leq\ThePct$ with $\Pc<\NPct$ and $1\leq\mu,\nu\leq\TheCone$ with $\mu<\nu$. Then the following relations hold: 
\begin{enumerate}
\item 
\label{lem:PMap_gens_sat_rels_it1}
\begin{enumerate}
\item \label{lem:PMap_gens_sat_rels_it1a}
$\Twist_{\NNStrand\Strand}\Twist_{\TheStrand\Strand}\Twist_{\NNStrand\Strand}^{-1}= \Twist_{\TheStrand\Strand}^{-1}\Twist_{\TheStrand\NNStrand}^{-1}\Twist_{\TheStrand\Strand}\Twist_{\TheStrand\NNStrand}\Twist_{\TheStrand\Strand}\; \text{ and } \; \hypertarget{lem:PMap_gens_sat_rels_it1a_2}{\textup{a')}}\;\Twist_{\NNStrand\Strand}^{-1}\Twist_{\TheStrand\Strand}\Twist_{\NNStrand\Strand}=\Twist_{\TheStrand\NNStrand}\Twist_{\TheStrand\Strand}\Twist_{\TheStrand\NNStrand}^{-1}$, 
\item \label{lem:PMap_gens_sat_rels_it1b}
$\Twist_{\Strand\Str}\Twist_{\TheStrand\Strand}\Twist_{\Strand\Str}^{-1}=\Twist_{\TheStrand\Str}^{-1}\Twist_{\TheStrand\Strand}\Twist_{\TheStrand\Str}\; \text{ and } \; \hypertarget{lem:PMap_gens_sat_rels_it1b_2}{\textup{b')}}\;\Twist_{\Strand\Str}^{-1}\Twist_{\TheStrand\Strand}\Twist_{\Strand\Str}=\Twist_{\TheStrand\Strand}\Twist_{\TheStrand\Str}\Twist_{\TheStrand\Strand}\Twist_{\TheStrand\Str}^{-1}\Twist_{\TheStrand\Strand}^{-1}$, 
\item \label{lem:PMap_gens_sat_rels_it1c}
$\TwistP_{\Strand\NPct}\Twist_{\TheStrand\Strand}\TwistP_{\Strand\NPct}^{-1}=\TwistP_{\TheStrand\NPct}^{-1}\Twist_{\TheStrand\Strand}\TwistP_{\TheStrand\NPct} \; \text{ and } \;\hypertarget{lem:PMap_gens_sat_rels_it1c_2}{\textup{c')}}\;\TwistP_{\Strand\NPct}^{-1}\Twist_{\TheStrand\Strand}\TwistP_{\Strand\NPct}=\Twist_{\TheStrand\Strand}\TwistP_{\TheStrand\NPct}\Twist_{\TheStrand\Strand}\TwistP_{\TheStrand\NPct}^{-1}\Twist_{\TheStrand\Strand}^{-1}$, 
\item \label{lem:PMap_gens_sat_rels_it1d}
$\TwistC_{\Strand\nu}\Twist_{\TheStrand\Strand}\TwistC_{\Strand\nu}^{-1}=\TwistC_{\TheStrand\nu}^{-1}\Twist_{\TheStrand\Strand}\TwistC_{\TheStrand\nu}\; \text{ and } \; \hypertarget{lem:PMap_gens_sat_rels_it1d_2}{\textup{d')}}\;\TwistC_{\Strand\nu}^{-1}\Twist_{\TheStrand\Strand}\TwistC_{\Strand\nu}=\Twist_{\TheStrand\Strand}\TwistC_{\TheStrand\nu}\Twist_{\TheStrand\Strand}\TwistC_{\TheStrand\nu}^{-1}\Twist_{\TheStrand\Strand}^{-1}$, 
\item \label{lem:PMap_gens_sat_rels_it1e}
$\TwistP_{\Strand\NPct}\TwistP_{\TheStrand\NPct}\TwistP_{\Strand\NPct}^{-1}=\TwistP_{\TheStrand\NPct}^{-1}\Twist_{\TheStrand\Strand}^{-1}\TwistP_{\TheStrand\NPct}\Twist_{\TheStrand\Strand}\TwistP_{\TheStrand\NPct}\; \text{ and } \; \hypertarget{lem:PMap_gens_sat_rels_it1e_2}{\textup{e')}}\;\TwistP_{\Strand\NPct}^{-1}\TwistP_{\TheStrand\NPct}\TwistP_{\Strand\NPct}=\Twist_{\TheStrand\Strand}\TwistP_{\TheStrand\NPct}\Twist_{\TheStrand\Strand}^{-1}$, 
\item \label{lem:PMap_gens_sat_rels_it1f}
$\TwistC_{\Strand\nu}\TwistC_{\TheStrand\nu}\TwistC_{\Strand\nu}^{-1}=\TwistC_{\TheStrand\nu}^{-1}\Twist_{\TheStrand\Strand}^{-1}\TwistC_{\TheStrand\nu}\Twist_{\TheStrand\Strand}\TwistC_{\TheStrand\nu}\; \text{ and } \; \hypertarget{lem:PMap_gens_sat_rels_it1f_2}{\textup{f')}}\;\TwistC_{\Strand\nu}^{-1}\TwistC_{\TheStrand\nu}\TwistC_{\Strand\nu}=\Twist_{\TheStrand\Strand}\TwistC_{\TheStrand\nu}\Twist_{\TheStrand\Strand}^{-1}$,
\end{enumerate} 
\item 
\label{lem:PMap_gens_sat_rels_it2}
\begin{enumerate}
\item \label{lem:PMap_gens_sat_rels_it2a}
$[\Twist_{\Str\Strr},\Twist_{\TheStrand\Strand}]=1$, \; $[\TwistP_{\Str\NPct},\Twist_{\TheStrand\Strand}]=1 \; \text{ and } \; [\TwistC_{\Str\nu},\Twist_{\TheStrand\Strand}]=1$, 
\item \label{lem:PMap_gens_sat_rels_it2b}
$[\Twist_{\NNStrand\NStrand},\Twist_{\TheStrand\Strand}]=1$, \; $[\Twist_{\Strand\Str},\TwistP_{\TheStrand\NPct}]=1$, \;
$[\TwistP_{\Strand\Pc},\TwistP_{\TheStrand\NPct}]=1$, 
\\
$[\Twist_{\Strand\Str},\TwistC_{\TheStrand\nu}]=1$, \; 
$[\TwistP_{\Strand\NPct},\TwistC_{\TheStrand\nu}]=1 \; \text{ and } \; [\TwistC_{\Strand\mu},\TwistC_{\TheStrand\nu}]=1$, 
\item \label{lem:PMap_gens_sat_rels_it2c}
$[\Twist_{\TheStrand\NNStrand}\Twist_{\TheStrand\Strand}\Twist_{\TheStrand\NNStrand}^{-1},\Twist_{\NNStrand\Str}]=1$, \; $[\Twist_{\TheStrand\NNStrand}\Twist_{\TheStrand\Strand}\Twist_{\TheStrand\NNStrand}^{-1},\TwistP_{\NNStrand\NPct}]=1 \;$ \text{ and } 
\\ 
$[\Twist_{\TheStrand\NNStrand}\Twist_{\TheStrand\Strand}\Twist_{\TheStrand\NNStrand}^{-1},\TwistC_{\NNStrand\nu}]=1$, 
\item \label{lem:PMap_gens_sat_rels_it2d} 
$[\Twist_{\TheStrand\Strand}\TwistP_{\TheStrand\Pc}\Twist_{\TheStrand\Strand}^{-1},\TwistP_{\Strand\NPct}]=1 \; \text{ and } \; [\Twist_{\TheStrand\Strand}\TwistP_{\TheStrand\Pc}\Twist_{\TheStrand\Strand}^{-1},\TwistC_{\Strand\nu}]=1$, 
\item \label{lem:PMap_gens_sat_rels_it2e}
$[\Twist_{\TheStrand\Strand}\TwistC_{\TheStrand\mu}\Twist_{\TheStrand\Strand}^{-1},\TwistC_{\Strand\nu}]=1$. 
\end{enumerate}
\end{enumerate}
In particular, these relations imply: 
\begin{enumerate}[resume]
\item \label{lem:PMap_gens_sat_rels_it3}
\begin{enumerate}
\item \label{lem:PMap_gens_sat_rels_it3a}
$\Twist_{\NNStrand\Str}\Twist_{\TheStrand\Strand}\Twist_{\NNStrand\Str}^{-1} = \Twist_{\TheStrand\Str}^{-1}\Twist_{\TheStrand\NNStrand}^{-1}\Twist_{\TheStrand\Str}\Twist_{\TheStrand\NNStrand}\Twist_{\TheStrand\Strand}\Twist_{\TheStrand\NNStrand}^{-1}\Twist_{\TheStrand\Str}^{-1}\Twist_{\TheStrand\NNStrand}\Twist_{\TheStrand\Str}$, 
\item[\textup{a')}] \label{lem:PMap_gens_sat_rels_it3a_2}
$\Twist_{\NNStrand\Str}^{-1}\Twist_{\TheStrand\Strand}\Twist_{\NNStrand\Str} = \Twist_{\TheStrand\NNStrand}\Twist_{\TheStrand\Str}\Twist_{\TheStrand\NNStrand}^{-1}\Twist_{\TheStrand\Str}^{-1}\Twist_{\TheStrand\Strand}\Twist_{\TheStrand\Str}\Twist_{\TheStrand\NNStrand}\Twist_{\TheStrand\Str}^{-1}\Twist_{\TheStrand\NNStrand}^{-1}$, 
\item \label{lem:PMap_gens_sat_rels_it3b}
$\TwistP_{\NNStrand\NPct}\Twist_{\TheStrand\Strand}\TwistP_{\NNStrand\NPct}^{-1} = \TwistP_{\TheStrand\NPct}^{-1}\Twist_{\TheStrand\NNStrand}^{-1}\TwistP_{\TheStrand\NPct}\Twist_{\TheStrand\NNStrand}\Twist_{\TheStrand\Strand}\Twist_{\TheStrand\NNStrand}^{-1}\TwistP_{\TheStrand\NPct}^{-1}\Twist_{\TheStrand\NNStrand}\TwistP_{\TheStrand\NPct}$, 
\item[\textup{b')}] \label{lem:PMap_gens_sat_rels_it3b_2}
$\TwistP_{\NNStrand\NPct}^{-1}\Twist_{\TheStrand\Strand}\TwistP_{\NNStrand\NPct} = \Twist_{\TheStrand\NNStrand}\TwistP_{\TheStrand\NPct}\Twist_{\TheStrand\NNStrand}^{-1}\TwistP_{\TheStrand\NPct}^{-1}\Twist_{\TheStrand\Strand}\TwistP_{\TheStrand\NPct}\Twist_{\TheStrand\NNStrand}\TwistP_{\TheStrand\NPct}^{-1}\Twist_{\TheStrand\NNStrand}^{-1}$, 
\item \label{lem:PMap_gens_sat_rels_it3c}
$\TwistC_{\NNStrand\nu}\Twist_{\TheStrand\Strand}\TwistC_{\NNStrand\nu}^{-1} = \TwistC_{\TheStrand\nu}^{-1}\Twist_{\TheStrand\NNStrand}^{-1}\TwistC_{\TheStrand\nu}\Twist_{\TheStrand\NNStrand}\Twist_{\TheStrand\Strand}\Twist_{\TheStrand\NNStrand}^{-1}\TwistC_{\TheStrand\nu}^{-1}\Twist_{\TheStrand\NNStrand}\TwistC_{\TheStrand\nu}$, 
\item[\textup{c')}] \label{lem:PMap_gens_sat_rels_it3c_2}
$\TwistC_{\NNStrand\nu}^{-1}\Twist_{\TheStrand\Strand}\TwistC_{\NNStrand\nu} = \Twist_{\TheStrand\NNStrand}\TwistC_{\TheStrand\nu}\Twist_{\TheStrand\NNStrand}^{-1}\TwistC_{\TheStrand\nu}^{-1}\Twist_{\TheStrand\Strand}\TwistC_{\TheStrand\nu}\Twist_{\TheStrand\NNStrand}\TwistC_{\TheStrand\nu}^{-1}\Twist_{\TheStrand\NNStrand}^{-1}$, 
\item \label{lem:PMap_gens_sat_rels_it3d}
$\TwistP_{\Strand\NPct}\TwistP_{\TheStrand\Pc}\TwistP_{\Strand\NPct}^{-1} = \TwistP_{\TheStrand\NPct}^{-1}\Twist_{\TheStrand\Strand}^{-1}\TwistP_{\TheStrand\NPct}\Twist_{\TheStrand\Strand}\TwistP_{\TheStrand\Pc}\Twist_{\TheStrand\Strand}^{-1}\TwistP_{\TheStrand\NPct}^{-1}\Twist_{\TheStrand\Strand}\TwistP_{\TheStrand\NPct}$, 
\item[\textup{d')}] \label{lem:PMap_gens_sat_rels_it3d_2}
$\TwistP_{\Strand\NPct}^{-1}\TwistP_{\TheStrand\Pc}\TwistP_{\Strand\NPct} = \Twist_{\TheStrand\Strand}\TwistP_{\TheStrand\NPct}\Twist_{\TheStrand\Strand}^{-1}\TwistP_{\TheStrand\NPct}^{-1}\TwistP_{\TheStrand\Pc}\TwistP_{\TheStrand\NPct}\Twist_{\TheStrand\Strand}\TwistP_{\TheStrand\NPct}^{-1}\Twist_{\TheStrand\Strand}^{-1}$, 
\item \label{lem:PMap_gens_sat_rels_it3e}
$\TwistC_{\Strand\nu}\TwistP_{\TheStrand\NPct}\TwistC_{\Strand\nu}^{-1} = \TwistC_{\TheStrand\nu}^{-1}\Twist_{\TheStrand\Strand}^{-1}\TwistC_{\TheStrand\nu}\Twist_{\TheStrand\Strand}\TwistP_{\TheStrand\NPct}\Twist_{\TheStrand\Strand}^{-1}\TwistC_{\TheStrand\nu}^{-1}\Twist_{\TheStrand\Strand}\TwistC_{\TheStrand\nu}$, 
\item[\textup{e')}] \label{lem:PMap_gens_sat_rels_it3e_2}
$\TwistC_{\Strand\nu}^{-1}\TwistP_{\TheStrand\NPct}\TwistC_{\Strand\nu} = \Twist_{\TheStrand\Strand}\TwistC_{\TheStrand\nu}\Twist_{\TheStrand\Strand}^{-1}\TwistC_{\TheStrand\nu}^{-1}\TwistP_{\TheStrand\NPct}\TwistC_{\TheStrand\nu}\Twist_{\TheStrand\Strand}\TwistC_{\TheStrand\nu}^{-1}\Twist_{\TheStrand\Strand}^{-1}$,  
\item \label{lem:PMap_gens_sat_rels_it3f}
$\TwistC_{\Strand\nu}\TwistC_{\TheStrand\mu}\TwistC_{\Strand\nu}^{-1} = \TwistC_{\TheStrand\nu}^{-1}\Twist_{\TheStrand\Strand}^{-1}\TwistC_{\TheStrand\nu}\Twist_{\TheStrand\Strand}\TwistC_{\TheStrand\mu}\Twist_{\TheStrand\Strand}^{-1}\TwistC_{\TheStrand\nu}^{-1}\Twist_{\TheStrand\Strand}\TwistC_{\TheStrand\nu}$, 
\item[\textup{f')}] \label{lem:PMap_gens_sat_rels_it3f_2}
$\TwistC_{\Strand\nu}^{-1}\TwistC_{\TheStrand\mu}\TwistC_{\Strand\nu} = \Twist_{\TheStrand\Strand}\TwistC_{\TheStrand\nu}\Twist_{\TheStrand\Strand}^{-1}\TwistC_{\TheStrand\nu}^{-1}\TwistC_{\TheStrand\mu}\TwistC_{\TheStrand\nu}\Twist_{\TheStrand\Strand}\TwistC_{\TheStrand\nu}^{-1}\Twist_{\TheStrand\Strand}^{-1}$. 
\end{enumerate}
\end{enumerate}
\end{lemma}
\begin{proof}The relations in \ref{lem:PMap_gens_sat_rels_it1} are of the form $AHA^{-1}=\tilde{A}H\tilde{A}^{-1}$ with $H$ as in Lemma~\ref{lem:conj_manipulates_supp} and $A,\tilde{A}$ pure $\freeprod$-equivariant homeomorphisms. This is equivalent to $\tilde{A}^{-1}AHA^{-1}\tilde{A}=H$. Thus, by Lemma \ref{lem:conj_manipulates_supp}, it remains to observe that $\tilde{A}^{-1}A(\freeprod(D_H))=\freeprod(D_H)$ or equivalent $A(\freeprod(D_H))=\tilde{A}(\freeprod(D_H))$. This is elaborated in Figure \ref{fig:PMap_gens_sat_rels_conj} for the relations \ref{lem:PMap_gens_sat_rels_it1}\ref{lem:PMap_gens_sat_rels_it1a}-\ref{lem:PMap_gens_sat_rels_it1f}. The remaining relations \ref{lem:PMap_gens_sat_rels_it1}\hyperlink{lem:PMap_gens_sat_rels_it1a_2}{a')}-\hyperlink{lem:PMap_gens_sat_rels_it1f_2}{f')} follow from analogous pictures where we twist along the red and blue arrows, respectively, in the opposite direction. 
\begin{figure}[H]
\centerline{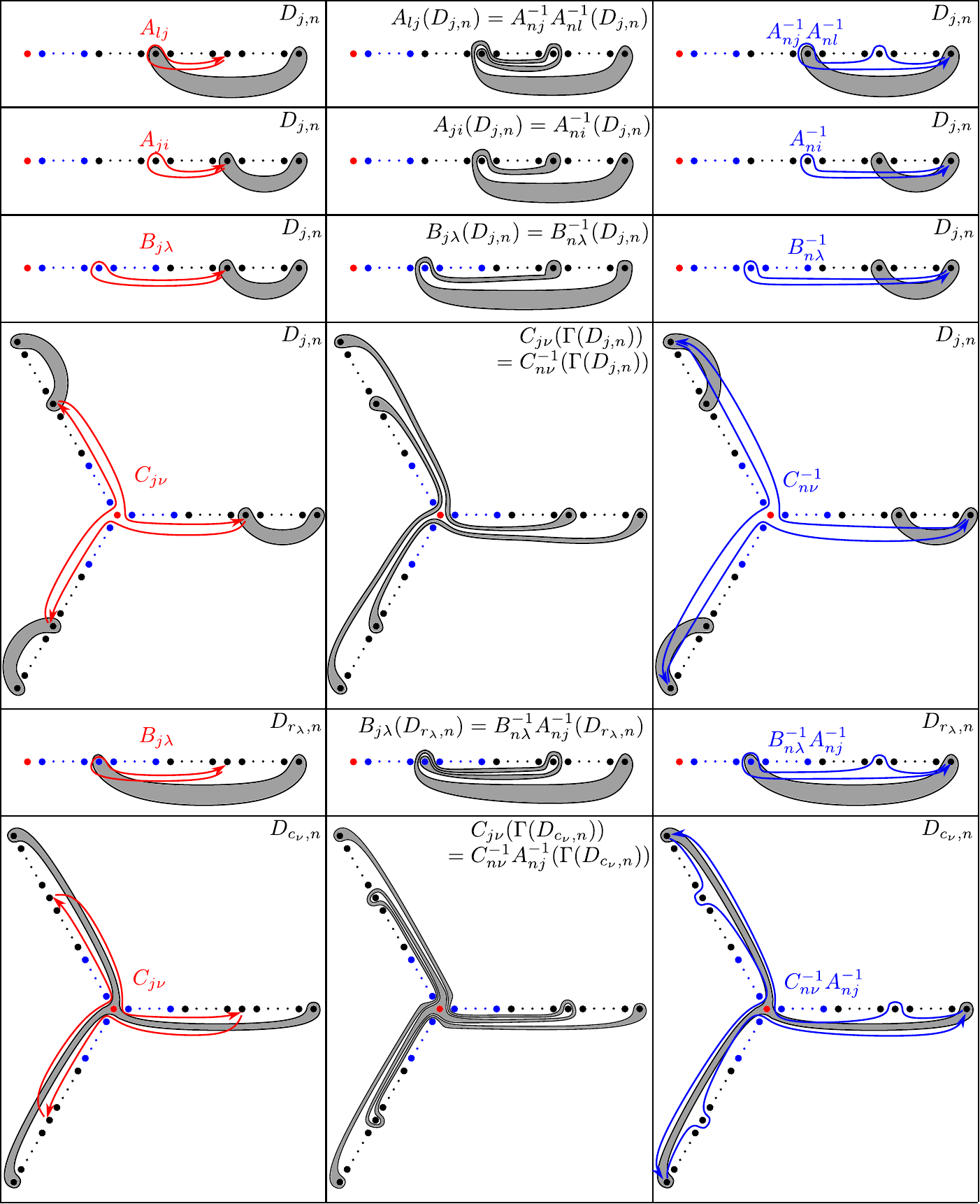}
\caption{Observation of the relations \ref{lem:PMap_gens_sat_rels_it1}\ref{lem:PMap_gens_sat_rels_it1a}-\ref{lem:PMap_gens_sat_rels_it1f} (from top to bottom) by consideration of the supporting disks.}
\label{fig:PMap_gens_sat_rels_conj}
\end{figure}

For the commutator relations in \ref{lem:PMap_gens_sat_rels_it2}, we check that the commuting mapping classes can be realized by homeomorphisms with disjoint support. \new{In the cases when two generators commute,} this follows directly from the definition of the generators on page \pageref{fig:disks_with_marked_pts}. \new{Further, in Figure \ref{fig:PMap_gens_sat_rels_comm},} we determine the supporting disks of 
$\Twist_{\TheStrand\NNStrand}\Twist_{\TheStrand\Strand}\Twist_{\TheStrand\NNStrand}^{-1}$, $\Twist_{\TheStrand\Strand}\TwistP_{\TheStrand\Pc}\Twist_{\TheStrand\Strand}^{-1}$ and $\Twist_{\TheStrand\Strand}\TwistC_{\TheStrand\mu}\Twist_{\TheStrand\Strand}^{-1}$.  \new{For each of these disks,} the $\freeprod$-orbit is disjoint from the support of the commuting generators. Thus, the commutator relations follow. In Figure \ref{fig:PMap_gens_sat_rels_comm} the support of exemplary commuting generators are depicted grayed out. 
\begin{figure}[H]
\centerline{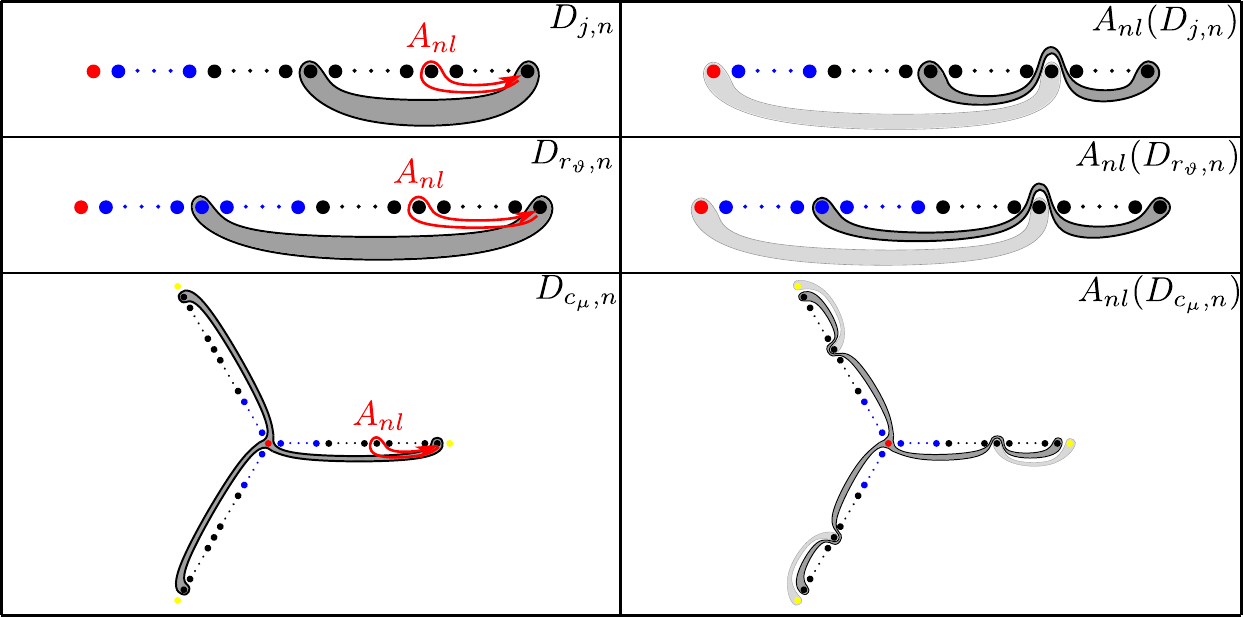}
\caption{Determining the supporting disks of $\Twist_{\TheStrand\NNStrand}\Twist_{\TheStrand\Strand}\Twist_{\TheStrand\NNStrand}^{-1}$, $\Twist_{\TheStrand\Strand}\TwistP_{\TheStrand\Pc}\Twist_{\TheStrand\Strand}^{-1}$ and $\Twist_{\TheStrand\Strand}\TwistC_{\TheStrand\mu}\Twist_{\TheStrand\Strand}^{-1}$.}
\label{fig:PMap_gens_sat_rels_comm}
\end{figure}

\new{For the additional conjugation relations,} we work out the first two examples. The remaining relations follow analogously. For this purpose, we emphasize that the relations from \ref{lem:PMap_gens_sat_rels_it1} imply:  
\begin{equation}
\label{lem:PMap_gens_sat_rels_eq}
\Twist_{\NNStrand\Str}\Twist_{\TheStrand\NNStrand}\Twist_{\TheStrand\Str}\mystackrel{\ref{lem:PMap_gens_sat_rels_it1}\hyperlink{lem:PMap_gens_sat_rels_it1_a2}{a')}}=\Twist_{\TheStrand\Str}\Twist_{\NNStrand\Str}\Twist_{\TheStrand\NNStrand}\mystackrel{\ref{lem:PMap_gens_sat_rels_it1}\ref{lem:PMap_gens_sat_rels_it1b}}=\Twist_{\TheStrand\NNStrand}\Twist_{\TheStrand\Str}\Twist_{\NNStrand\Str}. 
\end{equation}
Based on the commutator relation $[\Twist_{\TheStrand\NNStrand}\Twist_{\TheStrand\Strand}\Twist_{\TheStrand\NNStrand}^{-1},\Twist_{\NNStrand\Str}]=1$ from \ref{lem:PMap_gens_sat_rels_it2} and relation~(\ref{lem:PMap_gens_sat_rels_eq}), the first of the missing relations follows: 
\begin{align*}
& \Twist_{\NNStrand\Str}\Twist_{\TheStrand\NNStrand}\Twist_{\TheStrand\Strand}\textcolor{\short}{\Twist_{\TheStrand\NNStrand}^{-1}}=\Twist_{\TheStrand\NNStrand}\Twist_{\TheStrand\Strand}\Twist_{\TheStrand\NNStrand}^{-1}\Twist_{\NNStrand\Str} && \vert\;\Twist_{\TheStrand\Str}\cdot\;\cdot \Twist_{\TheStrand\NNStrand}
\\
\mystackrel{}\Leftrightarrow & \Twist_{\TheStrand\Str}\Twist_{\NNStrand\Str}\Twist_{\TheStrand\NNStrand}\Twist_{\TheStrand\Strand}=\Twist_{\TheStrand\Str}\Twist_{\TheStrand\NNStrand}\Twist_{\TheStrand\Strand}\Twist_{\TheStrand\NNStrand}^{-1}\Twist_{\NNStrand\Str}\Twist_{\TheStrand\NNStrand} &&
\\
\mystackrel{$\vee$}\Leftrightarrow & \textcolor{\col}{\Twist_{\TheStrand\Str}\Twist_{\NNStrand\Str}\Twist_{\TheStrand\NNStrand}}\Twist_{\TheStrand\Strand}=\Twist_{\TheStrand\Str}\Twist_{\TheStrand\NNStrand}\Twist_{\TheStrand\Strand}\Twist_{\TheStrand\NNStrand}^{-1}(\Twist_{\TheStrand\Str}^{-1}\textcolor{\col}{\Twist_{\TheStrand\Str})\Twist_{\NNStrand\Str}\Twist_{\TheStrand\NNStrand}} && 
\\
\mystackrel{(\ref{lem:PMap_gens_sat_rels_eq})}
\Leftrightarrow & \Twist_{\TheStrand\NNStrand}\Twist_{\TheStrand\Str}\Twist_{\NNStrand\Str}\Twist_{\TheStrand\Strand}=\Twist_{\TheStrand\Str}\Twist_{\TheStrand\NNStrand}\Twist_{\TheStrand\Strand}\Twist_{\TheStrand\NNStrand}^{-1}\Twist_{\TheStrand\Str}^{-1}\Twist_{\TheStrand\NNStrand}\Twist_{\TheStrand\Str}\Twist_{\NNStrand\Str} && \vert\;\Twist_{\TheStrand\Str}^{-1}\Twist_{\TheStrand\NNStrand}^{-1}\cdot\;\cdot \Twist_{\NNStrand\Str}^{-1}
\\
\mystackrel{}\Leftrightarrow & \Twist_{\NNStrand\Str}\Twist_{\TheStrand\Strand}\Twist_{\NNStrand\Str}^{-1}=\Twist_{\TheStrand\Str}^{-1}\Twist_{\TheStrand\NNStrand}^{-1}\Twist_{\TheStrand\Str}\Twist_{\TheStrand\NNStrand}\Twist_{\TheStrand\Strand}\Twist_{\TheStrand\NNStrand}^{-1}\Twist_{\TheStrand\Str}^{-1}\Twist_{\TheStrand\NNStrand}\Twist_{\TheStrand\Str}. && 
\end{align*}
From the last equation we also obtain the second conjugation relation: 
\begin{align*}
& \Twist_{\TheStrand\Str}^{-1}\Twist_{\TheStrand\NNStrand}^{-1}\Twist_{\TheStrand\Str}\Twist_{\TheStrand\NNStrand}\Twist_{\TheStrand\Strand}\Twist_{\TheStrand\NNStrand}^{-1}\Twist_{\TheStrand\Str}^{-1}\Twist_{\TheStrand\NNStrand}\Twist_{\TheStrand\Str}=\Twist_{\NNStrand\Str}\Twist_{\TheStrand\Strand}\Twist_{\NNStrand\Str}^{-1} \quad \vert\;\Twist_{\TheStrand\NNStrand}^{-1}\Twist_{\TheStrand\Str}^{-1}\Twist_{\TheStrand\NNStrand}\Twist_{\TheStrand\Str}\cdot\;\cdot \Twist_{\TheStrand\Str}^{-1}\Twist_{\TheStrand\NNStrand}^{-1}\Twist_{\TheStrand\Str}\Twist_{\TheStrand\NNStrand}
\\
\mystackrel{}\Leftrightarrow & \Twist_{\TheStrand\Strand}=\Twist_{\TheStrand\NNStrand}^{-1}\Twist_{\TheStrand\Str}^{-1}\textcolor{\col}{\Twist_{\TheStrand\NNStrand}\Twist_{\TheStrand\Str}\Twist_{\NNStrand\Str}}\Twist_{\TheStrand\Strand}\textcolor{\col}{\Twist_{\NNStrand\Str}^{-1}\Twist_{\TheStrand\Str}^{-1}\Twist_{\TheStrand\NNStrand}^{-1}}\Twist_{\TheStrand\Str}\Twist_{\TheStrand\NNStrand} 
\\
\mystackrel{(\ref{lem:PMap_gens_sat_rels_eq})}
\Leftrightarrow & \Twist_{\TheStrand\Strand}=\Twist_{\TheStrand\NNStrand}^{-1}\Twist_{\TheStrand\Str}^{-1}\Twist_{\NNStrand\Str}\Twist_{\TheStrand\NNStrand}\Twist_{\TheStrand\Str}\Twist_{\TheStrand\Strand}\Twist_{\TheStrand\Str}^{-1}\Twist_{\TheStrand\NNStrand}^{-1}\Twist_{\NNStrand\Str}^{-1}\Twist_{\TheStrand\Str}\Twist_{\TheStrand\NNStrand} 
\\
\mystackrel{$\vee$}\Leftrightarrow & \Twist_{\TheStrand\Strand}=\textcolor{\short}{(\Twist_{\NNStrand\Str}}\Twist_{\NNStrand\Str}^{-1})\Twist_{\TheStrand\NNStrand}^{-1}\Twist_{\TheStrand\Str}^{-1}\Twist_{\NNStrand\Str}\Twist_{\TheStrand\NNStrand}\Twist_{\TheStrand\Str}\Twist_{\TheStrand\Strand}\Twist_{\TheStrand\Str}^{-1}\Twist_{\TheStrand\NNStrand}^{-1}\Twist_{\NNStrand\Str}^{-1}\Twist_{\TheStrand\Str}\Twist_{\TheStrand\NNStrand}(\Twist_{\NNStrand\Str}\textcolor{\short}{\Twist_{\NNStrand\Str}^{-1})} \quad \vert\; \Twist_{\NNStrand\Str}^{-1}\cdot\;\cdot \Twist_{\NNStrand\Str} 
\\
\mystackrel{}\Leftrightarrow & \Twist_{\NNStrand\Str}^{-1}\Twist_{\TheStrand\Strand}\Twist_{\NNStrand\Str}=\textcolor{\col}{\Twist_{\NNStrand\Str}^{-1}\Twist_{\TheStrand\NNStrand}^{-1}\Twist_{\TheStrand\Str}^{-1}\Twist_{\NNStrand\Str}}\Twist_{\TheStrand\NNStrand}\Twist_{\TheStrand\Str}\Twist_{\TheStrand\Strand}\Twist_{\TheStrand\Str}^{-1}\Twist_{\TheStrand\NNStrand}^{-1}\textcolor{\col}{\Twist_{\NNStrand\Str}^{-1}\Twist_{\TheStrand\Str}\Twist_{\TheStrand\NNStrand}\Twist_{\NNStrand\Str}} 
\\
\mystackrel{\ref{lem:PMap_gens_sat_rels_it1}\hyperlink{lem:PMap_gens_sat_rels_it1a2}{a')},\hyperlink{lem:PMap_gens_sat_rels_it1b2}{b')}}\Leftrightarrow & \Twist_{\NNStrand\Str}^{-1}\Twist_{\TheStrand\Strand}\Twist_{\NNStrand\Str}=\Twist_{\TheStrand\NNStrand}\Twist_{\TheStrand\Str}\Twist_{\TheStrand\NNStrand}^{-1}\Twist_{\TheStrand\Str}^{-1}\textcolor{\short}{\Twist_{\TheStrand\NNStrand}^{-1}\Twist_{\TheStrand\NNStrand}\Twist_{\TheStrand\Str}^{-1}\Twist_{\TheStrand\NNStrand}^{-1}\Twist_{\TheStrand\NNStrand}\Twist_{\TheStrand\Str}}\Twist_{\TheStrand\Strand} 
\\
& \textcolor{\short}{\Twist_{\TheStrand\Str}^{-1}\Twist_{\TheStrand\NNStrand}^{-1}\Twist_{\TheStrand\NNStrand}\Twist_{\TheStrand\Str}\Twist_{\TheStrand\NNStrand}^{-1}\Twist_{\TheStrand\NNStrand}}\Twist_{\TheStrand\Str}\Twist_{\TheStrand\NNStrand}\Twist_{\TheStrand\Str}^{-1}\Twist_{\TheStrand\NNStrand}^{-1} 
\\
\mystackrel{}\Leftrightarrow & \Twist_{\NNStrand\Str}^{-1}\Twist_{\TheStrand\Strand}\Twist_{\NNStrand\Str}=\Twist_{\TheStrand\NNStrand}\Twist_{\TheStrand\Str}\Twist_{\TheStrand\NNStrand}^{-1}\Twist_{\TheStrand\Str}^{-1}\Twist_{\TheStrand\Strand}\Twist_{\TheStrand\Str}\Twist_{\TheStrand\NNStrand}\Twist_{\TheStrand\Str}^{-1}\Twist_{\TheStrand\NNStrand}^{-1}. 
\end{align*}
\end{proof}

Now the semidirect product structure of $\PMapIdOrb{\TheStrand}{\Sigma_\freeprod(\ThePct)}$ allows us to deduce a finite presentation in terms of the above generators. 

\begin{corollary}
\label{cor:pres_PMap_free_prod}
The pure mapping class group $\PMapIdOrb{\TheStrand}{\Sigma_\freeprod(\ThePct)}$ has a presentation with generators 
\[
\Twist_{\Strand\Str}, \TwistP_{\NStrand\NPct} \; \text{ and } \; \TwistC_{\NStrand\nu}, 
\]
for $1\leq\Str,\Strand,\NStrand\leq\TheStrand$ with $\Str<\Strand$, $1\leq\NPct\leq\ThePct$ and $1\leq\nu\leq\TheCone$ and the following defining relations for $1\leq\Str,\Strand,\NStrand,\NNStrand\leq\TheStrand$ with $\Str<\Strand<\NStrand<\NNStrand$, $1\leq\Pc,\NPct\leq\ThePct$ with $\Pc<\NPct$ and $1\leq\mu,\nu\leq\TheCone$ with $\mu<\nu$: 
\begin{enumerate}
\item 
\label{cor:pres_PMap_free_prod_rel1}
$[\Twist_{\Strand\Str},\Twist_{\NNStrand\NStrand}]=1$, 
$[\TwistP_{\Strand\NPct},\Twist_{\NNStrand\NStrand}]=1 \; \text{ and } \; 
[\TwistC_{\Strand\nu},\Twist_{\NNStrand\NStrand}]=1$, 
\item 
\label{cor:pres_PMap_free_prod_rel2}
$[\Twist_{\NNStrand\Str},\Twist_{\NStrand\Strand}]=1$, 
$[\TwistP_{\NNStrand\NPct},\Twist_{\NStrand\Strand}]=1$, 
$[\TwistP_{\NNStrand\NPct},\TwistP_{\NStrand\Pc}]=1$, 
$[\TwistC_{\NNStrand\nu},\Twist_{\NStrand\Strand}]=1$, 
\\
$[\TwistC_{\NNStrand\nu},\TwistP_{\NStrand\NPct}]=1 \; \text{ and } \; 
[\TwistC_{\NNStrand\nu},\TwistC_{\NStrand\mu}]=1$, 
\item 
\label{cor:pres_PMap_free_prod_rel3}
$[\Twist_{\NNStrand\NStrand}\Twist_{\NNStrand\Strand}\Twist_{\NNStrand\NStrand}^{-1},\Twist_{\NStrand\Str}]=1$, 
$[\Twist_{\NStrand\Strand}\Twist_{\NStrand\Str}\Twist_{\NStrand\Strand}^{-1},\TwistP_{\Strand\NPct}]=1$, 
$[\Twist_{\NStrand\Strand}\TwistP_{\NStrand\Pc}\Twist_{\NStrand\Strand}^{-1},\TwistP_{\Strand\NPct}]=1$, 
$[\Twist_{\NStrand\Strand}\Twist_{\NStrand\Str}\Twist_{\NStrand\Strand}^{-1},\TwistC_{\Strand\nu}]=1$, 
$[\Twist_{\NStrand\Strand}\TwistC_{\NStrand\mu}\Twist_{\NStrand\Strand}^{-1},\TwistC_{\Strand\nu}]=1 \; \text{ and } \; 
[\Twist_{\NStrand\Strand}\TwistP_{\NStrand\NPct}\Twist_{\NStrand\Strand}^{-1},\TwistC_{\Strand\nu}]=~1$, 
\item 
\label{cor:pres_PMap_free_prod_rel4}
$\Twist_{\Strand\Str}\Twist_{\NStrand\Strand}\Twist_{\NStrand\Str}=\Twist_{\NStrand\Str}\Twist_{\Strand\Str}\Twist_{\NStrand\Strand}=\Twist_{\NStrand\Strand}\Twist_{\NStrand\Str}\Twist_{\Strand\Str}$, 
\\
$\Twist_{\Strand\Str}\TwistP_{\Strand\NPct}\TwistP_{\Str\NPct}=\TwistP_{\Str\NPct}\Twist_{\Strand\Str}\TwistP_{\Strand\NPct}=\TwistP_{\Strand\NPct}\TwistP_{\Str\NPct}\Twist_{\Strand\Str}$ and 
\\
$\Twist_{\Strand\Str}\TwistC_{\Strand\nu}\TwistC_{\Str\nu}=\TwistC_{\Str\nu}\Twist_{\Strand\Str}\TwistC_{\Strand\nu}=\TwistC_{\Strand\nu}\TwistC_{\Str\nu}\Twist_{\Strand\Str}$. 
\end{enumerate}
\end{corollary}
\begin{proof}
We prove the claimed presentation of $\PMapIdOrb{\TheStrand}{\Sigma_\freeprod(\ThePct)}$ by induction on~$\TheStrand$. 

For $\TheStrand=0$ and $\TheStrand=1$, the same arguments as in the proof of Corollary \ref{cor:PMap_gen_set} show that $\PMapIdOrb{0}{\Sigma_\freeprod(\ThePct)}$ is trivial and $\PMapIdOrb{1}{\Sigma_\freeprod(\ThePct)}$ is a free group of rank $\ThePct+\TheCone$ with basis elements $\TwistP_{1\NPct},\TwistC_{1\nu}$ for $1\leq\NPct\leq\ThePct$ and $1\leq\nu\leq\TheCone$. That is, $\PMapIdOrb{\TheStrand}{\Sigma_\freeprod(\ThePct)}$ has the above presentation for $\TheStrand=0$ and $\TheStrand=1$. 

By induction hypothesis, we assume that $\PMapIdOrb{\TheStrand-1}{\Sigma_\freeprod(\ThePct)}$ has a presentation as claimed above. Moreover, we recall that the group $\PMapIdOrb{\TheStrand}{\Sigma_\freeprod(\ThePct)}$ by Corollary~\ref{cor:pure_orb_mcg_ses} has a semidirect product structure $\freegrp{\TheStrand-1+\ThePct+\TheCone}\rtimes\PMapIdOrb{\TheStrand-1}{\Sigma_\freeprod(\ThePct)}$. Invoking Lemma \ref{lem:semidir_prod_pres}, we obtain a presentation by generators $\Twist_{\Strand\Str}$ for $1\leq\Str<\Strand\leq\TheStrand$ and $\TwistC_{\NStrand\nu},\TwistP_{\NStrand\NPct}$ for $1\leq\NStrand\leq\TheStrand, 1\leq\nu\leq\TheCone, 1\leq\NPct\leq\ThePct$ with the following relations: Since $\freegrp{\TheStrand-1+\ThePct+\TheCone}$ is a free group, the set of relations named $R$ in Lemma \ref{lem:semidir_prod_pres} is empty in this case. From $\PMapIdOrb{\TheStrand-1}{\Sigma_\freeprod(\ThePct)}$ the group $\PMapIdOrb{\TheStrand}{\Sigma_\freeprod(\ThePct)}$ inherits the following relations for $1\leq\Str,\Strand,\NStrand,\NNStrand<\TheStrand$ with $ \Str<\Strand<\NStrand<\NNStrand$, $1\leq\Pc,\NPct\leq\ThePct$ with $\Pc<\NPct$ and $1\leq\mu,\nu\leq\TheCone$ with $\mu<\nu$ named $S$ in Lemma \ref{lem:semidir_prod_pres}: 
\begin{itemize}
\item 
$[\Twist_{\Strand\Str},\Twist_{\NNStrand\NStrand}]=1$, 
$[\TwistP_{\Strand\NPct},\Twist_{\NNStrand\NStrand}]=1 \; \text{ and } \; 
[\TwistC_{\Strand\nu},\Twist_{\NNStrand\NStrand}]=1$, 
\item 
$[\Twist_{\NNStrand\Str},\Twist_{\NStrand\Strand}]=1$, 
$[\TwistP_{\NNStrand\NPct},\Twist_{\NStrand\Strand}]=1$, 
$[\TwistP_{\NNStrand\NPct},\TwistP_{\NStrand\Pc}]=1$, 
$[\TwistC_{\NNStrand\nu},\Twist_{\NStrand\Strand}]=1$, 
\\
$[\TwistC_{\NNStrand\nu},\TwistP_{\NStrand\NPct}]=1 \; \text{ and } \; 
[\TwistC_{\NNStrand\nu},\TwistC_{\NStrand\mu}]=1$, 
\item 
$[\Twist_{\NNStrand\NStrand}\Twist_{\NNStrand\Strand}\Twist_{\NNStrand\NStrand}^{-1},\Twist_{\NStrand\Str}]=1$, 
$[\Twist_{\NStrand\Strand}\Twist_{\NStrand\Str}\Twist_{\NStrand\Strand}^{-1},\TwistP_{\Strand\NPct}]=1$, 
$[\Twist_{\NStrand\Strand}\TwistP_{\NStrand\Pc}\Twist_{\NStrand\Strand}^{-1},\TwistP_{\Strand\NPct}]=1$, 
$[\Twist_{\NStrand\Strand}\Twist_{\NStrand\Str}\Twist_{\NStrand\Strand}^{-1},\TwistC_{\Strand\nu}]=1$, 
$[\Twist_{\NStrand\Strand}\TwistC_{\NStrand\mu}\Twist_{\NStrand\Strand}^{-1},\TwistC_{\Strand\nu}]=1 \; \text{ and } \; 
[\Twist_{\NStrand\Strand}\TwistP_{\NStrand\NPct}\Twist_{\NStrand\Strand}^{-1},\TwistC_{\Strand\nu}]=1$, 
\item 
$\Twist_{\Strand\Str}\Twist_{\NStrand\Strand}\Twist_{\NStrand\Str}=\Twist_{\NStrand\Str}\Twist_{\Strand\Str}\Twist_{\NStrand\Strand}=\Twist_{\NStrand\Strand}\Twist_{\NStrand\Str}\Twist_{\Strand\Str}$, 
\\
$\Twist_{\Strand\Str}\TwistP_{\Strand\NPct}\TwistP_{\Str\NPct}=\TwistP_{\Str\NPct}\Twist_{\Strand\Str}\TwistP_{\Strand\NPct}=\TwistP_{\Strand\NPct}\TwistP_{\Str\NPct}\Twist_{\Strand\Str}\;$ and 
\\
$\Twist_{\Strand\Str}\TwistC_{\Strand\nu}\TwistC_{\Str\nu}=\TwistC_{\Str\nu}\Twist_{\Strand\Str}\TwistC_{\Strand\nu}=\TwistC_{\Strand\nu}\TwistC_{\Str\nu}\Twist_{\Strand\Str}$. 
\end{itemize}
Additionally, the generators satisfy the conjugation relations from Lemma \ref{lem:PMap_gens_sat_rels}. Below we list them in an order, that allows us to observe that these relations guarantee that $\PushPMap^{orb}$ embeds $\freegrp{\TheStrand-1+\ThePct+\TheCone}$ as a normal subgroup, i.e.\ the conjugation relations and the relations from $\PMapIdOrb{\TheStrand-1}{\Sigma_\freeprod(\ThePct)}$ yield a presentation as in Lemma~\ref{lem:semidir_prod_pres_it2}. 

It remains to show that this presentation is equivalent to the claimed presentation. Therefore, we explain how the conjugation relations fit into the presentation of $\PMapIdOrb{\TheStrand}{\Sigma_\freeprod(\ThePct)}$ claimed above. This requires the observations described below. \new{In (\ref{cor:pres_PMap_free_prod_eq1}),} we use: 
\begin{align*}
& \Twist_{\NNStrand\Strand}\Twist_{\TheStrand\Strand}\textcolor{\short}{\Twist_{\NNStrand\Strand}^{-1}}= \textcolor{\short}{\Twist_{\TheStrand\Strand}^{-1}\Twist_{\TheStrand\NNStrand}^{-1}}\Twist_{\TheStrand\Strand}\Twist_{\TheStrand\NNStrand}\Twist_{\TheStrand\Strand} && \vert \Twist_{\TheStrand\NNStrand}\Twist_{\TheStrand\Strand}\cdot\;\cdot \Twist_{\NNStrand\Strand}
\\
\mystackrel{}\Leftrightarrow & \textcolor{\col}{\Twist_{\TheStrand\NNStrand}\Twist_{\TheStrand\Strand}\Twist_{\NNStrand\Strand}}\Twist_{\TheStrand\Strand}=\Twist_{\TheStrand\Strand}\textcolor{\col}{\Twist_{\TheStrand\NNStrand}\Twist_{\TheStrand\Strand}\Twist_{\NNStrand\Strand}} && 
\\
\mystackrel{(\ref{cor:pres_PMap_free_prod_eq2_1})}\Leftrightarrow & \textcolor{\short}{\Twist_{\TheStrand\Strand}}\Twist_{\NNStrand\Strand}\Twist_{\TheStrand\NNStrand}\Twist_{\TheStrand\Strand}=\textcolor{\short}{\Twist_{\TheStrand\Strand}}\Twist_{\TheStrand\Strand}\Twist_{\NNStrand\Strand}\Twist_{\TheStrand\NNStrand} && \vert\; \Twist_{\TheStrand\Strand}^{-1}\cdot 
\\
\mystackrel{}\Leftrightarrow & \Twist_{\NNStrand\Strand}\Twist_{\TheStrand\NNStrand}\Twist_{\TheStrand\Strand}=\Twist_{\TheStrand\Strand}\Twist_{\NNStrand\Strand}\Twist_{\TheStrand\NNStrand} \text{ from } \ref{cor:pres_PMap_free_prod_rel4}. && 
\end{align*}

The observation in (\ref{cor:pres_PMap_free_prod_eq2}) is immediate. \new{For (\ref{cor:pres_PMap_free_prod_eq3}),} we observe: 
\begin{align*}
& \Twist_{\NNStrand\Str}\Twist_{\TheStrand\Strand}\textcolor{\short}{\Twist_{\NNStrand\Str}^{-1}}=\textcolor{\short}{\Twist_{\TheStrand\Str}^{-1}\Twist_{\TheStrand\NNStrand}^{-1}}\Twist_{\TheStrand\Str}\Twist_{\TheStrand\NNStrand}\Twist_{\TheStrand\Strand}\Twist_{\TheStrand\NNStrand}^{-1}\Twist_{\TheStrand\Str}^{-1}\Twist_{\TheStrand\NNStrand}\Twist_{\TheStrand\Str} && \vert\; \Twist_{\TheStrand\NNStrand}\Twist_{\TheStrand\Str}\cdot\;\cdot \Twist_{\NNStrand\Str}
\\
\mystackrel{}\Leftrightarrow & \textcolor{\col}{\Twist_{\TheStrand\NNStrand}\Twist_{\TheStrand\Str}\Twist_{\NNStrand\Str}}\Twist_{\TheStrand\Strand}=\Twist_{\TheStrand\Str}\Twist_{\TheStrand\NNStrand}\Twist_{\TheStrand\Strand}\Twist_{\TheStrand\NNStrand}^{-1}\Twist_{\TheStrand\Str}^{-1}\textcolor{\col}{\Twist_{\TheStrand\NNStrand}\Twist_{\TheStrand\Str}\Twist_{\NNStrand\Str}} && 
\\
\mystackrel{(\ref{cor:pres_PMap_free_prod_eq2_1})}\Leftrightarrow & \textcolor{\short}{\Twist_{\TheStrand\Str}}\Twist_{\NNStrand\Str}\Twist_{\TheStrand\NNStrand}\Twist_{\TheStrand\Strand}=\textcolor{\short}{\Twist_{\TheStrand\Str}}\Twist_{\TheStrand\NNStrand}\Twist_{\TheStrand\Strand}\Twist_{\TheStrand\NNStrand}^{-1}\textcolor{\short}{\Twist_{\TheStrand\Str}^{-1}\Twist_{\TheStrand\Str}}\Twist_{\NNStrand\Str}\Twist_{\TheStrand\NNStrand} && \vert\; \Twist_{\TheStrand\Str}^{-1}\cdot 
\\
\mystackrel{}\Leftrightarrow & \Twist_{\NNStrand\Str}\Twist_{\TheStrand\NNStrand}\Twist_{\TheStrand\Strand}=\Twist_{\TheStrand\NNStrand}\Twist_{\TheStrand\Strand}\Twist_{\TheStrand\NNStrand}^{-1}\Twist_{\NNStrand\Str}\textcolor{\short}{\Twist_{\TheStrand\NNStrand}} && \vert\;\cdot \Twist_{\TheStrand\NNStrand}^{-1} 
\\
\mystackrel{}\Leftrightarrow & \Twist_{\NNStrand\Str}\Twist_{\TheStrand\NNStrand}\Twist_{\TheStrand\Strand}\Twist_{\TheStrand\NNStrand}^{-1}=\Twist_{\TheStrand\NNStrand}\Twist_{\TheStrand\Strand}\Twist_{\TheStrand\NNStrand}^{-1}\Twist_{\NNStrand\Str} \text{ from } \ref{cor:pres_PMap_free_prod_rel3}. && 
\end{align*}
\new{To deduce (\ref{cor:pres_PMap_free_prod_eq4}),} we additionally observe 
\begin{equation}
\label{cor:pres_PMap_free_prod_eq_cyc_permut}
\Twist_{\NNStrand\Strand}\Twist_{\TheStrand\NNStrand}\Twist_{\TheStrand\Strand}\mystackrel{(\ref{cor:pres_PMap_free_prod_eq1})}=\Twist_{\TheStrand\Strand}\Twist_{\NNStrand\Strand}\Twist_{\TheStrand\NNStrand}\mystackrel{(\ref{cor:pres_PMap_free_prod_eq2_1})}=\Twist_{\TheStrand\NNStrand}\Twist_{\TheStrand\Strand}\Twist_{\NNStrand\Strand}. 
\end{equation}
This implies: 
\begin{align*}
& \textcolor{\short}{\Twist_{\NNStrand\Str}^{-1}}\Twist_{\TheStrand\Strand}\textcolor{\short}{\Twist_{\NNStrand\Str}} & \mystackrel{}= & \Twist_{\TheStrand\NNStrand}\Twist_{\TheStrand\Str}\Twist_{\TheStrand\NNStrand}^{-1}\Twist_{\TheStrand\Str}^{-1}\Twist_{\TheStrand\Strand}\Twist_{\TheStrand\Str}\Twist_{\TheStrand\NNStrand}\Twist_{\TheStrand\Str}^{-1}\Twist_{\TheStrand\NNStrand}^{-1} \hspace*{16mm} \vert\;\Twist_{\NNStrand\Str}\cdot\;\cdot \Twist_{\NNStrand\Str}^{-1}
\\
\mystackrel{}\Leftrightarrow & \Twist_{\TheStrand\Strand} & \mystackrel{}= & \textcolor{\col}{\Twist_{\NNStrand\Str}\Twist_{\TheStrand\NNStrand}\Twist_{\TheStrand\Str}}\Twist_{\TheStrand\NNStrand}^{-1}\Twist_{\TheStrand\Str}^{-1}\Twist_{\TheStrand\Strand}\Twist_{\TheStrand\Str}\Twist_{\TheStrand\NNStrand}\textcolor{\col}{\Twist_{\TheStrand\Str}^{-1}\Twist_{\TheStrand\NNStrand}^{-1}\Twist_{\NNStrand\Str}^{-1}} 
\\
& & \mystackrel{(\ref{cor:pres_PMap_free_prod_eq_cyc_permut})}= & \Twist_{\TheStrand\NNStrand}\Twist_{\TheStrand\Str}\Twist_{\NNStrand\Str}\Twist_{\TheStrand\NNStrand}^{-1}\Twist_{\TheStrand\Str}^{-1}\Twist_{\TheStrand\Strand}\Twist_{\TheStrand\Str}\Twist_{\TheStrand\NNStrand}\Twist_{\NNStrand\Str}^{-1}\Twist_{\TheStrand\Str}^{-1}\Twist_{\TheStrand\NNStrand}^{-1}  
\\
& & \mystackrel{$\vee$}= & \Twist_{\TheStrand\NNStrand}\Twist_{\TheStrand\Str}\textcolor{\col}{\Twist_{\NNStrand\Str}\Twist_{\TheStrand\NNStrand}^{-1}\Twist_{\TheStrand\Str}^{-1}(\Twist_{\NNStrand\Str}^{-1}}\Twist_{\NNStrand\Str})\Twist_{\TheStrand\Strand}
\\
& & & (\Twist_{\NNStrand\Str}^{-1}\textcolor{\col}{\Twist_{\NNStrand\Str})\Twist_{\TheStrand\Str}\Twist_{\TheStrand\NNStrand}\Twist_{\NNStrand\Str}^{-1}}\Twist_{\TheStrand\Str}^{-1}\Twist_{\TheStrand\NNStrand}^{-1} 
\\
& & \mystackrel{(\ref{cor:pres_PMap_free_prod_eq1},\ref{cor:pres_PMap_free_prod_eq2_1})}= & \textcolor{\short}{\Twist_{\TheStrand\NNStrand}\Twist_{\TheStrand\Str}\Twist_{\TheStrand\Str}^{-1}\Twist_{\TheStrand\NNStrand}^{-1}\Twist_{\TheStrand\Str}\Twist_{\TheStrand\Str}^{-1}}\Twist_{\TheStrand\NNStrand}^{-1}\Twist_{\TheStrand\Str}^{-1}\Twist_{\TheStrand\NNStrand}\Twist_{\TheStrand\Str}\Twist_{\NNStrand\Str}\Twist_{\TheStrand\Strand}
\\
&  &  & \Twist_{\NNStrand\Str}^{-1}\Twist_{\TheStrand\Str}^{-1}\Twist_{\TheStrand\NNStrand}^{-1}\Twist_{\TheStrand\Str}\Twist_{\TheStrand\NNStrand}\textcolor{\short}{\Twist_{\TheStrand\Str}\Twist_{\TheStrand\Str}^{-1}\Twist_{\TheStrand\NNStrand}\Twist_{\TheStrand\Str}\Twist_{\TheStrand\Str}^{-1}\Twist_{\TheStrand\NNStrand}^{-1}} 
\\
& & \mystackrel{}= & \Twist_{\TheStrand\NNStrand}^{-1}\Twist_{\TheStrand\Str}^{-1}\Twist_{\TheStrand\NNStrand}\Twist_{\TheStrand\Str}\Twist_{\NNStrand\Str}\Twist_{\TheStrand\Strand}\Twist_{\NNStrand\Str}^{-1}\Twist_{\TheStrand\Str}^{-1}\Twist_{\TheStrand\NNStrand}^{-1}\Twist_{\TheStrand\Str}\Twist_{\TheStrand\NNStrand}  
\\
\mystackrel{}\Leftrightarrow & \Twist_{\NNStrand\Str}\Twist_{\TheStrand\Strand}\Twist_{\NNStrand\Str}^{-1} &  \mystackrel{}= & \Twist_{\TheStrand\Str}^{-1}\Twist_{\TheStrand\NNStrand}^{-1}\Twist_{\TheStrand\Str}\Twist_{\TheStrand\NNStrand}\Twist_{\TheStrand\Strand}\Twist_{\TheStrand\NNStrand}^{-1}\Twist_{\TheStrand\Str}^{-1}\Twist_{\TheStrand\NNStrand}\Twist_{\TheStrand\Str}
\\
\mystackrel{(\ref{cor:pres_PMap_free_prod_eq3})}\Leftrightarrow & [\Twist_{\TheStrand\NNStrand}\Twist_{\TheStrand\Strand}\Twist_{\TheStrand\NNStrand}^{-1},\Twist_{\NNStrand\Str}] & \mystackrel{}= & 1. 
\end{align*}

Using these and further similar observations, the conjugation relations fit into the above presentation. \new{In the list below,} an equivalence numbered by (x-y) relies on an observation analogous to the observation for (x) described above. Let 
\linebreak 
$1\leq\Strr,\Str,\Strand,\NStrand,\NNStrand<\TheStrand$ with $\Strr<\Str<\Strand<\NStrand<\NNStrand$, \;$1\leq\Pc,\Pct,\NPct\leq\ThePct$ with $\Pc<\Pct<\NPct$ and $1\leq\mu,\nu,\omicron\leq\TheCone$ with $\mu<\nu<\omicron$. 
\begin{align*}
\Twist_{\NNStrand\NStrand}\Twist_{\TheStrand\Strand}\Twist_{\NNStrand\NStrand}^{-1}=\Twist_{\NNStrand\NStrand}^{-1}\Twist_{\TheStrand\Strand}\Twist_{\NNStrand\NStrand}=\Twist_{\TheStrand\Strand} & \mystackrel{}\Leftrightarrow [\Twist_{\NNStrand\NStrand},\Twist_{\TheStrand\Strand}]=1, 
\\
\label{cor:pres_PMap_free_prod_eq1}
\Twist_{\NNStrand\Strand}\Twist_{\TheStrand\Strand}\Twist_{\NNStrand\Strand}^{-1}= \Twist_{\TheStrand\Strand}^{-1}\Twist_{\TheStrand\NNStrand}^{-1}\Twist_{\TheStrand\Strand}\Twist_{\TheStrand\NNStrand}\Twist_{\TheStrand\Strand} & \mystackrel{(\ref{cor:pres_PMap_free_prod_eq2_1})}\Leftrightarrow \Twist_{\TheStrand\Strand}\Twist_{\NNStrand\Strand}\Twist_{\TheStrand\NNStrand}=\Twist_{\NNStrand\Strand}\Twist_{\TheStrand\NNStrand}\Twist_{\TheStrand\Strand}, \numbereq
\\
\label{cor:pres_PMap_free_prod_eq2}
\Twist_{\NNStrand\Strand}^{-1}\Twist_{\TheStrand\Strand}\Twist_{\NNStrand\Strand}=\Twist_{\TheStrand\NNStrand}\Twist_{\TheStrand\Strand}\Twist_{\TheStrand\NNStrand}^{-1} & \mystackrel{}\Leftrightarrow \Twist_{\TheStrand\Strand}\Twist_{\NNStrand\Strand}\Twist_{\TheStrand\NNStrand}=\Twist_{\NNStrand\Strand}\Twist_{\TheStrand\NNStrand}\Twist_{\TheStrand\Strand}, \numbereq
\\
\label{cor:pres_PMap_free_prod_eq3}
\Twist_{\NNStrand\Str}\Twist_{\TheStrand\Strand}\Twist_{\NNStrand\Str}^{-1}=\Twist_{\TheStrand\Str}^{-1}\Twist_{\TheStrand\NNStrand}^{-1}\Twist_{\TheStrand\Str}\Twist_{\TheStrand\NNStrand}\Twist_{\TheStrand\Strand}\Twist_{\TheStrand\NNStrand}^{-1}\Twist_{\TheStrand\Str}^{-1}\Twist_{\TheStrand\NNStrand}\Twist_{\TheStrand\Str} & \mystackrel{(\ref{cor:pres_PMap_free_prod_eq2_1})}\Leftrightarrow [\Twist_{\TheStrand\NNStrand}\Twist_{\TheStrand\Strand}\Twist_{\TheStrand\NNStrand}^{-1},\Twist_{\NNStrand\Str}]=1, \numbereq
\\
\label{cor:pres_PMap_free_prod_eq4}
\Twist_{\NNStrand\Str}^{-1}\Twist_{\TheStrand\Strand}\Twist_{\NNStrand\Str}=\Twist_{\TheStrand\NNStrand}\Twist_{\TheStrand\Str}\Twist_{\TheStrand\NNStrand}^{-1}\Twist_{\TheStrand\Str}^{-1}\Twist_{\TheStrand\Strand}\Twist_{\TheStrand\Str}\Twist_{\TheStrand\NNStrand}\Twist_{\TheStrand\Str}^{-1}\Twist_{\TheStrand\NNStrand}^{-1} & \mystackrel{(\ref{cor:pres_PMap_free_prod_eq1},\ref{cor:pres_PMap_free_prod_eq2_1})}\Leftrightarrow [\Twist_{\TheStrand\NNStrand}\Twist_{\TheStrand\Strand}\Twist_{\TheStrand\NNStrand}^{-1},\Twist_{\NNStrand\Str}]=1, \numbereq
\\
\label{cor:pres_PMap_free_prod_eq3_1}
\TwistP_{\NNStrand\NPct}\Twist_{\TheStrand\Strand}\TwistP_{\NNStrand\NPct}^{-1}=\TwistP_{\TheStrand\NPct}^{-1}\Twist_{\TheStrand\NNStrand}^{-1}\TwistP_{\TheStrand\NPct}\Twist_{\TheStrand\NNStrand}\Twist_{\TheStrand\Strand}\Twist_{\TheStrand\NNStrand}^{-1}\TwistP_{\TheStrand\NPct}^{-1}\Twist_{\TheStrand\NNStrand}\TwistP_{\TheStrand\NPct} & \mystackrel{(\ref{cor:pres_PMap_free_prod_eq2_2n})}\Leftrightarrow [\Twist_{\TheStrand\NNStrand}\Twist_{\TheStrand\Strand}\Twist_{\TheStrand\NNStrand}^{-1},\TwistP_{\NNStrand\NPct}]=1, \tag{\ref{cor:pres_PMap_free_prod_eq3}-1} 
\\
\label{cor:pres_PMap_free_prod_eq4_1}
\TwistP_{\NNStrand\NPct}^{-1}\Twist_{\TheStrand\Strand}\TwistP_{\NNStrand\NPct}=\Twist_{\TheStrand\NNStrand}\TwistP_{\TheStrand\NPct}\Twist_{\TheStrand\NNStrand}^{-1}\TwistP_{\TheStrand\NPct}^{-1}\Twist_{\TheStrand\Strand}\TwistP_{\TheStrand\NPct}\Twist_{\TheStrand\NNStrand}\TwistP_{\TheStrand\NPct}^{-1}\Twist_{\TheStrand\NNStrand}^{-1} & \mystackrel{(\ref{cor:pres_PMap_free_prod_eq1_4n},\ref{cor:pres_PMap_free_prod_eq2_2n})}\Leftrightarrow [\Twist_{\TheStrand\NNStrand}\Twist_{\TheStrand\Strand}\Twist_{\TheStrand\NNStrand}^{-1},\TwistP_{\NNStrand\NPct}]=1, \tag{\ref{cor:pres_PMap_free_prod_eq4}-1} 
\\
\label{cor:pres_PMap_free_prod_eq3_2}
\TwistC_{\NNStrand\nu}\Twist_{\TheStrand\Strand}\TwistC_{\NNStrand\nu}^{-1}=\TwistC_{\TheStrand\nu}^{-1}\Twist_{\TheStrand\NNStrand}^{-1}\TwistC_{\TheStrand\nu}\Twist_{\TheStrand\NNStrand}\Twist_{\TheStrand\Strand}\Twist_{\TheStrand\NNStrand}^{-1}\TwistC_{\TheStrand\nu}^{-1}\Twist_{\TheStrand\NNStrand}\TwistC_{\TheStrand\nu} & \mystackrel{(\ref{cor:pres_PMap_free_prod_eq2_3n})}\Leftrightarrow [\Twist_{\TheStrand\NNStrand}\Twist_{\TheStrand\Strand}\Twist_{\TheStrand\NNStrand}^{-1},\TwistC_{\NNStrand\nu}]=1, \tag{\ref{cor:pres_PMap_free_prod_eq3}-2}
\\
\label{cor:pres_PMap_free_prod_eq4_2}
\TwistC_{\NNStrand\nu}^{-1}\Twist_{\TheStrand\Strand}\TwistC_{\NNStrand\nu}=\Twist_{\TheStrand\NNStrand}\TwistC_{\TheStrand\nu}\Twist_{\TheStrand\NNStrand}^{-1}\TwistC_{\TheStrand\nu}^{-1}\Twist_{\TheStrand\Strand}\TwistC_{\TheStrand\nu}\Twist_{\TheStrand\NNStrand}\TwistC_{\TheStrand\nu}^{-1}\Twist_{\TheStrand\NNStrand}^{-1} & \mystackrel{(\ref{cor:pres_PMap_free_prod_eq1_5n},\ref{cor:pres_PMap_free_prod_eq2_3n})}\Leftrightarrow [\Twist_{\TheStrand\NNStrand}\Twist_{\TheStrand\Strand}\Twist_{\TheStrand\NNStrand}^{-1},\TwistC_{\NNStrand\nu}]=1, \tag{\ref{cor:pres_PMap_free_prod_eq4}-2} 
\\
\label{cor:pres_PMap_free_prod_eq2_1}
\Twist_{\Strand\Str}\Twist_{\TheStrand\Strand}\Twist_{\Strand\Str}^{-1}=\Twist_{\TheStrand\Str}^{-1}\Twist_{\TheStrand\Strand}\Twist_{\TheStrand\Str} & \mystackrel{}\Leftrightarrow \Twist_{\TheStrand\Str}\Twist_{\Strand\Str}\Twist_{\TheStrand\Strand}=\Twist_{\TheStrand\Strand}\Twist_{\TheStrand\Str}\Twist_{\Strand\Str}, \tag{\ref{cor:pres_PMap_free_prod_eq2}-1}
\\
\label{cor:pres_PMap_free_prod_eq1_1n}
\Twist_{\Strand\Str}^{-1}\Twist_{\TheStrand\Strand}\Twist_{\Strand\Str}=\Twist_{\TheStrand\Strand}\Twist_{\TheStrand\Str}\Twist_{\TheStrand\Strand}\Twist_{\TheStrand\Str}^{-1}\Twist_{\TheStrand\Strand}^{-1} & \mystackrel{(\ref{cor:pres_PMap_free_prod_eq2})}\Leftrightarrow \Twist_{\Strand\Str}\Twist_{\TheStrand\Strand}\Twist_{\TheStrand\Str}=\Twist_{\TheStrand\Strand}\Twist_{\TheStrand\Str}\Twist_{\Strand\Str}, \tag{\ref{cor:pres_PMap_free_prod_eq1}-1}
\\
\label{cor:pres_PMap_free_prod_eq2_2n}
\TwistP_{\Strand\NPct}\Twist_{\TheStrand\Strand}\TwistP_{\Strand\NPct}^{-1}=\TwistP_{\TheStrand\NPct}^{-1}\Twist_{\TheStrand\Strand}\TwistP_{\TheStrand\NPct} & \mystackrel{}\Leftrightarrow \TwistP_{\TheStrand\NPct}\TwistP_{\Strand\NPct}\Twist_{\TheStrand\Strand}=\Twist_{\TheStrand\Strand}\TwistP_{\TheStrand\NPct}\TwistP_{\Strand\NPct}, \tag{\ref{cor:pres_PMap_free_prod_eq2}-2}
\\
\label{cor:pres_PMap_free_prod_eq1_2n}
\TwistP_{\Strand\NPct}^{-1}\Twist_{\TheStrand\Strand}\TwistP_{\Strand\NPct}=\Twist_{\TheStrand\Strand}\TwistP_{\TheStrand\NPct}\Twist_{\TheStrand\Strand}\TwistP_{\TheStrand\NPct}^{-1}\Twist_{\TheStrand\Strand}^{-1} & \mystackrel{(\ref{cor:pres_PMap_free_prod_eq2_4n})}\Leftrightarrow \TwistP_{\Strand\NPct}\Twist_{\TheStrand\Strand}\TwistP_{\TheStrand\NPct}=\Twist_{\TheStrand\Strand}\TwistP_{\TheStrand\NPct}\TwistP_{\Strand\NPct}, \tag{\ref{cor:pres_PMap_free_prod_eq1}-2} 
\\
\label{cor:pres_PMap_free_prod_eq2_3n}
\TwistC_{\Strand\nu}\Twist_{\TheStrand\Strand}\TwistC_{\Strand\nu}^{-1}=\TwistC_{\TheStrand\nu}^{-1}\Twist_{\TheStrand\Strand}\TwistC_{\TheStrand\nu} & \mystackrel{}\Leftrightarrow \TwistC_{\TheStrand\nu}\TwistC_{\Strand\nu}\Twist_{\TheStrand\Strand}=\Twist_{\TheStrand\Strand}\TwistC_{\TheStrand\nu}\TwistC_{\Strand\nu}, \tag{\ref{cor:pres_PMap_free_prod_eq2}-3}
\\
\label{cor:pres_PMap_free_prod_eq1_3n}
\TwistC_{\Strand\nu}^{-1}\Twist_{\TheStrand\Strand}\TwistC_{\Strand\nu}=\Twist_{\TheStrand\Strand}\TwistC_{\TheStrand\nu}\Twist_{\TheStrand\Strand}\TwistC_{\TheStrand\nu}^{-1}\Twist_{\TheStrand\Strand}^{-1} & \mystackrel{(\ref{cor:pres_PMap_free_prod_eq2_5n})}\Leftrightarrow \TwistC_{\Strand\nu}\Twist_{\TheStrand\Strand}\TwistC_{\TheStrand\nu}=\Twist_{\TheStrand\Strand}\TwistC_{\TheStrand\nu}\TwistC_{\Strand\nu}, \tag{\ref{cor:pres_PMap_free_prod_eq1}-3}
\\
\Twist_{\Str\Strr}\Twist_{\TheStrand\Strand}\Twist_{\Str\Strr}^{-1}=\Twist_{\Str\Strr}^{-1}\Twist_{\TheStrand\Strand}\Twist_{\Str\Strr}=\Twist_{\TheStrand\Strand} & \mystackrel{}\Leftrightarrow [\Twist_{\Str\Strr},\Twist_{\TheStrand\Strand}]=1, 
\\
\TwistP_{\Str\NPct}\Twist_{\TheStrand\Strand}\TwistP_{\Str\NPct}^{-1}=\TwistP_{\Str\NPct}^{-1}\Twist_{\TheStrand\Strand}\TwistP_{\Str\NPct}=\Twist_{\TheStrand\Strand} & \mystackrel{}\Leftrightarrow [\TwistP_{\Str\NPct},\Twist_{\TheStrand\Strand}]=1, 
\\
\TwistC_{\Str\nu}\Twist_{\TheStrand\Strand}\TwistC_{\Str\nu}^{-1}=\TwistC_{\Str\nu}^{-1}\Twist_{\TheStrand\Strand}\TwistC_{\Str\nu}=\Twist_{\TheStrand\Strand} & \mystackrel{}\Leftrightarrow [\TwistC_{\Str\nu},\Twist_{\TheStrand\Strand}]=1. 
\end{align*}
\begin{align*}
\Twist_{\Strand\Str}\TwistP_{\TheStrand\Pct}\Twist_{\Strand\Str}^{-1}=\Twist_{\Strand\Str}^{-1}\TwistP_{\TheStrand\Pct}\Twist_{\Strand\Str}=\TwistP_{\TheStrand\Pct} & \mystackrel{}\Leftrightarrow [\Twist_{\Strand\Str},\TwistP_{\TheStrand\Pct}]=1, 
\\
\TwistP_{\Strand\Pc}\TwistP_{\TheStrand\Pct}\TwistP_{\Strand\Pc}^{-1}=\TwistP_{\Strand\Pc}^{-1}\TwistP_{\TheStrand\Pct}\TwistP_{\Strand\Pc}=\TwistP_{\TheStrand\Pct} & \mystackrel{}\Leftrightarrow [\TwistP_{\Strand\Pc},\TwistP_{\TheStrand\Pct}]=1, 
\\
\label{cor:pres_PMap_free_prod_eq1_4n}
\TwistP_{\Strand\Pct}\TwistP_{\TheStrand\Pct}\TwistP_{\Strand\Pct}^{-1}=\TwistP_{\TheStrand\Pct}^{-1}\Twist_{\TheStrand\Strand}^{-1}\TwistP_{\TheStrand\Pct}\Twist_{\TheStrand\Strand}\TwistP_{\TheStrand\Pct} & \mystackrel{(\ref{cor:pres_PMap_free_prod_eq2_2n})}\Leftrightarrow \Twist_{\TheStrand\Strand}\TwistP_{\TheStrand\Pct}\TwistP_{\Strand\Pct}=\TwistP_{\Strand\Pct}\Twist_{\TheStrand\Strand}\TwistP_{\TheStrand\Pct}, \tag{\ref{cor:pres_PMap_free_prod_eq1}-4}
\\
\label{cor:pres_PMap_free_prod_eq2_4n}
\TwistP_{\Strand\Pct}^{-1}\TwistP_{\TheStrand\Pct}\TwistP_{\Strand\Pct}=\Twist_{\TheStrand\Strand}\TwistP_{\TheStrand\Pct}\Twist_{\TheStrand\Strand}^{-1} & \mystackrel{}\Leftrightarrow \TwistP_{\TheStrand\Pct}\TwistP_{\Strand\Pct}\Twist_{\TheStrand\Strand}=\TwistP_{\Strand\Pct}\Twist_{\TheStrand\Strand}\TwistP_{\TheStrand\Pct}, \tag{\ref{cor:pres_PMap_free_prod_eq2}-4}
\\
\label{cor:pres_PMap_free_prod_eq3_3}
\TwistP_{\Strand\NPct}\TwistP_{\TheStrand\Pct}\TwistP_{\Strand\NPct}^{-1}=\TwistP_{\TheStrand\NPct}^{-1}\Twist_{\TheStrand\Strand}^{-1}\TwistP_{\TheStrand\NPct}\Twist_{\TheStrand\Strand}\TwistP_{\TheStrand\Pct}\Twist_{\TheStrand\Strand}^{-1}\TwistP_{\TheStrand\NPct}^{-1}\Twist_{\TheStrand\Strand}\TwistP_{\TheStrand\NPct} & \mystackrel{(\ref{cor:pres_PMap_free_prod_eq2_2n})}\Leftrightarrow [\Twist_{\TheStrand\Strand}\TwistP_{\TheStrand\Pct}\Twist_{\TheStrand\Strand}^{-1},\TwistP_{\Strand\NPct}]=1, \tag{\ref{cor:pres_PMap_free_prod_eq3}-3} 
\\
\label{cor:pres_PMap_free_prod_eq4_3}
\TwistP_{\Strand\NPct}^{-1}\TwistP_{\TheStrand\Pct}\TwistP_{\Strand\NPct}=\Twist_{\TheStrand\Strand}\TwistP_{\TheStrand\NPct}\Twist_{\TheStrand\Strand}^{-1}\TwistP_{\TheStrand\NPct}^{-1}\TwistP_{\TheStrand\Pct}\TwistP_{\TheStrand\NPct}\Twist_{\TheStrand\Strand}\TwistP_{\TheStrand\NPct}^{-1}\Twist_{\TheStrand\Strand}^{-1} & \mystackrel{(\ref{cor:pres_PMap_free_prod_eq1_4n},\ref{cor:pres_PMap_free_prod_eq2_2n})}\Leftrightarrow [\Twist_{\TheStrand\Strand}\TwistP_{\TheStrand\Pct}\Twist_{\TheStrand\Strand}^{-1},\TwistP_{\Strand\NPct}]=1, \tag{\ref{cor:pres_PMap_free_prod_eq4}-3} 
\\
\label{cor:pres_PMap_free_prod_eq3_4}
\TwistC_{\Strand\nu}\TwistP_{\TheStrand\Pct}\TwistC_{\Strand\nu}^{-1}=\TwistC_{\TheStrand\nu}^{-1}\Twist_{\TheStrand\Strand}^{-1}\TwistC_{\TheStrand\nu}\Twist_{\TheStrand\Strand}\TwistP_{\TheStrand\Pct}\Twist_{\TheStrand\Strand}^{-1}\TwistC_{\TheStrand\nu}^{-1}\Twist_{\TheStrand\Strand}\TwistC_{\TheStrand\nu} & \mystackrel{(\ref{cor:pres_PMap_free_prod_eq2_3n})}\Leftrightarrow [\Twist_{\TheStrand\Strand}\TwistP_{\TheStrand\Pct}\Twist_{\TheStrand\Strand}^{-1},\TwistC_{\Strand\nu}]=1, \tag{\ref{cor:pres_PMap_free_prod_eq3}-4} 
\\
\label{cor:pres_PMap_free_prod_eq4_4}
\TwistC_{\Strand\nu}^{-1}\TwistP_{\TheStrand\Pct}\TwistC_{\Strand\nu}=\Twist_{\TheStrand\Strand}\TwistC_{\TheStrand\nu}\Twist_{\TheStrand\Strand}^{-1}\TwistC_{\TheStrand\nu}^{-1}\TwistP_{\TheStrand\Pct}\TwistC_{\TheStrand\nu}\Twist_{\TheStrand\Strand}\TwistC_{\TheStrand\nu}^{-1}\Twist_{\TheStrand\Strand}^{-1} & \mystackrel{(\ref{cor:pres_PMap_free_prod_eq1_5n},\ref{cor:pres_PMap_free_prod_eq2_3n})}\Leftrightarrow [\Twist_{\TheStrand\Strand}\TwistP_{\TheStrand\Pct}\Twist_{\TheStrand\Strand}^{-1},\TwistC_{\Strand\nu}]=1. \tag{\ref{cor:pres_PMap_free_prod_eq4}-4} 
\end{align*}
\begin{align*}
\Twist_{\Strand\Str}\TwistC_{\TheStrand\nu}\Twist_{\Strand\Str}^{-1}=
\Twist_{\Strand\Str}^{-1}\TwistC_{\TheStrand\nu}\Twist_{\Strand\Str}=\TwistC_{\TheStrand\nu} & \mystackrel{}\Leftrightarrow [\Twist_{\Strand\Str},\TwistC_{\TheStrand\nu}]=1, 
\\
\TwistP_{\Strand\NPct}\TwistC_{\TheStrand\nu}\TwistP_{\Strand\NPct}^{-1}=\TwistP_{\Strand\NPct}^{-1}\TwistC_{\TheStrand\nu}\TwistP_{\Strand\NPct}=\TwistC_{\TheStrand\nu} & \mystackrel{}\Leftrightarrow [\TwistP_{\Strand\NPct},\TwistC_{\TheStrand\nu}]=1, 
\\
\TwistC_{\Strand\mu}\TwistC_{\TheStrand\nu}\TwistC_{\Strand\mu}^{-1}=\TwistC_{\Strand\mu}^{-1}\TwistC_{\TheStrand\nu}\TwistC_{\Strand\mu}=\TwistC_{\TheStrand\nu} & \mystackrel{}\Leftrightarrow [\TwistC_{\Strand\mu},\TwistC_{\TheStrand\nu}]=1, 
\\
\label{cor:pres_PMap_free_prod_eq1_5n}
\TwistC_{\Strand\nu}\TwistC_{\TheStrand\nu}\TwistC_{\Strand\nu}^{-1}=\TwistC_{\TheStrand\nu}^{-1}\Twist_{\TheStrand\Strand}^{-1}\TwistC_{\TheStrand\nu}\Twist_{\TheStrand\Strand}\TwistC_{\TheStrand\nu} & \mystackrel{(\ref{cor:pres_PMap_free_prod_eq2_3n})}\Leftrightarrow \TwistC_{\Strand\nu}\Twist_{\TheStrand\Strand}\TwistC_{\TheStrand\nu}=\TwistC_{\TheStrand\nu}\TwistC_{\Strand\nu}\Twist_{\TheStrand\Strand}, \tag{\ref{cor:pres_PMap_free_prod_eq1}-5}
\\
\label{cor:pres_PMap_free_prod_eq2_5n}
\TwistC_{\Strand\nu}^{-1}\TwistC_{\TheStrand\nu}\TwistC_{\Strand\nu}=\Twist_{\TheStrand\Strand}\TwistC_{\TheStrand\nu}\Twist_{\TheStrand\Strand}^{-1} & \mystackrel{}\Leftrightarrow \TwistC_{\TheStrand\nu}\TwistC_{\Strand\nu}\Twist_{\TheStrand\Strand}=\TwistC_{\Strand\nu}\Twist_{\TheStrand\Strand}\TwistC_{\TheStrand\nu}, \tag{\ref{cor:pres_PMap_free_prod_eq2}-5}
\\
\label{cor:pres_PMap_free_prod_eq3_5}
\TwistC_{\Strand\omicron}\TwistC_{\TheStrand\nu}\TwistC_{\Strand\omicron}^{-1}=\TwistC_{\TheStrand\omicron}^{-1}\Twist_{\TheStrand\Strand}^{-1}\TwistC_{\TheStrand\omicron}\Twist_{\TheStrand\Strand}\TwistC_{\TheStrand\nu}\Twist_{\TheStrand\Strand}^{-1}\TwistC_{\TheStrand\omicron}^{-1}\Twist_{\TheStrand\Strand}\TwistC_{\TheStrand\omicron} &  \mystackrel{(\ref{cor:pres_PMap_free_prod_eq2_3n})}\Leftrightarrow [\Twist_{\TheStrand\Strand}\TwistC_{\TheStrand\nu}\Twist_{\TheStrand\Strand}^{-1},\TwistC_{\Strand\omicron}]=1, \tag{\ref{cor:pres_PMap_free_prod_eq3}-5}
\\
\label{cor:pres_PMap_free_prod_eq4_5}
\TwistC_{\Strand\omicron}^{-1}\TwistC_{\TheStrand\nu}\TwistC_{\Strand\omicron}=\Twist_{\TheStrand\Strand}\TwistC_{\TheStrand\omicron}\Twist_{\TheStrand\Strand}^{-1}\TwistC_{\TheStrand\omicron}^{-1}\TwistC_{\TheStrand\nu}\TwistC_{\TheStrand\omicron}\Twist_{\TheStrand\Strand}\TwistC_{\TheStrand\omicron}^{-1}\Twist_{\TheStrand\Strand}^{-1} & \mystackrel{(\ref{cor:pres_PMap_free_prod_eq1_5n},\ref{cor:pres_PMap_free_prod_eq2_3n})}\Leftrightarrow [\Twist_{\TheStrand\Strand}\TwistC_{\TheStrand\nu}\Twist_{\TheStrand\Strand}^{-1},\TwistC_{\Strand\omicron}]=1. \tag{\ref{cor:pres_PMap_free_prod_eq4}-5} 
\end{align*}
\new{Together,} all these relations yield the presentation from above. 
\end{proof}

\subsection{Generating set and presentation of $\MapIdOrb{\TheStrand}{\Sigma_\freeprod(\ThePct)}$} 
\label{subsec:gen_sets_and_pres_Map}

Now we want to deduce a presentation of $\MapIdOrb{\TheStrand}{\Sigma_\freeprod(\ThePct)}$. For this purpose, let $H_\Strand$ for $1\leq\Strand<\TheStrand$ be the homeomorphism that performs the following half-twist on each $\freeprod$-translate of the disk $D_{\Strand,\Strand+1}$, i.e.\ the disk $D_{\Str,\Strand+1}$ from Figure \ref{fig:disks_with_marked_pts} for $\Str=\Strand$:  
\begin{figure}[H]
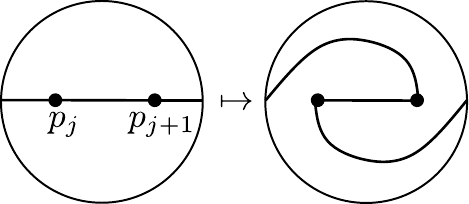
\caption{The half-twist $H_\Strand$.}
\label{fig:homeo_half-twist}
\end{figure}

For each $1\leq\NPct\leq\ThePct$ and $1\leq\nu\leq\TheCone$, let $\TwP{\NPct}:=\TwistP_{1\NPct}$ and $\TwC{\nu}:=\TwistC_{1\nu}$. \new{As for the pure generators,} we will use the names $H_\Strand,\TwP{\NPct}$ and $\TwC{\nu}$ as acronyms for the represented mapping classes. 

\begin{lemma}
\label{lem:orb_mcg_n_gen_set}
The group $\MapIdOrb{\TheStrand}{\Sigma_\freeprod(\ThePct)}$ is generated by 
\[
H_\Strand, \TwP{\NPct} \; \text{ and } \; \TwC{\nu}
\]
with $1\leq\Strand<\TheStrand$, $1\leq\NPct\leq\ThePct$ and $1\leq\nu\leq\TheCone$. 
\end{lemma}
\begin{proof}
Definition \ref{def:Map_orb(ThePct)} yields a short exact sequence
\[
1\rightarrow\PMapIdOrb{\TheStrand}{\Sigma_\freeprod(\ThePct)}\rightarrow\MapIdOrb{\TheStrand}{\Sigma_\freeprod(\ThePct)}\rightarrow\Sym_\TheStrand\rightarrow1. 
\]
Further, the elements $H_\Strand$ for $1\leq\Strand<\TheStrand$ map to the generating set of adjacent transpositions in $\Sym_\TheStrand$. Hence, the elements $H_\Strand$ and the pure mapping classes generate $\MapIdOrb{\TheStrand}{\Sigma_\freeprod(\ThePct)}$. Using Lemma \ref{cor:pres_PMap_free_prod}, the subgroup $\PMapIdOrb{\TheStrand}{\Sigma_\freeprod(\ThePct)}$ is generated by 
\[
\Twist_{\Strand\Str}, \TwistP_{\NStrand\NPct} \; \text{ and } \; \TwistC_{\NStrand\nu}
\]
with $1\leq\Str<\Strand\leq\TheStrand, 1\leq\NStrand\leq\TheStrand, 1\leq\NPct\leq\ThePct$ and $1\leq\nu\leq\TheCone$. Moreover, we observe 
\begin{align*}
H_{\Strand-1}^{-1}...H_{\Str+1}^{-1}(D_{\Str,\Str+1}) & =D_{\Str,\Strand}, 
\\
H_{\Strand-1}^{-1}...H_1^{-1}(D_{r_\NPct,1}) & =D_{r_\NPct,\Strand} \text{ and }
\\
H_{\NStrand-1}^{-1}...H_1^{-1}(D_{\cp_\nu,1}) & =D_{\cp_\nu,\NStrand},
\end{align*}
see Figure \ref{fig:Map_gens_rel_PMap_gens}. Using Lemma \ref{lem:conj_manipulates_supp}, this implies 
\begin{align*}
\label{eq:def_A_ji}
\Twist_{\Strand\Str} & = H_{\Strand-1}^{-1}...H_{\Str+1}^{-1}H_\Str^2H_{\Str+1}...H_{\Strand-1}, \numbereq 
\\
\label{eq:def_A_kr}
\TwistP_{\NStrand\NPct} & = H_{\NStrand-1}^{-1}...H_1^{-1}\TwP{\NPct}H_1...H_{\NStrand-1} \text{ and } \numbereq 
\\
\label{eq:def_A_kc}
\TwistC_{\NStrand\nu} & = H_{\NStrand-1}^{-1}...H_1^{-1}\TwC{\nu}H_1...H_{\NStrand-1}. \numbereq 
\end{align*}
Consequently, $H_\Strand, \TwP{\NPct}$ and $\TwC{\nu}$ with $1\leq\Strand<\TheStrand,1\leq\NPct\leq\ThePct$ and $1\leq\nu\leq\TheCone$ generate $\MapIdOrb{\TheStrand}{\Sigma_\freeprod(\ThePct)}$. 
\end{proof}

\begin{figure}[H]
\centerline{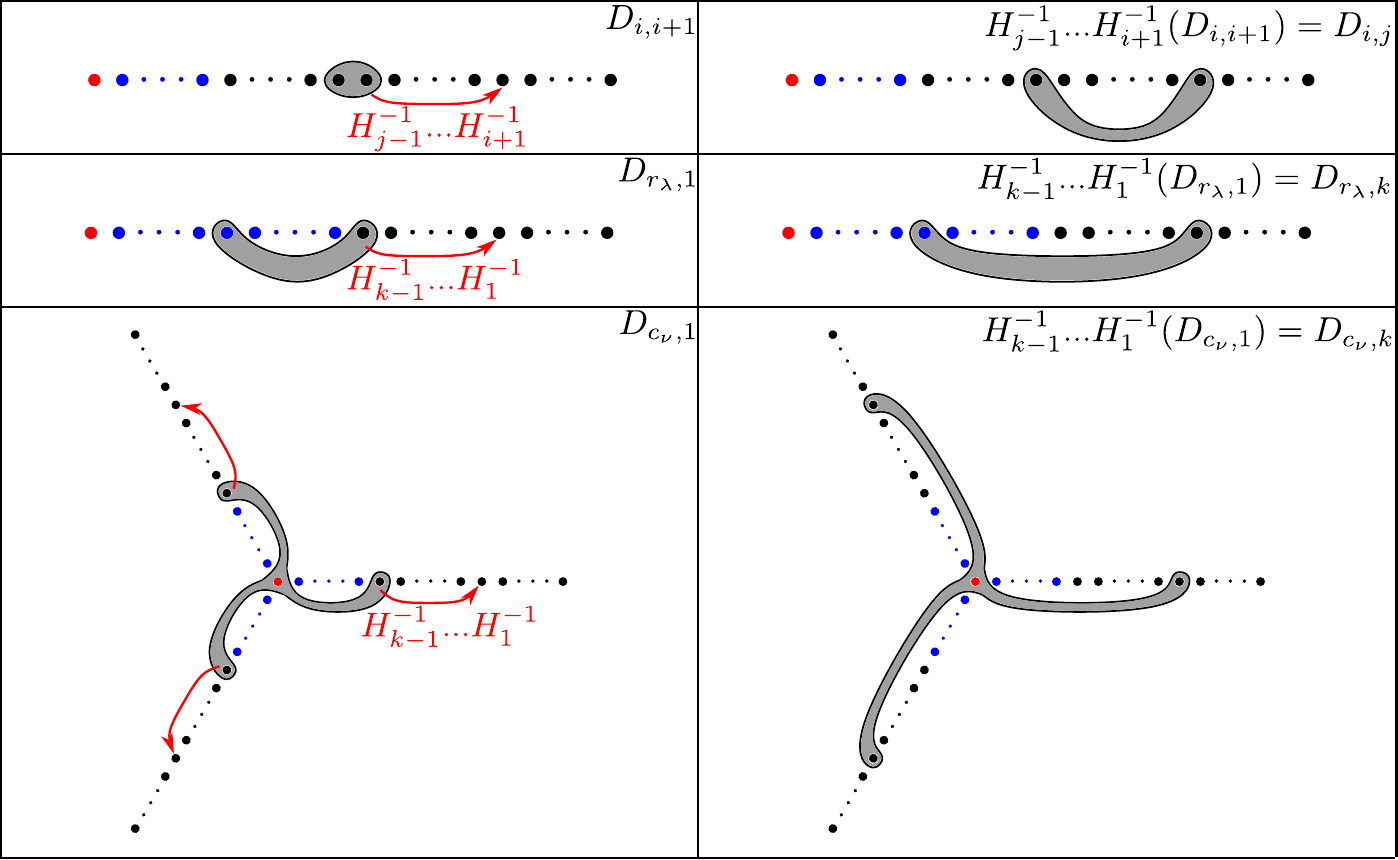}
\caption{\new{Relating the supporting disks of $H_\Str$ and $\Twist_{\Strand\Str}$, $\TwP{\NPct}$ and $\TwistP_{\NStrand\NPct}$, and $\TwC{\nu}$ and $\TwistC_{\NStrand\nu}$.}}
\label{fig:Map_gens_rel_PMap_gens}
\end{figure}

Subject to the above generating set, we will give a suitable set of relations that defines $\MapIdOrb{\TheStrand}{\Sigma_\freeprod(\ThePct)}$. For this purpose, we introduce the abbreviations
\begin{align*}
\label{eq:def_X_Strand}
\Twistn{\Strand}:=\Twist_{\TheStrand\Strand} & = H_{\TheStrand-1}^{-1}...H_{\Strand+1}^{-1}H_\Strand^2H_{\Strand+1}...H_{\TheStrand-1}, \numbereq
\\
\label{eq:def_X_r_Pct}
\TwistnP{\NPct}:=\TwistP_{\TheStrand\NPct} & = H_{\TheStrand-1}^{-1}...H_1^{-1}\TwP{\Pct}H_1...H_{\TheStrand-1} \text{ and } \numbereq
\\
\label{eq:def_X_tau_nu}
\TwistnC{\nu}:=\TwistC_{\TheStrand\nu} & = H_{\TheStrand-1}^{-1}...H_1^{-1}\TwC{\nu} H_1...H_{\TheStrand-1}. \numbereq 
\end{align*}
We recall that these elements generate the image of $\PushPMap^{orb}$, i.e.\ they generate a free subgroup $\freegrp{\TheStrand-1+\ThePct+\TheCone}$ in $\PMapIdOrb{\TheStrand}{\Sigma_\freeprod(\ThePct)}$. 

\new{In the following, the relations $H_\Str H_{\Str+1}H_\Str=H_{\Str+1}H_\Str H_{\Str+1}$ for $1\leq\Str\leq\TheStrand-2$ and $[H_\Strand,H_\NStrand]=1$ for $1\leq\Strand,\NStrand<\TheStrand$ with $\vert\Strand-\NStrand\vert\geq2$ are called \textit{braid and commutator relations} for $H_1,...,H_{\TheStrand-1}$.} 

\begin{lemma}
\label{lem:gens_sat_rels_mcg}
The generators $H_1,...,H_{\TheStrand-1},\TwP{1},...,\TwP{\ThePct},\TwC{1},...,\TwC{\TheCone}$ of $\MapIdOrb{\TheStrand}{\Sigma_\freeprod(\ThePct)}$ satisfy the following relations for $2\leq\Strand<\TheStrand$, $1\leq\Pc,\NPct\leq\ThePct$ with $\Pc<\NPct$ and $1\leq\mu,\nu\leq\TheCone$ with $\mu<\nu$: 
\begin{enumerate}
\item \label{lem:gens_sat_rels_mcg_rel1}
braid and commutator relations for the generators $H_1,...,H_{\TheStrand-1}$, 
\item \label{lem:gens_sat_rels_mcg_rel2}
\begin{enumerate}
\item \label{lem:gens_sat_rels_mcg_rel2a}
$[\TwP{\NPct},H_\Strand]=1$, 
\item \label{lem:gens_sat_rels_mcg_rel2b}
$[\TwC{\nu},H_\Strand]=1$, 
\end{enumerate}
\item \label{lem:gens_sat_rels_mcg_rel3}
\begin{enumerate}
\item \label{lem:gens_sat_rels_mcg_rel3a} 
$[H_1\TwP{\NPct}H_1,\TwP{\NPct}]=1$, 
\item \label{lem:gens_sat_rels_mcg_rel3b} 
$[H_1\TwC{\nu} H_1,\TwC{\nu}]=1$ and 
\end{enumerate}
\item \label{lem:gens_sat_rels_mcg_rel4}
\begin{enumerate}
\item \label{lem:gens_sat_rels_mcg_rel4a}
$[\TwP{\Pc},\TwistP_{2\NPct}]=1$ \; for \; $\TwistP_{2\NPct}=H_1^{-1}\TwP{\NPct}H_1$,
\item\label{lem:gens_sat_rels_mcg_rel4b}
$[\TwC{\mu},\TwistC_{2\nu}]=1$ \; for \; $\TwistC_{2\nu}=H_1^{-1}\TwC{\nu}H_1$ and
\item \label{lem:gens_sat_rels_mcg_rel4c}
$[\TwP{\NPct},\TwistC_{2\nu}]=1$ \; for \; $\TwistC_{2\nu}=H_1^{-1}\TwC{\nu}H_1$. 
\end{enumerate}
\end{enumerate}
In particular, these relations imply that 
\[
\MapIdOrb{\TheStrand-1}{\Sigma_\freeprod(\ThePct)}=\langle H_1,...,H_{\TheStrand-2},\TwP{\NPct},\TwC{\nu}\mid 1\leq\NPct\leq\ThePct, 1\leq\nu\leq\TheCone\rangle 
\] 
acts on $\freegrp{\TheStrand-1+\ThePct+\TheCone}$ via conjugation. More precisely, for $1\leq\Str,\Strand,\NStrand<\TheStrand$ with $\Str<\Strand<\NStrand$, $1\leq\Pc,\Pct,\NPct\leq\ThePct$ with $\Pc<\Pct<\NPct$ and $1\leq\mu,\nu,\omicron\leq\TheCone$ with $\mu<\nu<\omicron$, we have: 

\begin{align*}
H_\NStrand \Twistn{\Strand} H_\NStrand^{-1} & =H_\NStrand^{-1}\Twistn{\Strand} H_\NStrand=\Twistn{\Strand} && \text{ for } \NStrand\leq\TheStrand-2, 
\\
H_\Strand \Twistn{\Strand} H_\Strand^{-1} & =\Twistn{\Strand}^{-1}\Twistn{\Strand+1}\Twistn{\Strand} && \text{ for } \Strand\leq\TheStrand-2, 
\\
H_\Strand^{-1}\Twistn{\Strand} H_\Strand & =\Twistn{\Strand+1} && \text{ for } \Strand\leq\TheStrand-2, 
\\
H_{\Strand-1}\Twistn{\Strand} H_{\Strand-1}^{-1} & =\Twistn{\Strand-1} && \text{ for } \Strand<\TheStrand, 
\\
H_{\Strand-1}^{-1}\Twistn{\Strand} H_{\Strand-1} & =\Twistn{\Strand} \Twistn{\Strand-1}\Twistn{\Strand}^{-1} && \text{ for } \Strand<\TheStrand, 
\\
H_\Str \Twistn{\Strand} H_\Str^{-1} & =H_\Str^{-1}\Twistn{\Strand} H_\Str=\Twistn{\Strand} && \text{ for } \Str\leq\Strand-2, 
\\
\TwP{\NPct}\Twistn{\Strand}\TwP{\NPct}^{-1} & =\TwP{\NPct}^{-1}\Twistn{\Strand}\TwP{\NPct}=\Twistn{\Strand} && \text{ for } 2\leq\Strand, 
\\
\TwC{\nu}\Twistn{\Strand}\TwC{\nu}^{-1} & =\TwC{\nu}^{-1}\Twistn{\Strand} \TwC{\nu}=\Twistn{\Strand} && \text{ for } 2\leq\Strand, 
\\
\TwP{\NPct}\Twistn{1}\TwP{\NPct}^{-1} & =\TwistnP{\NPct}^{-1}\Twistn{1}\TwistnP{\NPct}, 
\\
\TwP{\NPct}^{-1}\Twistn{1}\TwP{\NPct} & =\Twistn{1}\TwistnP{\NPct}\Twistn{1}\TwistnP{\NPct}^{-1}\Twistn{1}^{-1}, 
\\
\TwC{\nu}\Twistn{1}\TwC{\nu}^{-1} & =\TwistnC{\nu}^{-1}\Twistn{1}\TwistnC{\nu}, 
\\
\TwC{\nu}^{-1}\Twistn{1}\TwC{\nu} & =\Twistn{1}\TwistnC{\nu}\Twistn{1}\TwistnC{\nu}^{-1}\Twistn{1}^{-1}, 
\\
H_\Strand \TwistnP{\Pct}H_\Strand^{-1} & =H_\Strand^{-1}\TwistnP{\Pct}H_\Strand=\TwistnP{\Pct} && \text{ for } \Strand\leq\TheStrand-2, 
\\
\TwP{\Pc}\TwistnP{\Pct}\TwP{\Pc}^{-1} & =\TwP{\Pc}^{-1}\TwistnP{\Pct}\TwP{\Pc}=\TwistnP{\Pct}, 
\\
\TwP{\Pct}\TwistnP{\Pct}\TwP{\Pct}^{-1} & =\TwistnP{\Pct}^{-1}\Twistn{1}^{-1}\TwistnP{\Pct}\Twistn{1}\TwistnP{\Pct}, 
\\
\TwP{\Pct}^{-1}\TwistnP{\Pct}\TwP{\Pct} & =\Twistn{1}\TwistnP{\Pct}\Twistn{1}^{-1}, 
\\
\TwP{\NPct}\TwistnP{\Pct}\TwP{\NPct}^{-1} & =\TwistnP{\NPct}^{-1}\Twistn{1}^{-1}\TwistnP{\NPct}\Twistn{1}\TwistnP{\Pct}\Twistn{1}^{-1}\TwistnP{\NPct}^{-1}\Twistn{1}\TwistnP{\NPct}, 
\\
\TwP{\NPct}^{-1}\TwistnP{\Pct}\TwP{\NPct} & =\Twistn{1}\TwistnP{\NPct}\Twistn{1}^{-1}\TwistnP{\NPct}^{-1}\TwistnP{\Pct}\TwistnP{\NPct}\Twistn{1}\TwistnP{\NPct}^{-1}\Twistn{1}^{-1}, 
\\
\TwC{\nu}\TwistnP{\Pct}\TwC{\nu}^{-1} & =\TwistnC{\nu}^{-1}\Twistn{1}^{-1}\TwistnC{\nu}\Twistn{1}\TwistnP{\Pct}\Twistn{1}^{-1}\TwistnC{\nu}^{-1}\Twistn{1}\TwistnC{\nu}, 
\\
\TwC{\nu}^{-1}\TwistnP{\Pct}\TwC{\nu} & =\Twistn{1}\TwistnC{\nu}\Twistn{1}^{-1}\TwistnC{\nu}^{-1}\TwistnP{\Pct}\TwistnC{\nu}\Twistn{1}\TwistnC{\nu}^{-1}\Twistn{1}^{-1}, 
\\
H_\Strand \TwistnC{\nu}H_\Strand^{-1} & =H_\Strand^{-1}\TwistnC{\nu}H_\Strand=\TwistnC{\nu} && \text{ for } \Strand\leq\TheStrand-2, 
\\
\TwP{\NPct}\TwistnC{\nu}\TwP{\NPct}^{-1} & =\TwP{\NPct}^{-1}\TwistnC{\nu}\TwP{\NPct}=\TwistnC{\nu}, 
\\
\TwC{\mu}\TwistnC{\nu}\TwC{\mu}^{-1} & =\TwC{\mu}^{-1}\TwistnC{\nu}\TwC{\mu}=\TwistnC{\nu}, 
\\
\TwC{\nu}\TwistnC{\nu}\TwC{\nu}^{-1} & =\TwistnC{\nu}^{-1}\Twistn{1}^{-1}\TwistnC{\nu}\Twistn{1}\TwistnC{\nu}, 
\\
\TwC{\nu}^{-1}\TwistnC{\nu}\TwC{\nu} & =\Twistn{1}\TwistnC{\nu}\Twistn{1}^{-1}, 
\\
\TwC{\omicron} \TwistnC{\nu}\TwC{\omicron}^{-1} & =\TwistnC{\omicron}^{-1}\Twistn{1}^{-1}\TwistnC{\omicron}\Twistn{1}\TwistnC{\nu}\Twistn{1}^{-1}\TwistnC{\omicron}^{-1}\Twistn{1}\TwistnC{\omicron}, 
\\
\TwC{\omicron}^{-1}\TwistnC{\nu}\TwC{\omicron} & =\Twistn{1}\TwistnC{\omicron}\Twistn{1}^{-1}\TwistnC{\omicron}^{-1}\TwistnC{\nu}\TwistnC{\omicron}\Twistn{1}\TwistnC{\omicron}^{-1}\Twistn{1}^{-1}. 
\end{align*}
\end{lemma}
\begin{proof}
The braid and commutator relations for $H_1,...,H_{\TheStrand-1}$ in \ref{cor:pres_PMap_free_prod_rel1} follow as in the surface case. 

The relations in \ref{lem:gens_sat_rels_mcg_rel3} are reformulations of 
\[
\Twist_{21}\TwistC_{2\nu}\TwistC_{1\nu}=\TwistC_{1\nu}\Twist_{21}\TwistC_{2\nu} \text{ and } \Twist_{21}\TwistP_{2\NPct}\TwistP_{1\NPct}=\TwistP_{1\NPct}\Twist_{21}\TwistC_{2\NPct}
\]
from relation \ref{cor:pres_PMap_free_prod_rel4}. Using the definitions of $\Twist_{21}$, $\TwistP_{2\NPct}$ and $\TwistC_{2\nu}$ from (\ref{eq:def_A_ji}), (\ref{eq:def_A_kr}) and (\ref{eq:def_A_kc}), these relations are equivalent to   
\begin{align*}
(H_1^{\textcolor{\short}{2}})\textcolor{\short}{(H_1^{-1}}\TwC{\nu}H_1)\TwC{\nu} & =\TwC{\nu}(H_1^{\textcolor{\short}{2}})\textcolor{\short}{(H_1^{-1}}\TwC{\nu}H_1) \text{ and }
\\
(H_1^{\textcolor{\short}{2}})\textcolor{\short}{(H_1^{-1}}\TwP{\NPct}H_1)\TwP{\NPct} & =\TwP{\NPct}(H_1^{\textcolor{\short}{2}})\textcolor{\short}{(H_1^{-1}}\TwP{\NPct}H_1). 
\end{align*}
These are equivalent to $[\TwC{\nu},H_1\TwC{\nu}H_1]=1$ and $[\TwP{\NPct},H_1\TwP{\NPct}H_1]=1$. 

The commutator relations \ref{lem:gens_sat_rels_mcg_rel2} and \ref{lem:gens_sat_rels_mcg_rel4} are a direct consequence of the fact that the commuting mapping classes can be realized by homeomorphisms with disjoint supports (see Figure~\ref{fig:gens_sat_rels_mcg}). 
\begin{figure}[H]
\centerline{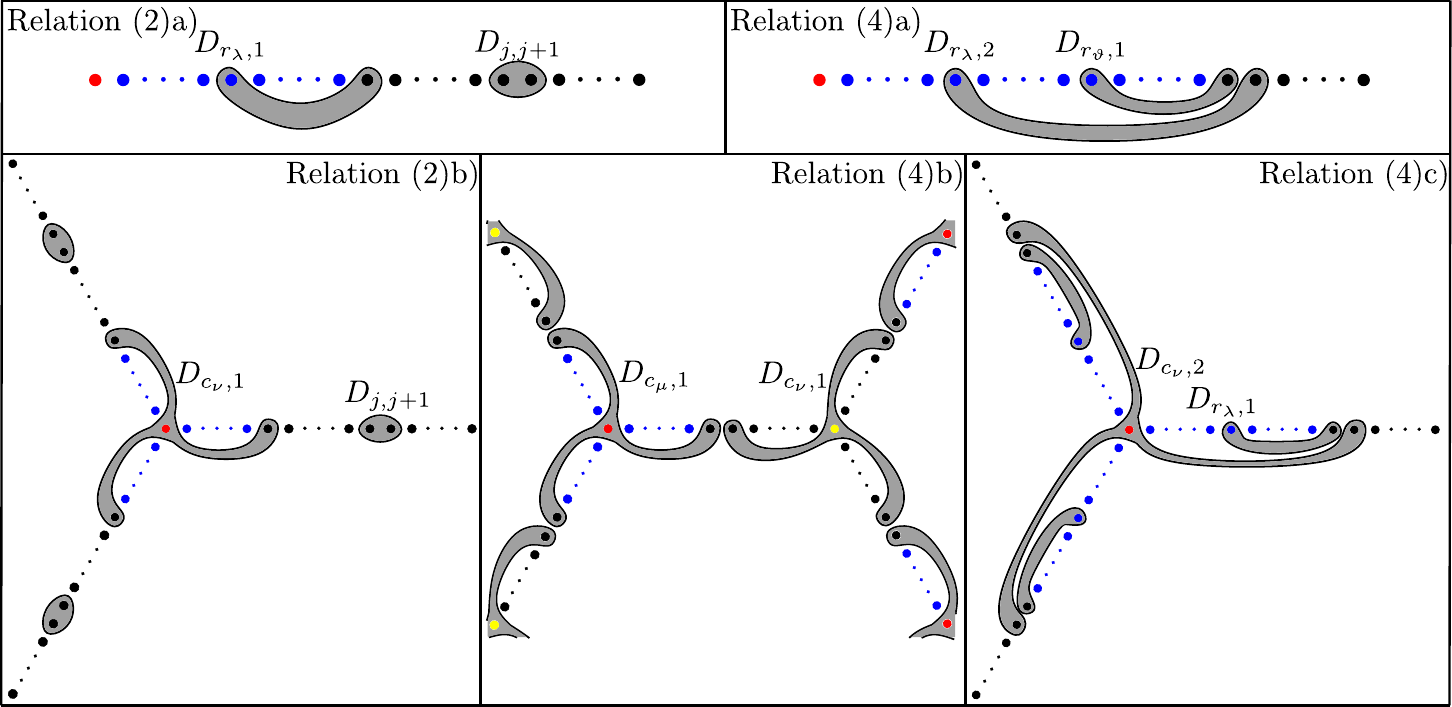}
\caption{Observation of the relations \ref{lem:gens_sat_rels_mcg_rel2}-\ref{lem:gens_sat_rels_mcg_rel4} by consideration of the supporting disks.}
\label{fig:gens_sat_rels_mcg}
\end{figure}

\new{For the proof that the above conjugation relations follow from the relations \ref{lem:gens_sat_rels_mcg_rel1} to \ref{lem:gens_sat_rels_mcg_rel4}, we refer to \cite[Lemma 5.24]{Flechsig2023}.} 
\end{proof}

The above relations and the semidirect product structure together allow us to deduce a presentation of $\MapIdOrb{\TheStrand}{\Sigma_\freeprod(\ThePct)}$. The argument is motivated by the braid combing in \cite[Proposition 3.1]{GonzalezMeneses2010}. 

\begin{proposition}
\label{prop:pres_map_kcp}
\new{For $\TheStrand\geq1$,} the group $\MapIdOrb{\TheStrand}{\Sigma_\freeprod(\ThePct)}$ is presented by generators
\[
H_1,...,H_{\TheStrand-1},\TwP{1},...,\TwP{\ThePct},\TwC{1},...,\TwC{\TheCone}
\]
and defining relations for $2\leq\Strand<\TheStrand$, $1\leq\Pc,\NPct\leq\ThePct$ with $\Pc<\NPct$ and $1\leq\mu,\nu\leq\TheCone$ with $\mu<\nu$: 
\begin{enumerate}
\item \label{prop:pres_map_kcp_rel1}
braid and commutator relations for the generators $H_1,...,H_{\TheStrand-1}$, 
\item \label{prop:pres_map_kcp_rel2}
\begin{enumerate}
\item \label{prop:pres_map_kcp_rel2a}
$[\TwP{\NPct},H_\Strand]=1$, 
\item \label{prop:pres_map_kcp_rel2b}$[\TwC{\nu},H_\Strand]=1$, 
\end{enumerate}
\item \label{prop:pres_map_kcp_rel3}
\begin{enumerate}
\item \label{prop:pres_map_kcp_rel3a}
$[H_1\TwP{\NPct}H_1,\TwP{\NPct}]=1$, 
\item \label{prop:pres_map_kcp_rel3b} 
$[H_1\TwC{\nu} H_1,\TwC{\nu}]=1$ and 
\end{enumerate}
\item \label{prop:pres_map_kcp_rel4}
\begin{enumerate}
\item \label{prop:pres_map_kcp_rel4a} 
$[\TwP{\Pc},\TwistP_{2\NPct}]=1$ \; for \; $\TwistP_{2\NPct}=H_1^{-1}\TwP{\NPct}H_1$, 
\item \label{prop:pres_map_kcp_rel4b} 
$[\TwC{\mu},\TwistC_{2\nu}]=1$ \; for \; $\TwistC_{2\nu}=H_1^{-1}\TwC{\nu}H_1$ and
\item \label{prop:pres_map_kcp_rel4c}
$[\TwP{\Pc},\TwistC_{2\nu}]=1$ \; for \; $\TwistC_{2\nu}=H_1^{-1}\TwC{\nu}H_1$. 

\end{enumerate}
\end{enumerate}
\end{proposition}
\begin{proof}
For $\TheStrand=1$, the group presented above is free of rank $\ThePct+\TheCone$. On the other hand, for $\TheStrand=1$, we have $\MapIdOrb{1}{\Sigma_\freeprod(\ThePct)}=\PMapIdOrb{1}{\Sigma_\freeprod(\ThePct)}$ and the latter group is free of rank $\ThePct+\TheCone$ by Corollary \ref{cor:pure_orb_mcg_ses}. 

\new{For} $\TheStrand\geq2$, we suppose that $\MapIdOrb{\TheStrand-1}{\Sigma_\freeprod(\ThePct)}$ has a presentation as claimed above. Further, we know from Lemma \ref{lem:orb_mcg_n_gen_set} that $\MapIdOrb{\TheStrand}{\Sigma_\freeprod(\ThePct)}$ is generated by $H_1,..,H_{\TheStrand-1},\TwP{1},...,\TwP{\ThePct},\TwC{1},...,\TwC{\TheCone}$. \new{Due to Lemma \ref{lem:gens_sat_rels_mcg},} these generators satisfy the above relations. Hence, 
we have a surjective homomorphism $\varphi$ from the group with the above presentation onto $\MapIdOrb{\TheStrand}{\Sigma_\freeprod(\ThePct)}$. It remains to check that this homomorphism is also injective. 
\\
\new{Let $W=\sigma_1^{\varepsilon_1}...\sigma_\TheSubdivision^{\varepsilon_\TheSubdivision}$ with $\sigma_\Subdivision\in\left\{H_\Strand,\TwP{\NPct},\TwC{\nu}\mid1\leq\Strand<\TheStrand,1\leq\NPct\leq\ThePct,1\leq\nu\leq\TheCone\right\}$ and $\varepsilon_\Subdivision\in\{\pm1\}$ be a word in the kernel of the homomorphism $\varphi$. Using that the word represents a pure mapping class, we can rewrite $W$ as a word in letters $\Twist_{\Strand\Str},\TwistP_{\NStrand\NPct}$ and $\TwistC_{\NStrand\nu}$ for $1\leq\Str,\Strand,\NStrand\leq\TheStrand, \Str<\Strand$, $1\leq\NPct\leq\ThePct$ and $1\leq\nu\leq\TheCone$ using the abbreviations from \eqref{eq:def_A_ji} to \eqref{eq:def_A_kc}, 
see \cite[Theorem 3.15]{Flechsig2023} for details on an analogous rewriting. In particular, this rewriting only uses the above relations. If we further use the abbreviations $\Twistn{\Strand}, \TwistnP{\NPct}$ and $\TwistnC{\nu}$ as introduced in \eqref{eq:def_X_Strand} to \eqref{eq:def_X_tau_nu}, then the conjugation relations from Lemma \ref{lem:gens_sat_rels_mcg} allow us to rewrite $W$ as a product} 
\[
W_1(\Twistn{1},...,\Twistn{\TheStrand-1},\TwistnP{1},...,\TwistnP{\ThePct},\TwistnC{1},...,\TwistnC{\TheCone})\cdot W_2(H_1,...,H_{\TheStrand-2},\TwP{1},...,\TwP{\ThePct},\TwC{1},...,\TwC{\TheCone}). 
\]
Since $\varphi(W)=1$ is contained in the pure orbifold mapping class group, we can use the semidirect product structure 
\[
\PMapIdOrb{\TheStrand}{\Sigma_\freeprod(\ThePct)}=\freegrp{\TheStrand-1+\ThePct+\TheCone}\rtimes\PMapIdOrb{\TheStrand-1}{\Sigma_\freeprod(\ThePct)} 
\]
proven in Theorem \ref{cor:pure_orb_mcg_ses}. The letters used in $W_1$ imply that $\varphi(W_1)$ is contained in the free group $\freegrp{\TheStrand-1+\ThePct+\TheCone}$. In particular, $\varphi(W_1)$ is pure. Thus, the word $\varphi(W_2)$ is also pure and contained in $\PMapIdOrb{\TheStrand-1}{\Sigma_\freeprod(\ThePct)}$. This implies that $\varphi(W_1)\cdot\varphi(W_2)$ is the unique decomposition of $\varphi(W)=1$ into the normal subgroup and quotient of the semidirect product, i.e.\ $\varphi(W_1)=1$ in $\freegrp{\TheStrand-1+\ThePct+\TheCone}$ and $\varphi(W_2)=1$ in $\PMapIdOrb{\TheStrand-1}{\Sigma_\freeprod(\ThePct)}$. The first equation directly implies $W_1=1$. Using the induction hypothesis, we further obtain $W_2=1$. Thus, $\varphi$ is an isomorphism and $\MapIdOrb{\TheStrand}{\Sigma_\freeprod(\ThePct)}$ has the above presentation. 
\end{proof}

\bibliographystyle{plain}
\bibliography{paper_mcg}

\end{document}